\documentclass[12pt]{amsart}
\usepackage{amssymb}
\usepackage{amsmath,amscd}
\usepackage{amsfonts,latexsym,verbatim,amscd,mathrsfs,color,array}
\usepackage[all,cmtip]{xy}
\usepackage{mathrsfs}
\usepackage{multirow}
\usepackage{graphicx}
\usepackage{amsfonts, amsthm}
\usepackage{bbm}
\usepackage[hmargin=1in,vmargin=1in]{geometry}
\usepackage{mathdots}

\allowdisplaybreaks

\def\Xint#1{\mathchoice
{\XXint\displaystyle\textstyle{#1}}%
{\XXint\textstyle\scriptstyle{#1}}%
{\XXint\scriptstyle\scriptscriptstyle{#1}}%
{\XXint\scriptscriptstyle\scriptscriptstyle{#1}}%
\!\int}
\def\XXint#1#2#3{{\setbox0=\hbox{$#1{#2#3}{\int}$ }
\vcenter{\hbox{$#2#3$ }}\kern-.6\wd0}}

\def\dashint{\Xint-}

\newcommand{\LB}{\left[}
\newcommand{\RB}{\right]}
\newcommand{\LA}{\langle}
\newcommand{\RA}{\rangle}

\newcommand{\N}{{\mathbb N}}

\newcommand{\R}{{\mathbb R}}

\newcommand{\gothg}{{\mathfrak g}}

\newtheorem{thm}{Theorem}[section]
\newtheorem{lemma}[thm]{Lemma}
\newtheorem*{lemma*}{Lemma}
\newtheorem{prop}[thm]{Proposition}
\newtheorem{cor}[thm]{Corollary}

   \newtheoremstyle{others}
     {3pt}
     {2pt}
     {}
     {}
     {\bf}
     {.}
     {.5em}
     {}
     
\theoremstyle{others}
\newtheorem{rmk}[thm]{Remark}
\newtheorem*{rmk*}{Remark}
\newtheorem{defn}[thm]{Definition}

\numberwithin{equation}{section}

\begin{document}

\title[Long-time existence]{Long-time existence for Yang-Mills flow}
\author{Alex Waldron}
\address{Simons Center for Geometry and Physics, Stony Brook, NY 11794}
\email{awaldron@scgp.stonybrook.edu}

\begin{abstract} We establish that finite-time singularities do not occur in four-dimensional Yang-Mills flow, confirming the conjecture of Schlatter, Struwe, and Tahvildar-Zadeh \cite{sstz}. The proof relies on a weighted energy identity and sharp decay estimates in the neck region.
\end{abstract}

\maketitle

\thispagestyle{empty}

\tableofcontents

\section{Introduction}

Let $E \to M$ be a vector bundle over a four-dimensional Riemannian manifold without boundary. Recall that a time-dependent connection $A = A(t)$ solves the Yang-Mills flow if
\begin{equation}\tag{YM}
\frac{\partial A}{\partial t} = - D^*_A F_A.
\end{equation}
Here $F_A$ is the curvature, and $D_A^*$ denotes the adjoint (with respect to a fixed metric) of the covariant differential on $\gothg_E$-valued forms. The semi-parabolic evolution equation (YM) is the gradient flow of the well-known Yang-Mills functional
\begin{equation*}
YM(A) = \frac{1}{2} \int_M |F_A|^2 \, dV.
\end{equation*}
For the relevant background, the reader may refer to \S 1 of the author's previous paper \cite{instantons} or to the superb textbook by Donaldson and Kronheimer \cite{donkron}.


In the case that $A$ is a compatible connection on a holomorphic bundle over a compact Kahler manifold, owing to early observations of Donaldson \cite{donsurface} and later work of Simpson \cite{simpson}, (YM) is known to have ideal properties---in particular, the flow exists smoothly for all time, and converges to a Hermitian-Yang-Mills connection on a stable bundle. The resulting correspondence between algebraic and differential-geometric structures is known as the Donaldson-Uhlenbeck-Yau Theorem \cite{donstablehd, uhlenbeckyau}. Natural generalizations to the case of unstable holomorphic bundles have more recently been obtained by Daskalopoulos and Wentworth \cite{daskwent04, daskwent07}, Sibley \cite{sibley}, and Sibley and Wentworth \cite{sibleywentworth}.

In a general Riemannian context, the behavior of (YM) is strongly influenced by the dimension of the underlying manifold. Elementary scaling arguments indicate that blowup should be ruled out in dimension 3 or lower, and should occur with ease in dimension 5 or higher. The former was demonstrated formally by Rade \cite{rade}, and the latter by Naito \cite{naito}. Four is the critical dimension, where singularity formation by energy-invariant scaling, or ``bubbling,'' may occur.

While bubbling must be expected for a general minimizing sequence of a conformally-invariant energy functional, the question of whether and how singularities occur under the corresponding parabolic evolution equation is more complex. The singularity removal theorem of Uhlenbeck \cite{uhlenbeckremov} in the elliptic setting might suggest smoothness of (YM) in the parabolic domain. However, it was shown by Chang, Ding, and Ye \cite{cdy} that finite-time singularities occur quite naturally in the analogous situation of two-dimensional harmonic map flow.

Nevertheless, Schlatter, Struwe, and Tahvildar-Zadeh \cite{sstz} found that the same method in the context of (YM), rather than producing finite-time blowup, instead yields a long-time solution with energy concentrating exponentially at infinite time. This led them to conjecture that long-time existence holds for solutions of (YM) in general.


In his thesis \cite{instantons, thesis}, the author demonstrated that (YM) in dimension four indeed differs fundamentally from harmonic map flow in dimension two, in that smallness of half of the curvature is sufficient to guarantee long-time existence. This paper establishes the following general result.

\begin{thm}\label{mainthm} Suppose that $A(t)$ is a smooth solution of (YM) over $M^4 \times \LB 0 , T \right),$ with $T < \infty,$ and
\begin{equation}\label{globalenineq}
\sup_{0 \leq t < T} \int_M |F(t)|^2 \, dV + \int_{0}^T \!\!\!\! \int_M | D^*F |^2 \, dV dt < \infty.
\end{equation}
For any $\epsilon > 0$ and $x \in M,$ there exists $\lambda > 0$ such that
\begin{equation}\label{mainthmlimsupbound}
\limsup_{t \to T} \int_{B_{\lambda}(x)} |F(t)|^2 \, dV < \epsilon.
\end{equation}
Moreover, $\lim_{t \to T} A(t)$ exists in $C^\infty_{loc}.$
\end{thm}

\noindent Over a closed manifold, (\ref{globalenineq}) is immediate from the global energy identity. Applying the short-time existence results of Struwe \cite{struwe}, we may settle the conjecture discussed above. 

\begin{cor} Assume $M^4$ is compact. Then any classical solution of (YM) extends smoothly for all time.

For any Sobolev $H^1$ initial connection, there exists a weak solution defined on $M \times \LB 0, \infty \right)$ which attains the initial data, in the sense of Struwe \cite{struwe}, and is smooth modulo gauge for $t > 0.$\footnote{Since the De Turck trick for (YM) depends a priori on the choice of reference connection, uniqueness is not guaranteed. Struwe \cite{struwe} proves uniqueness of weak solutions under the assumption of irreducibility, and Kozono, Maeda, and Naito \cite{kozono} prove that any sufficiently regular solution is unique modulo gauge. Uniqueness of classical solutions is folklore.}
\end{cor}

The remainder of the paper is devoted to proving our main quantitative result, Theorem \ref{technicaltheorem}, from which Theorem \ref{mainthm} follows straightforwardly. 
Theorem \ref{technicaltheorem} follows chiefly from Propositions \ref{shorttimeprop}-\ref{taubarprop}, which subsume all of the technical results developed in \S \ref{splitevolutionsection}-\ref{decaysection} and the Appendix. 
The main result of these sections, Theorem \ref{decaythm}, provides a parabolic generalization of Uhlenbeck's decay estimates \cite{uhlenbeckremov}, and may be of independent interest. In \S \ref{epsilonregsection}, small-energy regularity results are recalled and adapted for present use.

The next section, \S \ref{stressection}, contains a derivation of the basic energy identity, (\ref{mainenidentity}), which will allow finite-time concentration of the curvature to be ruled out. The idea of the proof, and some of the main difficulties, are explained in \S \ref{explanatoryremark}, p. \pageref{explanatoryremark}.


\subsection*{Acknowledgements} The author thanks Simon Donaldson, Song Sun, Xiuxiong Chen, and Yuanqi Wang for their encouragement and for participating in a preprint seminar. He also thanks Michael Struwe and Toti Daskalopoulos for discussions during the Oberwolfach workshop ``Nonlinear Evolution Problems" in June 2016.

\section{Stress-energy identities}\label{stressection}

Denote the \emph{stress-energy tensor}
$$S_{ij} = g^{k\ell}\LA F_{ik} , F_{j\ell} \RA - \frac{1}{4}g_{ij} g^{k\ell}g^{mn} \LA F_{km}, F_{\ell n} \RA.$$
Working in geodesic coordinates centered at an arbitrary point of $M,$ we apply the divergence operator
\begin{equation}\label{sdiv}
\begin{split}
\nabla_i S_{ij} & = \LA \nabla_i F_{ik} , F_{jk} \RA + \LA F_{ik} , \nabla_i F_{jk} \RA - \frac{1}{4} \nabla_j \LA F_{k \ell}, F_{k \ell} \RA.
\end{split}
\end{equation}
The second Bianchi identity reads
\begin{equation*}
\begin{split}
\LA F_{ik} , \nabla_i F_{jk} \RA & = \LA F_{ik} , \nabla_j F_{ik} \RA + \LA F_{ik} , \nabla_k F_{ji} \RA \\
& = \frac{1}{2}\LA F_{ik}, \nabla_j F_{ik} \RA = \frac{1}{4} \nabla_j \LA F_{i k}, F_{i k} \RA.
\end{split}
\end{equation*}
Substituting into (\ref{sdiv}) gives
\begin{equation}
\begin{split}
\nabla_i S_{ij} = \LA D^*F_k , F_{kj} \RA.
\end{split}
\end{equation}
Taking another divergence, we obtain
\begin{equation}\label{sidentity}
\begin{split}
\nabla_j \! \nabla_i S_{ij} = \nabla_j \LA D^*F_k , F_{kj} \RA & = \LA \nabla_j D^*F_k, F_{kj} \RA + \LA D^*F_k, \nabla_j F_{kj} \RA \\
& = - \LA D D^*F , F \RA + |D^*F|^2.
\end{split}
\end{equation}

Recall that under (YM), the curvature $F = F_{A(t)}$ evolves according to
$$\left(\frac{\partial}{\partial t} + D D^* \right) F= 0.$$
Taking an inner-product with $F$ and adding (\ref{sidentity}), we obtain the (coordinate-independent) identity
\begin{equation}\label{divergenceevolution}
\begin{split}
\frac{1}{2} \frac{\partial}{\partial t} |F|^2 + |D^*F|^2 = \nabla^i \nabla^j S_{ij}.
\end{split}
\end{equation}
Integrating (\ref{divergenceevolution}) against $\chi,$ compactly supported, yields
\begin{equation}\label{chienidentity}
\begin{split}
\frac{1}{2} \frac{d}{dt} \int | F |^2 \chi \, dV + \int | D^*F |^2 \chi \, dV & = \int S_{ij} \nabla^i \nabla^j \chi \, dV.
\end{split}
\end{equation}

Note that $S_{ij}$ is traceless in dimension four. Further let $\chi = \eta \cdot \varphi,$ and write
$$X^i = \nabla^i \eta, \qquad \nabla^i X^j = \mu g^{ij} + h^{ij}.$$
for an arbitrary smooth function $\mu = \mu(x).$ Then
\begin{equation*}
\begin{split}
\nabla^i \nabla^j \chi 
& = \left( \mu g^{ij} + h^{ij}\right) \varphi + X^i \nabla^j \varphi + X^j \nabla^i \varphi + \eta \nabla^i \nabla^j \varphi.
\end{split}
\end{equation*}
Substituting into (\ref{chienidentity}), we have
\begin{equation}\label{etaenidentity}
\begin{split}
\frac{1}{2} \frac{d}{dt} \int | F |^2 \eta\varphi \, dV + \int | D^*F |^2 \eta\varphi \, dV & = \int S_{ij} \left( h^{ij} \varphi + \left( 2 X^i + \eta \nabla^i \right) \nabla^j \varphi \right) \, dV.
\end{split}
\end{equation}

We now let
$$\eta = \frac{r^2}{2}, \qquad X = \nabla \eta.$$
Then by Gauss's Lemma, we have
$$X = X^i \frac{\partial}{\partial X^i} = r \frac{\partial}{\partial r}.$$
Note further that
$$\nabla^i X^j = g^{ij} + h^{ij}$$
with
\begin{equation}\label{Cgsmallness}
|h^{ij}|_g \leq C_g r^2.
\end{equation}
Here $C_g$ is a constant depending on the covariant derivatives of $Rm_g,$ which may be taken arbitrarily small on a unit ball after rescaling.

Assuming $\varphi = \varphi(r),$ we compute
\begin{align*}
\nabla^j \varphi = \frac{X^j}{r} \varphi'(r), \quad \qquad &
\nabla^i \nabla^j \varphi = \left( g^{ij} + h^{ij} \right)\frac{\varphi'(r)}{r} + \frac{X^i X^j}{r^2} \left( \varphi''(r) - \frac{\varphi'(r)}{r} \right) \\
\Delta \varphi & = \varphi''(r) + \frac{3 + h}{r} \varphi'(r)
\end{align*}
where $h = h^{ij}g_{ij}.$ Moreover
\begin{align*}
\left( 2 X^i + \frac{r^2}{2} \nabla^i \right) \nabla^j \varphi & = \left( g^{ij} + h^{ij} \right) \frac{r \varphi'(r)}{2} + \frac{X^i X^j}{2} \left( \varphi''(r) + \frac{3}{r} \varphi'(r)\right) \\
& = \left( g^{ij} + h^{ij} - X^i X^j \frac{h}{r^2} \right) \frac{X(\varphi)}{2} + \frac{X^i X^j }{2}\Delta \varphi.
\end{align*}
Substituting into (\ref{etaenidentity}) yields
\begin{equation}\label{priortoenineq}
\begin{split}
\frac{1}{2} \frac{d}{dt} \int | F |^2 \varphi \, r^2 dV  + \int | D^*F |^2 \varphi \, r^2 dV
& = \int X^i X^j S_{ij} \Delta \varphi \, dV + \int H_\varphi^{ij} S_{ij} \, dV
\end{split}
\end{equation}
where
\begin{align*}
H_\varphi^{ij} & = 2 h^{ij} \varphi + \left( h^{ij} - \frac{X^i X^j }{r^2}h \right) X(\varphi).
\end{align*}

Fix a radial cutoff $\varphi_1(r)$ for $B_1 \subset B_2,$ satisfying
$$| \nabla \varphi_1 | \leq 2, \qquad |\Delta_{\R^4} \varphi_1 | \leq 8$$ 
and put $\varphi_\lambda(r) = \varphi_1(r/\lambda)$ for $\lambda \leq 1.$
Note that $|X(\varphi_\lambda)| \leq 4,$ hence
\begin{equation}\label{Hestimate}
\left| H_{\varphi_\lambda}^{ij} \right|_g \leq C_g r^2
\end{equation}
where $C_g$ is a multiple of the constant of (\ref{Cgsmallness}).

With the exception of \S 3 below, \emph{we will assume henceforth that the metric is flat}, i.e.
$$h^{ij} = H^{ij} = 0.$$
The case of a general metric follows by carrying a small error term through the calculations.

Define
\begin{equation}\label{Sdefinition}
S(\lambda , \tau_1, \tau_2) = 2 \int_{\tau_1}^{\tau_2} \!\!\!\! \int X^i X^j S_{ij} \Delta \varphi_\lambda \, dV dt.
\end{equation}
Denote the annulus
$$U_\lambda^\rho = \bar{B}_\rho \setminus B_\lambda$$
and notice the obvious bound
\begin{equation}\label{obviousbound}
| S(\lambda, \tau_1 , \tau_2) | \leq C \int_{\tau_1}^{\tau_2} \!\!\!\! \int_{U_\lambda^{2\lambda}} |F|^2 \, dV dt.
\end{equation}
Finally, let $\varphi = \varphi_\lambda$ in (\ref{priortoenineq}), and integrate in time. We obtain the localized and weighted energy identity:
\begin{equation}\label{mainenidentity}
\int | F(\tau_2) |^2 \varphi_\lambda r^2 dV + 2 \int_{\tau_1}^{\tau_2} \!\!\!\! \int | D^*F |^2 \varphi_\lambda r^2 dV dt = \int | F(\tau_1) |^2 \varphi_\lambda r^2 dV + S(\lambda , \tau_1, \tau_2).
\end{equation}

\begin{lemma}\label{illustrativelemma} (Baby case of Proposition \ref{taubarprop}, p. \pageref{taubarprop}.)
Let $$0 \leq \tau_1 < \tau_2 < T, \qquad 0 < \epsilon \leq E, \qquad \lambda_1 > 0, \qquad  \lambda_2 = \frac{\lambda_1}{10} \sqrt{\frac{\epsilon}{E}}.$$ Assume
\begin{equation}\label{illustrativeassumptions}
\begin{aligned}
& \int_{B_{\lambda_2}} | F(\tau_2) |^2 \, dV \leq E, \qquad \qquad  & \int_{ U_{\lambda_1/2}^{\lambda_1} } |F(\tau_1)|^2 \, dV \geq \epsilon \\
& \int_{U_{\lambda_2}^{2\lambda_1}} |F(\tau_2)|^2 \, dV \leq \frac{\epsilon}{100}, \qquad \qquad  &  S(\lambda_1, \tau_1, \tau_2) \geq - \frac{\epsilon \lambda_1^2}{100}.
\end{aligned}
\end{equation}
Then
$$ \int_{\tau_1}^{\tau_2} \!\!\!\! \int_{B_{2 \lambda_1}} |D^*F|^2 \, dV dt > \frac{\epsilon}{100}.$$
In particular, if $M$ is compact, then
\begin{equation*}
YM(\tau_2) + \frac{\epsilon}{100} < YM(\tau_1).
\end{equation*}
\end{lemma}
\begin{proof} Direct from (\ref{mainenidentity}), with $\lambda = \lambda_1.$
\end{proof}

\subsection{Idea of the proof}\label{explanatoryremark} We sketch the proof of long-time existence, using the above Lemma as an illustration.

Lemma \ref{illustrativelemma} may be summarized as follows: if a certain amount of energy concentrates from scale $\lambda_1$ to a smaller scale $\lambda_2,$ then we pay a small ``tax'' in overall energy. Since the total energy is nonnegative and decreasing, this could only occur finitely many times under (YM). Therefore, if the assumptions (\ref{illustrativeassumptions}) were all justified, the curvature scale (see Definition \ref{lambdadef} below) would remain bounded away from zero, ruling out blowup.

Unfortunately, the last item of (\ref{illustrativeassumptions}) may fail; in other words, for small $\lambda,$ the cutoff term in (\ref{mainenidentity}) may be too large. 
Fortunately, moving the cutoff from $\lambda$ to $\sqrt{\lambda}$ has just the effect of solving this problem. 


Recall that in her proof of removal of singularities for Yang-Mills fields \cite{uhlenbeckremov}, Uhlenbeck gave us the following optimal estimate. Supposing that a Yang-Mills connection over $U_\lambda^1$ has energy less than $\epsilon,$ its curvature must decay as
\begin{equation}\label{uhlenbeckdecay}
|F(x)| \leq C \sqrt{\epsilon} \left(\frac{\lambda^2}{|x|^4} + 1 \right) \qquad \left( \lambda \leq x \leq 1 \right).
\end{equation}
The original proof directly uses the Yang-Mills equation and Bianchi identity, together with an eigenvalue estimate for the connection in Coulomb gauge (which is trickier). The gauge fixing can also be performed separately on $S^3$---see R\aa de \cite{radecay}.

Let us assume that Uhlenbeck's decay estimate (\ref{uhlenbeckdecay}) holds in the present (parabolic) situation, with $\lambda = \lambda_1.$ 
Note that the curvature at radius $\sqrt{\lambda}$ is bounded by $\sqrt{\epsilon}$ times a constant. Then (\ref{obviousbound}) reads
\begin{align*}\label{key}
\left| S\left( \sqrt{\lambda}, \tau_1, \tau_2 \right) \right| & \leq C \int_{\tau_1}^{\tau_2} \!\!\!\! \int_{U_{\sqrt{\lambda}}^{ 2\sqrt{\lambda} }} |F|^2 \, dV dt \\
& \leq C \left( \tau_2 - \tau_1 \right) \epsilon ( \sqrt{\lambda} )^4 \\
& \leq C \left(\tau_2 - \tau_1 \right) \epsilon \lambda^2.
\end{align*}
Hence, 
the fourth item of (\ref{illustrativeassumptions}) would indeed be justified, provided that
$$\tau_2 - \tau_1 < \frac{1}{100 C}.$$

However, after moving the cutoff, the weight $r^2$ in (\ref{mainenidentity}) can no longer be ignored in the intermediate region $\lambda \leq r \leq \sqrt{\lambda}.$ This causes the second major difficulty of the proof.

In view of (\ref{uhlenbeckdecay}), the weight is not large enough to affect the first and third terms of (\ref{mainenidentity}), and the argument of Lemma \ref{illustrativelemma} still yields
\begin{equation}\label{explanatorydstarfbound}
\int_{\tau_1}^{\tau_2} \!\!\!\! \int_{B_{2\sqrt{\lambda}}} |D^*F|^2 \, r^2 dV dt \geq c \epsilon \lambda^2.
\end{equation}
But, thinking of $|D^*F|$ as a subharmonic function on $\R^4,$ we would only expect $|D^*F(x)| \leq \delta/|x|^2.$ In this case, the $r^2$ factor in (\ref{explanatorydstarfbound}) would clearly prevent us from extracting a substantive lower bound on $\int \!\! \int |D^*F|^2 \, dV dt.$

To overcome this second problem, one simply needs another (optimal) estimate. The identity
$$D^* D^* F = 0$$
suggests that the co-closed 1-form $D^* F$ may decay more strongly than the Green's function. Let us assume
\begin{equation}\label{explanatorysharpdstarfbound}
\left( \int_{\tau_1}^{\tau_2} |D^*F(x) |^2 dt \right)^{1/2} \leq C \delta_\lambda \frac{\lambda}{|x|^3} \qquad \left( \lambda \leq x \leq 2\sqrt{\lambda} \right)
\end{equation}
where
$$\delta_\lambda^2 = \int_{\tau_1}^{\tau_2} \!\!\!\! \int_{B_{\lambda}} |D^*F|^2 \, dV dt.$$
Both (\ref{uhlenbeckdecay}) and (\ref{explanatorysharpdstarfbound}) will turn out to be essentially justified under (YM). 

From (\ref{explanatorysharpdstarfbound}), we then have
\begin{align}\label{explanatorydstarfsecondsharpbound}
\nonumber \int_{\tau_1}^{\tau_2} \!\!\!\! \int_{B_{2\sqrt{\lambda}}} |D^*F|^2 \, r^2 dV dt & = \int_{\tau_1}^{\tau_2} \!\! \left(\int_{B_{\lambda}} + \int_{B_{2\sqrt{\lambda}} \setminus B_{\lambda}} \right) |D^*F|^2 \, r^2 dV dt\\
\nonumber & \leq \delta_\lambda^2 \lambda^2 + C \delta_\lambda^2 \int_\lambda^{2\sqrt{\lambda}} \frac{ \lambda^2}{r^6} \, r^5 dr \\
\nonumber & \leq C \delta_\lambda^2 \lambda^2 \int_\lambda^{2\sqrt{\lambda}} \, \frac{dr}{r} \\
& \leq C \delta_\lambda^2 \lambda^2 | \log \lambda |.
\end{align}
Combining (\ref{explanatorydstarfbound}) and (\ref{explanatorydstarfsecondsharpbound}), we obtain
\begin{align*}
C \delta_\lambda^2 \lambda^2 | \log \lambda | \geq \int_{\tau_1}^{\tau_2} \!\!\!\! \int_{B_{2\sqrt{\lambda}}} |D^*F|^2 \, r^2 dV dt \geq c \epsilon \lambda^2.
\end{align*}
Cancelling $\lambda^2$ and rearranging yields
\begin{equation}\label{explanatoryharmonicbound}
\delta_\lambda^2 \geq \frac{c \epsilon}{|\log \lambda|}.
\end{equation}

Now assume, for the sake of contradiction, that blowup occurs at time $T < \infty,$ and put $\lambda_i = \kappa^{i} \to 0$ for a small constant $\kappa > 0.$ Choose times $T - c \leq \tau_i < T$ at which the curvature is on the scale $\lambda_i,$ and let
$$\delta_i^2 = \int_{\tau_i}^{\tau_{i + 1}} \!\!\!\! \int_{B_{\lambda_i}} |D^*F|^2 \, dV dt.$$
The bound (\ref{explanatoryharmonicbound}) reads
$$\delta_i^2 \geq \frac{c \epsilon}{|\log \lambda_i |} =  \frac{c \epsilon}{i}.$$
Then
$$YM(0) \geq \int_{0}^{T} \!\!\!\! \int_{M} |D^*F|^2 \, dV dt \geq \sum \delta_i^2 \geq c \epsilon \sum \frac{1}{i}.$$
This contradicts the divergence of the harmonic series.




\section{$\epsilon$-regularity}\label{epsilonregsection}

In what follows, $\epsilon_0 > 0$ will be a universal constant, which may decrease repeatedly in the course of the proofs. The constant $R_0 > 0,$ which may also decrease, will depend on the local geometry of $M,$ and also possibly on $k \in \N.$ 

Note that the results of this section do not assume the metric on $M$ to be flat. All later results may be stated for general metrics in a similar fashion.

\begin{prop}\label{epsilonreg} Let
$$x_0 \in M, \quad \quad k \in \N, \quad \quad 0 < R < R_0, \quad \quad R^2 \leq \tau < T.$$
Write $B_R = B_R(x_0),$ and assume
$$\sup_{\tau - R^2 \leq t \leq \tau}\|F(t)\|^2_{L^2(B_{R})} \leq \epsilon < \epsilon_0.$$
Then for
$$\tau - \frac{R^2}{2} \leq t \leq \tau$$
there hold
\begin{align}\tag{$a$}
\|\nabla^{(k)} F(t)\|_{L^\infty (B_{R/2})} & \leq C_k R^{-3-k} \, \|F\|_{L^2\left( B_R \times \LB \tau - R^2, \tau \RB \right)} \leq C_k R^{-2-k} \sqrt{\epsilon} \\
\tag{$b$}
\|\nabla^{(k)} D^*F(t)\|_{L^\infty (B_{R/2})} & \leq C_k R^{-3-k} \, \|D^* F\|_{L^2 \left( B_R \times \LB \tau - R^2, \tau \RB \right)}.
\end{align}
Here $C_k$ is a constant depending only on $k.$
\end{prop}
\begin{proof} See \cite{instantons}, Proposition 3.2.\footnote{The present form of $(a)$ follows by combining the second item in \cite{instantons}, Proposition 3.2, with the local energy inequality.

Note that previous versions of this article, as well as \cite{instantons}, neglected to mention the dependence of $R_0$ on $k.$ Alternatively, if one allows the constants $C_k$ to depend on the geometry of $M,$ then one may let $R_0 = \min \LB inj(M), 1 \RB.$} 
\end{proof}

\begin{lemma}\label{antibubble} 
For $R < R_0,$ assume
\begin{equation*}
\begin{split}
\sup_{0 \leq t \leq \tau} \| F(t) \|^2_{L^2 \left( B_R \right)} \leq E, \quad \quad
\|D^*F\|_{L^2 \left( B_R \times \LB 0, \tau \RB \right)} \leq \delta
\end{split}
\end{equation*}
and put
$$\gamma = 2 \delta \left(\delta + 4 \frac{\sqrt{\tau E }}{R} \right).$$

\noindent $(a)$ For $0 \leq t \leq \tau,$ there hold
\begin{align*}
\|F(\tau)\|^2_{L^2(B_{R/2})} \leq \|F(t)\|^2_{L^2(B_R)} + \gamma, \quad \quad
\|F(t)\|^2_{L^2(B_{R/2})} \leq \|F(\tau)\|^2_{L^2(B_R)} + \gamma.
\end{align*}

\vspace{2mm}

\noindent $(b)$ If
\begin{equation}\label{antibubblebest}
\|F(\tau)\|^2_{L^2(B_{R})} + \gamma \leq \epsilon < \epsilon_0
\end{equation}
then for $R^2 \leq t \leq \tau$ and $k \in \N,$ there hold
\begin{equation*}
\begin{split}
\| \nabla^{(k)} F(t) \|_{L^\infty\left( B_{R / 2} \right)} \leq C_k R^{-2-k} \, \sqrt{ \epsilon}, \quad \quad
\|\nabla^{(k)} D^*F(t)\|_{L^\infty (B_{R/2})} \leq C_k R^{-3-k} \, \delta.
\end{split}
\end{equation*}
\end{lemma}
\begin{proof} Let $\varphi= \varphi_{R/2}$ be a cutoff for $B_{R/2} \subset B_{R}$ with $|\nabla \varphi | \leq 4 / R.$ Integrating by parts once in (\ref{divergenceevolution}), we have
\begin{equation}
\begin{split}
\frac{1}{2} \frac{d}{dt} \int | F |^2 \varphi \, dV + \int | D^*F |^2 \varphi \, dV + \int \LA D^*F^i, F_{ij} \nabla^j \varphi \RA \, dV = 0.
\end{split}
\end{equation}
Integrating from $t$ to $\tau$ and applying H\"older's inequality on the last term, we obtain
\begin{align*}
\left|  \int \left(|F(\tau)|^2 - |F(t)|^2 \right) \! \varphi \, dV \right| & \leq \gamma.
\end{align*}
This implies both items of ($a$).

To prove ($b$), note that (\ref{antibubblebest}) and ($a$) imply
$$\|F(t)\|^2_{L^2\left(B_{3R/4} \right)} \leq 2 \epsilon \quad \quad \left( \forall \,\,\,\,  0 < t < \tau \right).$$
The result then follows from Proposition \ref{epsilonreg}, which may be applied for any $R^2 \leq \tau' \leq \tau.$ 
\end{proof}

\begin{lemma}\label{annularantibubble} Let $R, \tau$ as above, and replace $\gamma$ by $10 \gamma / (1 - \alpha),$ for $0 < \alpha < 1.$ Write
$$U_1 = U_{\alpha R}^R(x_0), \quad \quad U_2 = U_{\alpha^2 R}^{\alpha^{-1} R}(x_0)$$
and assume
\begin{equation*}
\begin{split}
\sup_{0 \leq t \leq \tau} \| F(t) \|^2_{L^2 \left( U_2 \right)} \leq E, \quad \quad
\|D^*F\|_{L^2 \left( U_2 \times \LB 0, \tau \RB \right)} \leq \delta.
\end{split}
\end{equation*}

\noindent $(a)$ For $0 \leq t \leq \tau,$ there hold
\begin{align*}
\|F(\tau)\|^2_{L^2(U_1)} \leq \|F(t)\|^2_{L^2(U_2)} + \gamma, \quad \quad
\|F(t)\|^2_{L^2(U_1)} \leq \|F(\tau)\|^2_{L^2(U_2)} + \gamma.
\end{align*}

\vspace{2mm}

\noindent $(b)$ If
\begin{equation*}
\|F(\tau)\|^2_{L^2(U_2)} + \gamma \leq \epsilon < \epsilon_0
\end{equation*}
then for $R^2 \leq t \leq \tau$ and $k \in \N,$ there hold
\begin{equation*}
\begin{split}
\| \nabla^{(k)} F(t) \|_{L^\infty\left( U_1 \right)} \leq C_{k, \alpha} R^{-2-k} \, \sqrt{ \epsilon}, \quad \quad
\|\nabla^{(k)} D^*F(t)\|_{L^\infty (U_1)} \leq C_{k, \alpha} R^{-3-k} \, \delta.
\end{split}
\end{equation*}
\end{lemma}
\begin{proof} As above, letting $\varphi$ be a cutoff for $U_1 \Subset U_2.$ 
\end{proof}

\begin{prop}\label{spendingprop} For $R < R_0,$ assume 
\begin{equation}\label{spendingpropouterassn}
\| F(0) \|^2_{L^2(B_R)} \leq \epsilon_1 < \epsilon_0, \quad \quad
\sup_{0 \leq t < T} \| F(t) \|^2_{L^2\left( U_{R/2}^{R} \right)} \leq \epsilon_2 < \epsilon_0.
\end{equation}
Then
\begin{equation}\label{spendingpropest}
\| F(t) \|^2_{L^2(B_{R / 2})}  \leq \epsilon_1 e^{- c t / R^2} + C \epsilon_2 \left(1 - e^{- c t / R^2} \right) \qquad (0 \leq t < T).
\end{equation}
\end{prop}
\begin{proof} Let $\eta >\epsilon_0.$ Define $0 < T_0 \leq T$ to
be the maximal time such that
\begin{equation}\label{spendingpropassn}
\|F(t)\|^2_{L^2(B_R)} < \eta \quad \quad \left( \forall \,\,\,\, 0 \leq t < T_0\right).
\end{equation}

Let $v = |F|,$ and recall the differential inequality
$$\left( \partial_t + \Delta \right) v \leq A v^2 + Bv$$
where $A$ is a universal constant and $B$ is a multiple of $\| Rm \|_{L^\infty}.$

Let $\varphi = \varphi_{R/2}$ as above. Multiplying by $\varphi^2 v$ and integrating by parts, we obtain
\begin{equation*}
\begin{split}
\frac{1}{2} \frac{d}{dt} \left(\int \varphi^2 v^2\right) + \int \nabla(\varphi^2 v) \cdot \nabla v & \leq A\int \varphi^2 v^{3} + B \int \varphi^2 v^2 \\
\frac{1}{2} \frac{d}{dt} \left( \int \varphi^2 v^2 \right) + \int |\nabla(\varphi v)|^2  & \leq \int |\nabla \varphi |^2 v^2 + A \int \varphi^2 v^3 + B \int \varphi^2 v^2.
\end{split}
\end{equation*} 
For $0 \leq t < T_0,$ applying the Sobolev and H\"older inequalities on $B_R$ yields
\begin{equation*}
\frac{1}{2} \frac{d}{dt} \int \left(\varphi v\right)^2 + \left( \frac{1}{C_S} - A\sqrt{\eta} \right) \left(\int (\varphi v)^4 \right)^{1/2}  \leq \|\nabla \varphi\|^2_{L^{\infty}} \int_{U_{R/2}^{R}} v^2  + B \int \left( \varphi v \right)^2.
\end{equation*}
Applying H\"older's inequality on the left-hand side, and the assumption (\ref{spendingpropouterassn}) on the right-hand side, we obtain
\begin{equation}\label{spendingpropdiffineq}
\frac{d}{dt} \int \left(\varphi v\right)^2 + \frac{\alpha}{R^2} \int (\varphi v)^2  \leq 8R^{-2} \epsilon_2
\end{equation}
where
$$\alpha = c \left( C_S^{-1} - A\sqrt{\eta} \right) - B R^2.$$
We assume that $\eta$ and $R_0$ are sufficiently small that $\alpha > c / 2  C_S .$  

Rewrite (\ref{spendingpropdiffineq}) as
$$\frac{d}{dt} \left( e^{\alpha t / R^2} \int \left(\varphi v\right)^2\right) \leq 8R^{-2} \epsilon_2 e^{\alpha t / R^2}$$
and integrate in time, to obtain
\begin{equation*}
\begin{split}
\int_{B_{R/2}} v(t)^2 & \leq e^{- \alpha t / R^2} \int_{B_R} v(0)^2 + C \epsilon_2 \left(1 - e^{- \alpha t / R^2} \right).
\end{split}
\end{equation*}
This establishes (\ref{spendingpropest}) for $0 \leq t < T_0.$

Assume, for the sake of contradiction, that $T_0 < T.$ For $t < T_0,$ (\ref{spendingpropest}) reads
\begin{equation*}
\begin{split}
\int_{B_{R/2}} |F(t)|^2 \leq \epsilon_1 + C \epsilon_2.
\end{split}
\end{equation*}
Provided that
$$2\left(1 + C \right) \epsilon_0 < \eta$$
we have
\begin{equation}\label{spendingproplastest}
\int_{B_R} |F(t)|^2 = \int_{B_{R/2}} |F(t)|^2 + \int_{U_{R/2}^{R}} |F(t)|^2 \leq \epsilon_1 + C \epsilon_2 \leq \frac{1}{2} \eta.
\end{equation}
Since the flow is smooth for $t < T,$ the bound (\ref{spendingproplastest}) persists at $t = T_0.$ This contradicts the maximality of $T_0$ subject to the open condition (\ref{spendingpropassn}); hence $T_0 = T,$ and (\ref{spendingpropest}) is proved.
\end{proof}

\section{Split evolution of curvature}\label{splitevolutionsection}

In this section, we study the system of evolution equations comprised by $F$ and $D^*F$ under (YM).
Following a similar strategy to that of R\aa de \cite{radecay}, we will work in cylindrical coordinates. However, because the flow (YM) is not preserved by general conformal changes (unless the solution is static), we must continue to use the Euclidean metric. Hence, in the parabolic context, cylindrical coordinates are only a computational device.

Under smallness of the curvature in an annular region, we shall find that the full system (\ref{mainsystem}) is governed by the system of three differential inequalities (\ref{f1evol}-\ref{g2evol}).\footnote{It is noteworthy that the linear part of these inequalities is precisely the same as in the rotationally symmetric reduction of (YM) studied by Schlatter et.al. \cite{sstz}. See also the author's thesis \cite{thesis} for a detailed analysis.
} These will yield the optimal decay exponents for both $F$ and $D^*F$ in Theorem \ref{decaythm}, crucial to the proof.


\subsection{Hodge Laplacian under a conformal change}

Let $\rho(x)$ be a smooth function on $M= M^n,$ with $n$ even. Set
$$\bar{g} = e^{2\rho} g, \quad \quad X = \bar{g}^{-1} d\rho.$$

\begin{lemma}\label{hodge} For $\omega \in \Omega^k,$ the Hodge Laplacians for $g$ and $\bar{g}$ are related by
\begin{equation*}
e^{-2\rho} \left( \Delta_g \omega - 2 \, d \rho \wedge d^* \omega \right) =\Delta_{\bar{g}} \omega + \left( n - 2k \right) d\iota_X \omega  + \left( n - 2(k+1) \right) \iota_X d \omega
\end{equation*}
where the adjoint $d^*$ is with respect to $g.$

Let $n=4.$ If $\alpha \in \Omega^2$ satisfies
$d \alpha = 0,$ then
\begin{equation}\tag{$a$}
e^{-2\rho} \Delta_g \alpha = \Delta_{\bar{g}} \alpha + 2 e^{-2\rho} d\rho \wedge d^*\alpha.
\end{equation}
If $\beta \in \Omega^1$ satisfies
$d^*\beta = 0,$ then
\begin{equation}\tag{$b$}
e^{-2\rho} \Delta_g \beta = \Delta_{\bar{g}} \beta + 2 \, d \left( \beta (X) \right).
\end{equation}
\end{lemma}
\begin{proof}
We write $\ast = \ast_g$ and
$$\bar{\ast} = \ast_{\bar{g}} = e^{(n-2k) \rho} \ast$$
on $k$-forms. Then
\begin{align}
\nonumber \Delta_{\bar{g}} \omega & = - \left( \bar{\ast} d \bar{\ast} d + d \bar{\ast} d \bar{\ast} \right) \omega \\
\nonumber & = -e^{(n-2(n-k))\rho} \ast d \left( e^{\left(n-2(k+1)\right) \rho} \ast d\omega \right) - d \left( e^{(n - 2(n-k+1)) \rho} \ast d \left( e^{(n-2k)\rho} \ast \omega \right) \right) \\
& = - e^{(-n+2k) \rho} \ast d \left( e^{(n-2(k+1)) \rho} \ast d \omega \right) - d \left( e^{(-n+2k-2)\rho} \ast d \left( e^{(n-2k)\rho} \ast \omega \right) \right) \\
\nonumber & = - e^{-2\rho} \ast d \ast d \omega - (n-2(k+1)) \ast \left(  e^{-2 \rho} d \rho \wedge \ast d\omega \right) \\
\nonumber & \quad \quad - d \left( (n-2k) \ast \left( e^{-2 \rho} d \rho \wedge \ast \omega \right) + e^{-2 \rho} \ast d \ast \omega \right).
\end{align}
Note that $X = e^{-2\rho} g^{-1} d\rho,$ and $\ast( d\rho \wedge \ast \alpha ) = \iota_{g^{-1} d\rho} \alpha.$ We therefore obtain
$$\Delta_{\bar{g}} \omega = e^{-2\rho} \Delta_g \omega - \left( n-2(k+1) \right) \iota_X d\omega - (n-2k) \, d \left( \iota_X \omega \right) + 2 e^{-2\rho} d\rho \wedge \ast d \ast \omega$$
as desired.
\end{proof}

\begin{rmk} The previous Lemma is equally valid if we replace $d$ by the covariant differential $D = D_A$ on forms valued in any bundle.
\end{rmk}

\begin{lemma}\label{eigenvalues} Let $A_\theta$ be a connection on $E \to S^3,$ with $\|F_\theta\|^2_{L^2(S^3)} < \epsilon_0.$

\vspace{1mm}

\noindent $(a)$ For $\Omega \in \Omega^2(S^3, \gothg_E),$ assume that either $d_\theta \Omega = 0$ or $\Omega = d_\theta \omega$ for $\omega \in \Omega^1.$ Then
\begin{equation*}
\left( \Delta_\theta \Omega, \Omega \right) \geq \left(4 - C\|F_\theta\|_{L^2(S^3)} \right) \|\Omega\|^2_{L^2(S^3)}.
\end{equation*}

\vspace{1mm}

\noindent $(b)$ For $\omega \in \Omega^1(S^3, \gothg_E),$ there holds
\begin{equation*}
\left( \Delta_\theta \omega, \omega \right) \geq \left(3 - C\|F_\theta\|_{L^2(S^3)} \right) \| \omega \|^2_{L^2(S^3)}.
\end{equation*}

\vspace{1mm}

\noindent $(c)$ For $\phi \in \Omega^0(S^3, \gothg_E),$ assume that $ \phi = d_\theta^* \omega$ for $\omega \in \Omega^1(S^3, \gothg_E).$ Then
\begin{equation*}
\left( \Delta_\theta \phi, \phi\right) \geq \left(3 - C\|F_\theta\|_{L^2(S^3)} \right) \|\phi \|^2_{L^2(S^3)}.
\end{equation*}
\end{lemma}
\begin{proof} These are proven as in \cite{radecay}, Lemma 2.1. The proofs depend on the knowledge that the first eigenvalue of the Hodge Laplacian on real-valued closed 2-forms on $S^3$ is 4, while the first eigenvalue on 1-forms is 3.
\end{proof}

\subsection{Evolution equations in cylindrical coordinates} As stated in \S 2, in this section and for the remainder of the paper, we will assume for simplicity that the metric $g$ on $M$ is flat.

Denote the cylindrical metric
$$\bar{g} = r^{-2} g$$
on $B_1 \setminus \{0\} \simeq \R_{-} \times S^3,$ and introduce the cylindrical coordinate $s = \log r.$ In the notation of Lemma \ref{hodge}, we have
$$\rho = - s, \quad \quad X = - \frac{\partial}{\partial s}.$$
Write
\begin{equation*}
\begin{split}
\alpha & = F = ds \wedge \Phi + \Omega \\
\beta & = D^*F = \phi \, ds + \omega 
\end{split}
\end{equation*}
where
\begin{equation}\label{betadef}
\begin{split}
\phi(s,t) \in \Omega^0(S^3, \gothg), \quad \quad \omega(s,t), \Phi(s,t) & \in \Omega^1(S^3, \gothg ), \quad \quad \Omega(s,t) \in \Omega^2(S^3, \gothg).
\end{split}
\end{equation}
The Bianchi identity $DF= D\alpha = 0$ is equivalent to
\begin{equation}\label{secondbianchi1}
d_\theta \Omega = 0
\end{equation}
\begin{equation}\label{secondbianchi2}
\partial_s \Omega  = d_\theta \Phi
\end{equation}
while the identity $D^*D^*F = D^* \beta = 0$ reads
\begin{equation}\label{firstbianchi}
\partial_s (r^2 \, \phi) = r^2 \, d_{\theta}^{\bar{*}} \omega.
\end{equation}
All derivatives are covariant with respect to $A(t),$ and $d_\theta^{\bar{*}}$ denotes the adjoint with respect to the round metric on $S^3.$

The Hodge Laplacian with respect to $\bar{g}$ is computed as follows.
\begin{equation*}
\begin{split}
D \left( ds \wedge \Phi \right) & = - ds \wedge d_\theta \Phi \\
D^{\bar{*}} D \left( ds \wedge \Phi \right) & = \partial_s d_\theta \Phi + ds \wedge d_\theta^{\bar{*}} d_\theta \Phi
\end{split}
\end{equation*}
and
\begin{equation*}
\begin{split}
D^{\bar{*}} \left( ds \wedge \Phi \right) & = - \partial_s \Phi - ds \wedge d_\theta^{\bar{*}} \Phi \\
D D^{\bar{*}} \left( ds \wedge \Phi \right) & = - ds \wedge \partial_s^2 \Phi - d_\theta \partial_s \Phi + ds \wedge d_\theta d_\theta^{\bar{*}} \Phi
\end{split}
\end{equation*}
hence
\begin{equation*}\label{cylinderhodge}
\begin{split}
\Delta_{\bar{g}} (ds \wedge \Phi) & = ds \wedge \left( - \partial_s^2 + \Delta_\theta \right) \Phi + \LB \partial_s , d_\theta \RB \Phi \\
& =  ds \wedge \left( - \partial_s^2 + \Delta_\theta \right) \Phi + \Phi \# \Phi.
\end{split}
\end{equation*}
Here $\#$ denotes a contraction involving the cylindrical metric, and similar expressions of the form $\Delta_{\bar{g}} = - \partial_s^2 + \Delta_\theta + \Phi \#$ hold for the other components.


According to Lemma \ref{hodge}, the pair of evolution equations
\begin{equation*}
\begin{split}
\left( \partial_t + \Delta_A \right) F_{ij} & = 0 \\
\left( \partial_t + \Delta_A\right) D^*F_i & = r^{-2} \, \bar{g}^{jk}\LB F_{ij}, D^*F_k \RB
\end{split}
\end{equation*}
is equivalent to a system
\begin{equation}\label{mainsystem}
\left\{
\begin{split} (i) \quad & r^2 \partial_t \Omega  = \left( \partial_s^2 - \Delta_{\theta} \right) \Omega + \alpha \# \alpha \\
(ii) \quad & r^2 \partial_t \Phi  = \left( \partial_s^2 - \Delta_{\theta} \right) \Phi + 2 r^2 \omega + \alpha \# \alpha \\
(iii) \quad & r^2 \partial_t \omega  = \left( \partial_s^2 - \Delta_{\theta} \right) \omega + 2 \, d_\theta \phi + \alpha \# \beta \\
(iv) \quad & r^2 \partial_t \phi  = \left( \partial_s^2 + 2 \partial_s - \Delta_{\theta} \right) \phi + \alpha \# \beta. \\
\end{split}
\right.
\end{equation}
Note the commutation formulae
\begin{align*}
d_\theta^2 & = \Omega \#, & \LB \partial_t, d_\theta \RB & = \omega \# \\
\LB d_\theta, \Delta_\theta \RB & = \nabla_\theta \Omega \# + \Omega \# \nabla_\theta, &
\LB \partial_s^2 , d_\theta \RB & = \partial_s \Phi \# + \Phi \# \partial_s.
\end{align*}
Applying $d_\theta$ to $(iii),$ we obtain
\begin{equation}\tag{$iii'$}
\begin{split}
\quad r^2 \partial_t d_\theta \omega & = \left(\partial_s^2 - \Delta_\theta \right)d_\theta \omega + \Omega \# \phi + r^2 \omega \# \omega + \left( \partial_s \Phi + \nabla_\theta \alpha \right) \# \beta + \alpha \# \left( \partial_s \omega + \nabla_\theta \beta \right) \\
& = \left(\partial_s^2 - \Delta_\theta \right)d_\theta \omega + E_{(iii')}.
\end{split}
\end{equation}
 Also let
$$\xi(s,\theta,t) = r \cdot \phi(s,\theta,t) $$
to obtain
\begin{equation}\tag{$iv'$}
\begin{split}
r^2 \partial_t \xi & = \left( \partial_s^2 - \Delta_\theta - 1 \right) \xi + r \alpha \# \beta  \\
& = \left( \partial_s^2 - \Delta_\theta - 1 \right) \xi + E_{(iv')}.
\end{split}
\end{equation}

Writing $|\cdot | = |\cdot|_{\bar{g}}, \bar{\nabla} = \nabla_{\bar{g}},$ and $\alpha = \alpha(s,t),$ etc., we now let
\begin{samepage}
\begin{align}\label{componentsdefinition}
\nonumber e(s,t) & = \sup_{S^3} \left( \left| \alpha \right| + \left| \bar{\nabla}\alpha \right| + \left| \bar{\nabla}^{(2)}\alpha \right| \right), & h(s,t) & = \sup_{S^3} \left( \left| \beta \right| +  \left|  \bar{\nabla} \beta \right| \right) \\
\nonumber f^2(s,t) & = \int_{S^3} \left| \alpha \right|^2 d \Theta, & g^2(s,t) & = \int_{S^3} \left| \beta \right|^2 d \Theta \\
f_1^2(s,t) & = \int_{S^3} \left| \Omega \right|^2 d \Theta, & g_1^2(s,t) & = \int_{S^3} \left| d_\theta\omega \right|^2 d \Theta \\
\nonumber f_2^2(s,t) & = \int_{S^3} \left| \partial_s \Omega \right|^2 d \Theta, & g_2^2(s,t)  & = \int_{S^3} \left| \xi \right|^2 d \Theta \\
\nonumber &  & g_3^2(s,t) & = \int_{S^3} \left| \partial_s \xi \right|^2 d \Theta.
\end{align} 
\end{samepage}
By removing $L^\infty$ norms and applying H\"older's inequality, we may estimate
\begin{align}\label{g1error}
\nonumber \int \LA d_\theta \omega , E_{(iii')} \RA & \leq C g_1 \left(eg + r^2 g h + hf \right) \\
& \leq C e g_1 h
\end{align}
where we have used $r^2 g \leq C e,$ from the definition. Also
\begin{equation}\label{g2error}
\int \LA \xi , E_{(iv')} \RA \leq C r e g g_2.
\end{equation}

Following R\aa de \cite{radecay}, we compute
\begin{equation*}
\begin{split}
2f_1 \partial_s f_1 = \partial_s f_1^2 & = 2  \int \LA \Omega , \partial_s \Omega \RA \leq 2 f_1 f_2 \\
\end{split}
\end{equation*}
and
\begin{equation}\label{spherical1}
\left( \partial_s f_1 \right)^2  \leq f_2^2.
\end{equation}
Also
\begin{equation}\label{spherical2}
2 \left( \left( \partial_s f_1\right)^2 + f_1 \partial_s^2 f_1 \right) = \partial_s^2 \left( f_1^2 \right) = 2 \left( f_2^2 + \int \LA \Omega, \partial_s^2 \Omega \RA \right).
\end{equation}
Subtracting (\ref{spherical2}) from (\ref{spherical1}) yields
\begin{equation}\label{spherical3}
f_1 \partial_s^2 f_1 \geq \int \LA \Omega, \partial_s^2 \Omega \RA.
\end{equation}
Note lastly that
\begin{equation}\label{spherical4}
2 f_1 \partial_t f_1 = \partial_t f_1^2 = 2 \int \LA \Omega, \partial_t \Omega \RA.
\end{equation}
Subtracting (\ref{spherical4}) from (\ref{spherical3}) yields
\begin{equation}\label{spherical5}
\begin{split}
f_1 \left( r^2 \partial_t - \partial_s^2 \right) f_1 \leq \int \LA \Omega , \left( r^2 \partial_t - \partial_s^2 \right) \Omega \RA
& = \int \LA \Omega , - \Delta_\theta \Omega + \alpha \# \alpha \RA \\
& \leq - \left( 4 - C f(s) \right) f_1^2 + e(s) ff_1 \\
& \leq \left(-4 + C e(s) \right) ff_1
\end{split}
\end{equation}
by (\ref{secondbianchi1}) and Lemma \ref{eigenvalues}$a$, provided
\begin{equation}\label{esmallness}
e^2(s,t) < \epsilon_0
\end{equation}
as we assume henceforth.

Define the heat operator
\begin{equation}\label{square}
\square = \partial_t - \left( \partial_r^2 + \frac{1}{r}\partial_r - \frac{4}{r^2} \right).
\end{equation}
Dividing by $r^2 f_1$ in (\ref{spherical5})
yields\footnote{This is valid in the sense of distributions, as can be shown by replacing $f_1^2$ by $f_1^2 + \epsilon$ and letting $\epsilon \to 0.$} 
\begin{equation*}
\square f_1(r,t) \leq \frac{C e}{r^2} f.
\end{equation*}
The same calculation, using (\ref{g1error}-\ref{g2error}) as well as Lemma \ref{eigenvalues}$c$ together with the fact
\begin{equation}\label{chiisPhi}
\xi = r \phi = d_\theta^{\bar{*}} \left( - r^{-1} \Phi \right)
\end{equation}
yields similar evolution equations for $g_1$ and $g_2.$ Overall, having only assumed (\ref{esmallness}), we obtain the system of differential inequalities 
\begin{align}
\label{f1evol} \square f_1 & \leq \frac{C e}{r^2} f \\
\label{g1evol} \square g_1 & \leq \frac{C e}{r^2} h \\
\label{g2evol} \square g_2 & \leq \frac{C e}{r} g.
\end{align}

\subsection{Size of curvature components}\label{sizeofcomponentssection} Finally, we record the relationship between the component functions and the size of the curvature in the Euclidean metric $g.$ Redefining $e$ and $h$ up to constants, we have 
\begin{align}\label{sizeofcomponents}
\nonumber e(r,t) & = r^2 \sup_{S^3} \left( \left| F \right|_g + r \left| \nabla F \right|_g + r^2 \left| \nabla^{(2)} F \right|_g \right) \\
h(r,t) & = r\sup_{S^3} \left( \left| D^*F \right|_g + r\left| \nabla D^*F \right|_g \right) \\
\nonumber f^2(r,t) & = r^4 \int_{S^3} \left| F \right|_g^2 \, d\Theta, \quad \quad \quad g^2(r,t) = r^2 \int_{S^3} \left| D^*F \right|_g^2 \, d\Theta.
\end{align}
Note from the definition that
$$r^2 h(r,t) \leq e(r,t).$$
If we assume that $(r,t)$ is such that
\begin{equation*}
\sup_{t-r^2 \leq \tau \leq t} \| F(\tau)\|^2_{L^2 \left( U^{2r}_{r/2} \right)} < \epsilon_0
\end{equation*}
then Lemma \ref{annularantibubble} implies
\begin{equation*}
f_1(r,t) + f_2(r,t) \leq C e(r,t) \leq C r^{-1} \| F\|_{L^2 \left(U^{2r}_{r/2} \times \LB t-r^2 , t \RB \right)}
\end{equation*}
\begin{equation*}
g_1 + r^{-1}\left( g_2 + g_3\right) \leq C h \leq C r^{-2} \| D^*F \|_{L^2 \left(U^{2r}_{r/2} \times \LB t-r^2 , t \RB \right)}.
\end{equation*}

On the other hand, by (\ref{secondbianchi2}), (\ref{chiisPhi}), and Lemma \ref{eigenvalues}$b,$ we have
\begin{equation*}
f \leq f_1 + C\left( f_2 + r g_2 \right).
\end{equation*}
Moreover, the identity (\ref{firstbianchi}) may be rewritten
$$d_\theta^{\bar{*}}\omega = r^{-2}\partial_s \left( r^{2} \phi \right) = r^{-1}\left( \partial_s + 1 \right)\xi$$
which implies
$$\int_{S^3} |d_\theta^{\bar{*}} \omega |^2 \leq C r^{-2} \left( g_2^2 + g_3^2 \right).$$
Again by Lemma \ref{eigenvalues}$b,$ this yields
\begin{equation}\label{gsize}
\begin{split}
g(r,t) & \leq C \left( g_1 + r^{-1}\left( g_2 + g_3 \right) \right).
\end{split}
\end{equation}

We adopt the convention
$$\int_\rho^R \, dV = \int_{U_\rho^R} \, dV.$$

\begin{prop}\label{sizeofcurvprop} Assume $e^2 < \epsilon_0,$ and let $0 < \alpha < 1.$ Then
\begin{equation}\tag{$a$}
e(r,t) \leq C_\alpha \sup_{U_{\alpha r}^{\alpha^{-1}r} \times \LB t - (1 - \alpha)r^2, t \RB }\left(  f_1 + r g_2 \right)
\end{equation}
\begin{equation}\tag{$b$}
\begin{split}
r^6 \sup_{ t_1 \leq t \leq t_2 } h(r,t)^2 + r^4 \int_{t_1}^{t_2} h(r,t)^2 \, dt \leq C_\alpha \int_{t_1 - r^2}^{t_2} \! \int_{\alpha r}^{\alpha^{-1}r} \left( g_1^2 +  r^{-2} g_2^2 \right) \, dV dt.
\end{split}
\end{equation}
\end{prop}
\begin{proof} We prove only ($b$), since ($a$) is similar. Let $\alpha = 1/e$ for simplicity.

Choose a smooth cutoff $\psi_0(s,t) \geq 0$ with
$$\psi_0(s,t) \equiv 1 \qquad \left( -1/2 \leq s \leq 1/2, \quad t \geq 0 \right) $$
$$\text{supp} \, \psi_0 \subset \left( -1, 1 \right) \times \left(-1 , \infty \right)$$
$$\left| \partial_s \psi_0 \right| + \left| \partial_t \psi_0 \right| \leq 4$$
and put
$$\psi(s, t) = \psi_{\sigma , \tau} (s,t) = \psi_0 \left(s - \sigma , e^{-2\sigma} \left( t - \tau \right) \right).$$

Let $\sigma = \log r,$ take an inner product with $\psi^2 \xi$ in ($iv'$), and integrate in space and time, to obtain
\begin{equation*}
\begin{split}
\frac{1}{2} \int & \psi(\cdot, \tau)^2 \left| \xi(\cdot, \tau ) \right|^2 r^2 \, d\bar{V}  + \int_{-\infty}^{\tau} \! \int \psi^2 \left( |\xi|^2 + \left| \partial_s \xi \right|^2 + \left| d_\theta \xi \right|^2 \right) \, d\bar{V} dt \\
& \leq \int_{-\infty}^{\tau} \! \int \left( r^2 \psi \partial_t \psi | \xi |^2 + 2 \psi \left| \partial_s \psi \right| |\xi | \left| \partial_s \xi \right| + \psi^2 \LA \xi , E_{(iv')} \RA \right) \, d\bar{V} dt
\end{split}
\end{equation*}
Applying Young's inequality and absorbing $\psi | \partial_s \xi |,$ while noting that $r^2 \partial_t \psi \leq C,$ we obtain
\begin{equation*}
\int_{-\infty}^{\tau} \! \int \psi^2 \left( | \xi |^2 + \left| \partial_s \xi \right|^2 + \left| d_\theta \xi \right|^2 \right) \, d\bar{V} dt \leq C \int_{\tau - r^2}^{\tau} \! \int_{\sigma - 1}^{\sigma + 1} \left( |\xi|^2 + \psi^2 \LA \xi , E_{(iv')} \RA \right) d \bar{V} dt
\end{equation*}
and from (\ref{g2error})
\begin{equation}\label{g3est}
\int_{\tau - r^2}^{\tau} \! \int \psi^2 g_3^2 \, ds dt \leq C \int_{\tau-r^2}^{\tau} \! \int_{\sigma - 1}^{\sigma + 1} g_2 \left( g_2 + \psi^2 reg \right) \, ds dt .
\end{equation}

Integrating (\ref{gsize}) and inserting (\ref{g3est}) yields
\begin{equation*}
\begin{split}
\int_{\tau - r^2}^\tau \! \int \psi^2 g^2 \, d\bar{V} dt & \leq \int_{\tau - r^2}^\tau \! \int \psi^2 \left( g_1^2 + r^{-2} \left( g_2^2 + g_3^2 \right) \right) \, d\bar{V} dt \\
& \leq C \int_{\tau - r^2}^\tau \! \int_{\sigma - 1}^{\sigma + 1} \left( g_1^2 + r^{-2} g_2 \left( g_2 + \psi^2 r e g \right) \right) \, d\bar{V} dt \\
\end{split}
\end{equation*}
Now apply Young's inequality
$$r^{-1} e g g_2 \leq r^{-2} g_2^2 + e^2 g^2 $$
and rearrange, to obtain
\begin{equation*}
\begin{split}
\int_{\tau - r^2}^\tau \! \int \psi^2 \left( 1 - C e^2 \right) g^2 \, d\bar{V} dt
& \leq C \int_{\tau - r^2}^\tau \! \int_{\sigma - 1}^{\sigma + 1} \left( g_1^2 + r^{-2} g_2^2 \right) \, d\bar{V} dt.
\end{split}
\end{equation*}
For $Ce^2 < C\epsilon_0 < 1/2,$ in view of Proposition \ref{epsilonreg}$b$, we have
\begin{equation}\label{priortofubini}
r^6 h^2(r,\tau) \leq C \int_{\tau - r^2}^{\tau} \! \int_{\alpha^{1/2}r}^{\alpha^{-1/2} r} g^2 \, dV dt \leq C \int_{\tau - r^2}^{\tau} \! \int_{\alpha r}^{\alpha^{-1} r} \left( g_1^2 +  r^{-2} g_2^2 \right) \, dV dt.
\end{equation}
This bounds the first term of ($b$). To bound the second term, we integrate (\ref{priortofubini}) and use Fubini's Theorem:
\begin{align}\label{fubini}
\nonumber r^6 \int_{t_1}^{t_2} h(r,\tau)^2 \, d\tau & \leq C \int_{t_1}^{t_2} \!\!\!\! \int_{\tau - r^2}^{\tau} \! \int_{\alpha r}^{ \alpha^{-1} r} \left( g_1^2 +  r^{-2} g_2^2 \right) \, dV dt d\tau \\
& \leq C \int_{t_1 - r^2}^{t_2} \! \int_{t}^{t + r^2}\, d\tau \int_{\alpha r}^{\alpha^{-1} r} \left( g_1^2 +  r^{-2} g_2^2 \right) \, dV dt \\
\nonumber & \leq C r^2 \int_{t_1 - r^2}^{t_2} \int_{\alpha r}^{\alpha^{-1} r} \left( g_1^2 +  r^{-2} g_2^2 \right) \, dV dt.
\end{align} 
This completes the proof of ($b$).
\end{proof}

\section{Decay estimates}\label{decaysection}

\begin{thm}\label{decaythm} Let $A(t)$ be a smooth solution of the Yang-Mills flow on $U_\rho^1 \times \LB 0, T \right).$ Choose
\begin{align*}
\nonumber & 0 \leq  \epsilon_1, \epsilon_2  \leq \epsilon < \epsilon_0, \quad \quad 0 \leq \delta_i \quad (i = 1,\ldots, 4) \\
\nonumber & 0 \leq \alpha_1, \alpha_2 \leq \alpha < 1/2, \qquad 0 < \beta_1 < 1 \\
\nonumber & 0 \leq \beta_2 < 2 - \alpha_2,  \qquad \beta_2 \neq 1 
\end{align*}
such that
\begin{align}
\label{epsilon12assn} & \epsilon_1 \rho^{3} \leq \epsilon_2 \leq \epsilon_1 \rho^{-3}, \quad \quad \sqrt{\rho} \delta_4 \leq \delta_3 \leq \delta_4 \\
\label{delta34assn} & \delta_3^2 \leq \min \left[ \delta_1^2, \epsilon_1 \right] , \quad \quad \delta_4^2 \leq \min \left[ \delta_2^2, \epsilon_2 \right].
\end{align}

In the notation of (\ref{sizeofcomponents}), assume that for $\rho \leq r \leq 1,$ there hold:
\begin{samepage}
\begin{align*}
& (1) \quad \sup_{0 \leq t < T} e(r,t) \leq \sqrt{\epsilon_1} \left( \frac{\rho}{r} \right)^{2 - \alpha_1} +  \sqrt{\epsilon_2} r^{2 - \alpha_2} \\
& (2) \quad \sup_{0 \leq t < T} r^2 h(r,t) \leq \delta_1 
\left( \frac{\rho}{r} \right)^{1 - \beta_1 } + \delta_2 r^{3 - \beta_2 } \\
& (3) \quad \sup_{\rho \leq r \leq 2\rho} \int_{0}^{T} \!\! h(r,t)^2 \, dt \leq \delta_3^2 \rho^{-2}, \quad \quad \sup_{1/2 \leq r \leq 1}\int_{0}^{T} \!\! h(r,t)^2 \, dt \leq \delta_4^2.
\end{align*}
\end{samepage}
Let $\beta_+ = \left( \beta_2 - 1 \right)_+,$ and
$$w^a(r,t) = \left( \frac{r^2}{r^2 + t} \right)^{a/2}.$$
Then for all
\begin{equation}\label{rtassumptions}
\begin{split}
\rho \leq r_1 \leq r \leq r_2 \leq 1, \quad \quad 0 \leq t_1 \leq t < t_2 \leq T
\end{split}
\end{equation}
and an appropriate $C_1$ depending on $\alpha_i, \beta_i,$ there hold:
\begin{equation}\tag{$1'$}
e(r,t) \leq C_1 \left( \begin{split} \sqrt{\epsilon_1} \left( \frac{\rho}{r} \right)^{2 - \alpha_1} \left( w^{3} (r,t) + \left(\frac{\rho}{r} \right)^{\alpha_1} \right) \quad \quad \quad \\
+ \,\, \sqrt{\epsilon_2} r^{2 - \alpha_2} \left(w^{\alpha_2}(r,t) w^{3 - \alpha_2} (1, t) + r^{\alpha_2}\right)
\end{split}
 \right)
\end{equation}
\begin{equation}\tag{$2'$}
r^2 h(r,t) \leq  C_1 \! \left( \begin{split} \left( \frac{\rho}{r} \right)^{1 - \beta_1 } \!\! \left( \delta_1 w^3(r,t) + \delta_3 \left( \frac{\rho}{r} \right)^{\beta_1} \right) \quad \quad \quad \\
+ \,\, r^{3 - \beta_2} \left(\delta_2 w^{\beta_2}(r,t) w^{3 - \beta_2}(1,t) + \delta_4 r^{\beta_2} \right)
\end{split} \right)
\end{equation}
\begin{equation}\tag{$3'$}
\begin{split}
& \int_{t_1}^{t_2} \!\!\!\! \int_{r_1}^{r_2} \!\! h^2 \, dV dt \leq C_1 \rho^2 \left(\delta_1^2
\left(\frac{r_2}{\rho} \right)^{2\beta_1} \frac{t_2 - t_1}{t_2 + r_2^2} w^{4}(r_2, t_1) + \delta_3^2 \log\left(1 +  \frac{r_2}{r_1} \right) \right) \\
& \quad \quad \quad \quad \quad \quad \quad + C_1 r_2^{6 - 2 \beta_+} \left(\delta_2^2 w^{2 \beta_+}(r_2,t_1) w^{4 - 2 \beta_2} (1, t_1) + \delta_4^2 r_2^{2 \beta_+} \right).
 \end{split}
\end{equation}
\end{thm}

\begin{proof} Let $e, f, f_1,$ etc., be as in \S \ref{splitevolutionsection}. From (\ref{f1evol}) and (1), we have
\begin{equation}\label{etau}
\begin{split}
\square f_1(r,t) & \leq \frac{C}{r^2} \left( \sqrt{\epsilon_1} \left( \dfrac{\rho}{r} \right)^{2-\alpha_1} + \sqrt{\epsilon_2} r^{2 - \alpha_2} \right)^2 \\
& \leq \frac{C}{r^2} \left( \epsilon_1 \left( \dfrac{\rho}{r} \right)^{4 - 2\alpha_1} + \epsilon_2 r^{4 - 2\alpha_2} \right) =: \eta.
\end{split}
\end{equation}
Per the Appendix, we construct a comparison function
\begin{equation}\label{ucomparison}
\begin{split}
u & = u_\varphi + u_\psi + u_\xi + u_\eta \\
\end{split}
\end{equation}
with
\begin{equation}\label{ucomparison2}
\square u_\varphi = \square u_\psi = \square u_\xi = 0, \quad \quad \square u_\eta = \eta
\end{equation}
and where $u_\varphi, u_\psi, u_\xi$ agree with $f_1$ at $t = 0, r = \rho,$ and $r = 1,$ respectively, and $u_\eta$ vanishes on the parabolic boundary.
By the comparison principle, we have
\begin{equation}\label{decaythm1}
f_1 \leq u.
\end{equation}

By Proposition \ref{trivialinitialprop}, we have
\begin{equation}\label{decaythm2}
|u_\varphi (r,t)| \leq C \left( \sqrt{\epsilon_1} \rho^{2-\alpha_1} r^{\alpha_1 - 2} w^{4 - \alpha_1} (r,t) w^{2 + \alpha_1}(1,t) + \sqrt{\epsilon_2} r^{2 - \alpha_2} w^{\alpha_2}(r,t) w^{6 - \alpha_2}(1, t) \right).
\end{equation}
Proposition \ref{insideprop}$a$ implies
\begin{equation}\label{decaythm3}
\begin{split}
| u_\psi (r,t) | & \leq C \left( \sqrt{\epsilon_1} + \sqrt{\epsilon_2} \rho^{2 - \alpha_2} \right) \left( \frac{\rho}{r} \right)^2 \\
& \leq C \sqrt{\epsilon_1} \left( \frac{\rho}{r} \right)^2
\end{split}
\end{equation}
where we have used (\ref{epsilon12assn}). Proposition \ref{outsideprop}$a$ implies
\begin{equation}\label{decaythm4}
\begin{split}
| u_\xi (r,t) | & \leq C \left( \sqrt{\epsilon_1} \rho^{2 - \alpha_1} + \sqrt{\epsilon_2} \right) r^2 \\
& \leq C \sqrt{\epsilon_2} r^2
\end{split}
\end{equation}
where we have again used (\ref{epsilon12assn}). Since $4 - 2\alpha_i > 2,$ we may apply (\ref{etau}) and Proposition \ref{inhomogeneoushardprop}, to obtain
\begin{equation}\label{decaythm5}
|u_\eta (r,t) | \leq C \left( \epsilon_1 \left(\frac{\rho}{r} \right)^2 + \epsilon_2 r^2 \right) \leq C \sqrt{\epsilon} \left( \sqrt{\epsilon_1} \left(\frac{\rho}{r} \right)^2 +\sqrt{\epsilon_2} r^2 \right).
\end{equation}
In view of (\ref{decaythm1}-\ref{decaythm5}), we have shown that $f_1$ obeys ($1'$).

To prove ($2'$) and ($3'$), we use induction on time (the continuity method). Let
\begin{align*}
w_0(r,t) & = \frac{1}{r^2} \left( \delta_1 \left( \frac{\rho}{r} \right)^{1 - \beta_1 } w^3(r,t) + \delta_2 r^{3 - \beta_2} w^{\beta_2}(r,t) w^{3 - \beta_2}(1,t) \right).
\end{align*}
For a given $0 \leq \tau \leq T,$ we make the hypothesis
\begin{equation}\label{bootstraphypothesis}
\int_{0}^{\tau} \!\!\!\! \int_{r}^{2r} \left( h - C_1 w_0 \right)_+^2 \, dV dt \leq C_1 \left( \delta_3^2 \rho^2 + \delta_4^2 r^6 \right) \quad \left( \forall \,\, \rho \leq r \leq 1/2 \right).
\end{equation}
Note from (2) that
\begin{equation}\label{w010bound}
h(r,t) \leq 3 w_0(r,t) \qquad \left( 0 \leq t \leq r^2 \right).
\end{equation}
Hence for $C_1 \geq 3$ and $\tau \leq r^2,$ (\ref{bootstraphypothesis}) is trivially satisfied. Also note that for $r = \rho$ and $r = 1/2,$ (\ref{bootstraphypothesis}) follows automatically from (3).

We claim that for $\epsilon < \epsilon_0$ and an appropriate choice of $C_1,$ 
(\ref{bootstraphypothesis}) is satisfied for $\tau = T.$
Since $h(r,t)$ is continuous, it suffices to establish the implication
\begin{equation}\label{bootstrapimplication}
(\ref{bootstraphypothesis}) \text{ holds for } C_1 \quad \Rightarrow \quad (\ref{bootstraphypothesis}) \text{ holds for } C_1/2
\end{equation}
for arbitrary $\tau < T.$ For, the set of $\tau$ satisfying (\ref{bootstraphypothesis}) must then be a nonempty subinterval of $\LB 0 , T \right)$ which is both open and closed, hence equal to the entire interval.

To prove (\ref{bootstrapimplication}), fix $0 \leq \tau <T$ and assume (\ref{bootstraphypothesis}).
Write
$$h_0 = \min \LB h, C_1 w_0 \RB, \quad \quad \quad h_1 = \left( h - C_1 w_0 \right)_+$$
$$h = h_0 + h_1.$$
According to (\ref{g1evol}-\ref{g2evol}) and ($1$-$2$), we have
\begin{equation}\label{etadefs}
\begin{split}
\square g_1(r,t) & \leq C \sqrt{\epsilon}  \left( \left( \dfrac{\rho}{r} \right)^{2-\alpha} + r^{2-\alpha} \right) \frac{ h_0 + h_1}{r^2} = : \eta_0 + \eta_1 \\
\square g_2(r,t) & \leq C \sqrt{\epsilon} \left( \left( \frac{\rho}{r}\right)^{2-\alpha} + r^{2-\alpha} \right) \frac{h_0 + h_1}{r} = r \left( \eta_0 + \eta_1 \right).
\end{split}
\end{equation}
As in (\ref{ucomparison}-\ref{ucomparison2}), we construct comparison functions
$$v = v_\varphi + v_\psi + v_\xi + v_0 + v_1, \quad \quad \quad \tilde{v} = \tilde{v}_\varphi + \tilde{v}_\psi + \tilde{v}_\xi + \tilde{v}_0 + \tilde{v}_1$$
for $g_1$ and $g_2,$ respectively, where
\begin{align*}
\square v_0 & = \eta_0, \quad  \quad \square v_1 = \eta_1 \\
\square \tilde{v}_0 & = r\eta_0, \quad \quad \square \tilde{v}_1 = r \eta_1.
\end{align*}
By the comparison principle, we have
\begin{equation}\label{vcomparison}
g_1 \leq v, \quad \quad g_2 \leq \tilde{v}
\end{equation}
on $U^4(\rho, 1) \times \LB 0, \tau \RB.$

By Propositions \ref{trivialinitialprop}, \ref{insideprop}$b,d,$ and \ref{outsideprop}, we have
\begin{align}
\label{vinitialest} r^2 v_\varphi(r,t) & \leq C \left( \delta_1 \left( \frac{\rho}{r} \right)^{1 - \beta_1 }w^{5 - \beta_1}(r,t)w^{1 + \beta_1}(1,t) + \delta_2 r^{3 - \beta_2} w^{1 + \beta_2}(r,t) w^{5 - \beta_2}(1,t) \right) \\
\label{vsidest} r^6 v_\psi^2(r,t) & + \int_{0}^{\tau} \!\!\!\! \int_{r}^{2r} v_\psi^2 \, dV dt \leq C \delta_3^2 \rho^2, \quad \quad \quad r^4 v_\xi^2(r,t) + \int_{0}^{\tau} \!\!\!\! \int_{r}^{2r} v_\xi^2 \, dV dt \leq C \delta_4^2 r^8 \\
\label{vtildeinitialest} r \tilde{v}_\varphi(r,t) & \leq C \left( \delta_1 \left( \frac{\rho}{r} \right)^{1 - \beta_1 } w^{4 - \beta_1}(r,t) w^{2 + \beta_1}(1,t) + \delta_2 r^{3 - \beta_2} w^{\beta_2}(r,t) w^{6 - \beta_2}(1,t) \right) \\
\label{vtildesidest} r^4 \tilde{v}_\psi^2(r,t) & + \int_{0}^{\tau} \!\!\!\! \int_{r}^{2r} \tilde{v}_\psi^2 \, r dr dt \leq C \delta_3^2 \frac{\rho^4}{r^2}, \quad \quad \quad r^2 \tilde{v}_\xi^2(r,t) + \int_{0}^{\tau} \!\!\!\! \int_{r}^{2r} \tilde{v}_\xi^2 \, r dr dt \leq C \delta_4^2 r^6.
\end{align}

According to (\ref{etadefs}), $\eta_0$ satisfies
\begin{align*}
r^2 \eta_0(r,t) & \leq C C_1 \sqrt{\epsilon} w_0(r,t) \leq C C_1 \sqrt{\epsilon}
\left( \delta_1 \frac{\rho^{1-\beta_1}}{r^{3 - \beta_1}} w^3(r,t) + \delta_2 r^{1 - \beta_2} w^{\beta_2}(r,t) w^{3 - \beta_2}(1,t)
  \right).
\end{align*}
Applying Proposition \ref{inhomogeneoushardprop}, we immediately have
\begin{equation}\label{v0desiredest}
|v_0(r,t)| \leq C C_1 \sqrt{\epsilon} w_0(r,t).
\end{equation}

Next, to estimate $\tilde{v}_0,$ note that
\begin{align}\label{eta0toreturnto}
\nonumber r^2 \left( r \eta_0(r,t) \right) & \leq C C_1 \sqrt{\epsilon} \left(\left( \frac{\rho}{r} \right)^{2 - \alpha_1} + r^{2-\alpha_2} \right) r w_0(r,t) \\
& \leq C C_1 \sqrt{\epsilon} 
\left(
\delta_1 \frac{\rho^{1 - \beta_1}}{r^{2 - \beta_1}} w^3(r,t) + \delta_2 \left( \rho^{2 - \alpha_1} r^{\alpha_1 - \beta_2} + r^{4 - \alpha_2 - \beta_2} \right) w^{\beta_2}(r,t) w^{3 - \beta_2}(1,t)
  \right).
\end{align}
Provided that $\beta_1 < 1,$ and $\alpha_2 + \beta_2 < 2,$ Proposition \ref{inhomogeneoushardprop} implies
\begin{align}\label{v0tildedesiredest}
\nonumber | \tilde{v}_0(r,t) | & \leq C C_1 \sqrt{\epsilon} \left(
\delta_1 \frac{\rho^{1 - \beta_1}}{r^{2 - \beta_1}} w^3(r,t) + \delta_2 \left(\rho^{2 - \alpha_1} r^{\alpha_1 - \beta_2} w^{\beta_2}(r,t) w^{3 - \beta_2}(1,t) + r^2 w^3(1,t) \right)
  \right) \\
\nonumber & \leq C C_1 \sqrt{\epsilon} \left(
\delta_1 \frac{\rho^{1 - \beta_1}}{r^{2 - \beta_1}} w^3(r,t) + \delta_2 r^{2 - \beta_2} w^{\beta_2}(r,t) w^{3 - \beta_2}(1,t) \right) \\
& \leq C C_1 \sqrt{\epsilon} r w_0(r,t).
\end{align}

By (\ref{bootstraphypothesis}), $\eta_1$ satisfies
\begin{align}\label{v1givenest}
\nonumber \int_0^\tau \!\!\!\! \int_{r}^{2r} r^4 \eta_1^2 \, dV dt & \leq C C_1 \epsilon \left(\left( \frac{\rho}{r} \right)^{4 - 2\alpha_1} + r^{4 - 2\alpha_2} \right) \left( \delta_3^2 \rho^2 + \delta_4^2 r^6 \right) \\
& \leq C C_1 \epsilon \left( \delta_3^2 \left( \rho^{6 - 2\alpha_1} r^{-4 + 2 \alpha_1} + \rho^2 r^{4 - 2 \alpha_2} \right) + \delta_4^2 \left( \rho^{4 - 2\alpha_1} r^{2 + 2\alpha_1} + r^{10 - 2\alpha_2} \right) \right).
\end{align}
By Proposition \ref{inhomogeneousprop}$a,c,$ we have
\begin{equation}\label{v1firstest}
r^6 v_1^2(r,t) + \int_0^\tau \!\!\!\! \int_{r}^{2r} v_1^2 \, dV dt \leq C C_1 \epsilon \left( \delta_3^2 \left( \rho^2 + \rho^2 r^2 \right) + \delta_4^2 \left( \rho^{4 - 2\alpha_1}r^2 + r^{8} \right) \right).
\end{equation}
Provided that $\alpha_1 \leq 1/2$ and $\delta_4 \rho \leq \delta_3,$ we may rewrite the third term
$$\delta_4^2 \rho^{3} r^2 \leq \delta_4^2 \left( \rho^4 + r^8 \right) \leq \delta_3^2 \rho^2 + \delta_4^2 r^8.$$
Then (\ref{v1firstest}) becomes
\begin{align}
\label{v1desiredintest} \int_0^\tau \!\!\!\! \int_{r}^{2r} v_1^2 \, dV dt & \leq C C_1 \epsilon \left(\delta_3^2 \rho^2 + \delta_4^2 r^8 \right) \\
\label{v1desiredsupest} r^2 v_1(r,t) & \leq C \sqrt{C_1 \epsilon} \left( \delta_3 \frac{\rho}{r} + \delta_4 r^3 \right).
\end{align}

Returning to (\ref{v1givenest}), we also have
\begin{align}\label{reta1est}
\nonumber & \int_0^\tau \!\!\!\! \int_{r}^{2r} r^4 (r \eta_1)^2 \, dV dt \\
& \quad \qquad  \leq C C_1 \epsilon \left( \delta_3^2 \left( \rho^{6 - 2\alpha_1} r^{-2 + 2 \alpha_1} + \rho^2 r^{6 - 2 \alpha_2} \right) + \delta_4^2 \left( \rho^{4 - 2\alpha_1} r^{4 + 2\alpha_1} + r^{12 - 2\alpha_2} \right) \right).
\end{align}
Proposition \ref{inhomogeneousprop}$c$ implies
\begin{equation}\label{v1tildefirstest}
\int_0^\tau \!\!\!\! \int_{r}^{2r} \tilde{v}_1^2 \, dV dt \leq C C_1 \epsilon \left( \delta_3^2 \left( \rho^4 + \rho^2 r^4 \right) + \delta_4^2 \left( \rho^3 r^4 + r^{8} \right) \right).
\end{equation}
Assuming also that $\rho \delta_4 \leq \delta_3 \leq \delta_4,$ we may simplify
$$\delta_3^2 \rho^2 r^4 \leq \delta_3^2 \rho^4 + \delta_4^2 r^8$$
$$\delta_4^2 \rho^3 r^4 \leq \delta_4^2 \left( \rho^6 + r^8 \right) \leq \delta_3^2 \rho^4 + \delta_4^2 r^8.$$
Then (\ref{v1tildefirstest}) becomes
\begin{equation}\label{v1tildedesiredintest}
\int_0^\tau \!\!\!\! \int_{r}^{2r} \tilde{v}_1^2 \, dV dt \leq C C_1 \epsilon \left( \delta_3^2 \rho^4 + \delta_4^2 r^{8} \right).
\end{equation}

From (\ref{reta1est}) and Proposition \ref{inhomogeneousprop}$a,$ provided that $\alpha_2 \leq 1/2,$ we also have
\begin{equation}\label{v1tildethirdest}
r^6 \tilde{v}_1^2 \leq C C_1 \epsilon \left( \delta_3^2 \left( \rho^4 + \rho^2 r^5 \right) + \delta_4^2 \left( \rho^{4 - 2 \alpha_1} r^4 + r^{10} \right) \right).
\end{equation}
Note that
$$\delta_3^2 \rho^2 r^5 \leq \delta_3^2 \rho^4 + \delta_4^2 r^{10}.$$
Since $\alpha_1 \leq 1/2,$ we have $\rho^{\frac{4 - 5\alpha_1}{3}} \delta_4 \leq \sqrt{\rho} \delta_4 \leq \delta_3,$ and
\begin{align*}
\delta_4^2 \rho^{4 - 2 \alpha_1} r^4 & \leq \delta_4^2 \left( \rho^{\frac{5}{3}(4 - 2\alpha_1)} + r^{10} \right)  \leq \delta_4^2 \left( \rho^{\frac{2}{3}(4 - 5\alpha_1)} \rho^4 + r^{10} \right)  \leq \delta_3^2 \rho^4 + \delta_4^2 r^{10}.
\end{align*}
Hence (\ref{v1tildethirdest}) simplifies to
\begin{align}
\nonumber r^6 \tilde{v}_1^2 & \leq C C_1 \epsilon \left( \delta_3^2 \rho^4 + \delta_4^2 r^{10} \right) \\
\label{v1tildedesiredsupest} r \tilde{v}_1 & \leq C \sqrt{C_1 \epsilon} \left( \delta_3 \left( \frac{\rho}{r} \right)^2 + \delta_4 r^3 \right)
\end{align}

Using the abbreviation
$$\dashint_{P_{r,t}} = r^{-6} \int_{t - r^2}^t \int_{3r/4}^{3r/2} \, dV dt$$
we now define
\begin{align*}
v'_0(r,t) & = \left( \dashint_{P_{r,t}} v_\varphi^2 + v_0^2 \right)^{1/2},  & \tilde{v}'_0(r,t) & = \left( \dashint_{P_{r,t}} \tilde{v}_\varphi^2 + \tilde{v}_0^2 \right)^{1/2} \\
v'_1(r,t) & = \left( \dashint_{P_{r,t}} v_\psi^2 + v_\xi^2 + v_1^2 \right)^{1/2}, &  \tilde{v}'_1(r,t) & = \left( \dashint_{P_{r,t}} \tilde{v}_\psi^2 + \tilde{v}_\xi^2 + \tilde{v}_1^2 \right)^{1/2}.
\end{align*}
In view of (\ref{vinitialest}), (\ref{vtildeinitialest}), (\ref{v0desiredest}), and (\ref{v0tildedesiredest}), we have
\begin{equation}\label{v0'supbound}
v'_0 + r^{-1} \tilde{v}'_0 \leq C \left( 1 + C_1 \sqrt{\epsilon} \right) w_0.
\end{equation}
For $2\rho \leq r \leq 1/4,$ from (\ref{vsidest}) and (\ref{v1desiredintest}), we have
\begin{align}\label{v1'intest}
\nonumber \int_{r^2}^{\tau} \!\!\! \int_r^{2r} (v_1')^2 \, dV dt & \leq C \int_{0}^{\tau} \!\!\! \int_{3r/4}^{3r} v_\psi^2 + v_\xi^2 + v_1^2 \, dV dt \\
& \leq C \left( 1 + C_1 \epsilon \right) \left( \delta_3^2 \rho^2 + \delta_4^2 r^8 \right)
\end{align}
where the first line follows by Fubini's Theorem as in (\ref{fubini}).
Similarly, by (\ref{vtildesidest}) and (\ref{v1tildedesiredintest}), we have
\begin{equation}\label{v1'tildeintest}
\int_{r^2}^{\tau} \!\! \int_r^{2r} (\tilde{v}'_1)^2 \, r dr dt \leq C \left( 1 + C_1 \epsilon \right) \left( \delta_3^2 \rho^2 \left( \frac{\rho}{r} \right)^2 + \delta_4^2 r^6 \right).
\end{equation}
Lastly, by (\ref{vsidest}), (\ref{vtildesidest}), (\ref{v1desiredsupest}) and (\ref{v1tildedesiredsupest}), we have
\begin{equation}\label{v1'finalsupbound}
r^2 v_1' + r \tilde{v}_1' \leq C \left( 1 + \sqrt{C_1 \epsilon} \right) \left( \delta_3 \frac{\rho}{r} + \delta_4 r^3 \right).
\end{equation}
By Proposition \ref{sizeofcurvprop}, (\ref{vcomparison}), and (\ref{v0'supbound}), there holds
\begin{align}
\nonumber h(r,t) \leq C \left( \dashint_{P_{r,t}} g_1^2 + r^{-2} g_2^2 \right)^{1/2} & \leq C \left( \dashint_{P_{r,t}} v^2 + r^{-2} \tilde{v}^2 \right)^{1/2} \\
\nonumber & \leq C \left( v'_0 + v'_1 + r^{-1} \left( \tilde{v}'_0 + \tilde{v}'_1 \right) \right) \\
\label{handv'bound} & \leq C \left( 1 + C_1 \sqrt{\epsilon} \right) w_0 + C \left( v'_1 + r^{-1}\tilde{v}'_1 \right).
\end{align}

To close the bootstrap, we now assume
\begin{equation}\label{bootstrapsmallness}
C \left( 1 + C_1 \sqrt{\epsilon} \right) < C_1 / 2.
\end{equation}
Then (\ref{handv'bound}) becomes
\begin{equation}
h(r,t) - \frac{C_1}{2} w_0(r,t) \leq C \left( v'_1 + r^{-1} \tilde{v}'_1 \right).
\end{equation}
Noting (\ref{w010bound}) and applying (\ref{v1'intest}-\ref{v1'tildeintest}), for $r^2 \leq \tau,$ we obtain
\begin{align*}
\int_0^\tau \!\!\! \int_r^{2r} \left( h - \frac{C_1}{2} w_0 \right)_+^2 \, dV dt & \leq C \int_{r^2}^\tau \!\! \int_{r}^{2r} (v_1')^2 + r^{-2} (\tilde{v}_1')^2 \, dV dt  \\
 & \leq \frac{C_1}{2} \left( \delta_3^2 \rho^2 + \delta_4^2 r^6 \right).
\end{align*}
This concludes the proof of (\ref{bootstrapimplication}), from which we deduce (\ref{bootstraphypothesis}) with $\tau = T.$

Finally, ($2'$) follows from (\ref{v1'finalsupbound}-\ref{handv'bound}), while ($3'$) follows by integrating (\ref{handv'bound}) and applying (\ref{wintegral}-\ref{wintegral2}), together with (\ref{v1'intest}-\ref{v1'tildeintest}).

To complete the proof of ($1'$), we must modify slightly the bounds on $g_2.$ Recall that
$$r g_2 \leq r^2h \leq e.$$
In light of (\ref{delta34assn}), in all of the above estimates, we may let
$$\beta_1 = \alpha_1, \qquad \beta_2 = \alpha_2 + 1$$
$$\delta_1 = \delta_3 = \sqrt{\epsilon_1}, \qquad \delta_2 = \delta_4 = \sqrt{\epsilon_2}.$$
In lieu of (\ref{vtildeinitialest}), we have
\begin{equation}\label{vtildealtinitialest}
 r \tilde{v}_\varphi(r,t) \leq C \!\! \left( \sqrt{\epsilon_1} \left( \frac{\rho}{r} \right)^{2 - \alpha_1 } w^{5 - \alpha_1}(r,t) w^{1 + \alpha_1}(1,t) + \sqrt{\epsilon_2} r^{2 - \alpha_2} w^{1 + \alpha_2}(r,t) w^{5 - \alpha_2}(1,t) \right)\!\!.
\end{equation}
Returning to (\ref{eta0toreturnto}), we have
\begin{align*}
\nonumber r^2 \left( r \eta_0(r,t) \right) & \leq C\left(\sqrt{\epsilon_1} \left( \frac{\rho}{r} \right)^{2 - \alpha_1} + \sqrt{\epsilon_2} r^{2-\alpha_2} \right) r w_0(r,t) \\
& \leq C \sqrt{\epsilon \epsilon_1} \frac{\rho^{3 - 2\alpha_1}}{r^{4 - 2\alpha_1}} w^3(r,t) \\
& \qquad + C \sqrt{\epsilon \epsilon_2} \left( \rho^{1 - \alpha_1} r^{\alpha_1 - \alpha_2} + \rho^{2 - \alpha_1} r^{\alpha_1 - \beta_2} + r^{4 - \alpha_2 - \beta_2} \right) w^{\beta_2}(r,t) w^{3 - \beta_2}(1,t).
\end{align*}
Proposition \ref{inhomogeneoushardprop} implies
\begin{align}
\nonumber | \tilde{v}_0(r,t) | & \leq C \sqrt{\epsilon} \left(
\sqrt{\epsilon_1} \frac{\rho^{2 - \alpha_1}}{r^{3 - \alpha_1}} w^3(r,t) + \sqrt{\epsilon_2} r^{1-\alpha_2} w^{\beta_2}(r,t) w^{3 - \beta_2}(1,t) \right) \\
\label{v0tildealtsupbound} r | \tilde{v}_0(r,t) | & \leq C \sqrt{\epsilon} \left(
\sqrt{\epsilon_1} \left( \frac{\rho}{r} \right)^{2 - \alpha_1} w^3(r,t) + \sqrt{\epsilon_2} r^{2-\alpha_2} w^{\alpha_2}(r,t) w^{3 - \alpha_2}(1,t) \right).
\end{align}
Combining (\ref{vtildesidest}), (\ref{v1tildedesiredsupest}), (\ref{vtildealtinitialest}), and (\ref{v0tildealtsupbound}), we conclude from (\ref{vcomparison}) that $r g_2$ satisfies ($1'$), as was shown earlier for $f_1.$ Proposition \ref{sizeofcurvprop}$a$ now implies that $e(r,t)$ obeys ($1'$), as desired.
\end{proof}

\begin{lemma}\label{decaylemma} For $\epsilon, \lambda, \tau > 0,$ there exists $\delta_0 = \delta_0(\epsilon,\lambda,\tau) > 0$ as follows. Let $A(t)$ be a smooth solution on $U_\lambda^1\times \LB 0, T \right),$ with
\begin{align}
\label{decaylemmaassumption1} \sup_{\substack{\lambda \leq \rho \leq 1/2 \\ 0 < t < T}} \int_{\rho}^{2\rho} |F(t)|^2 \, dV & \leq  \epsilon < \epsilon_0 \\
\label{decaylemmaassumption2} \int_{0}^{T} \!\!\!\! \int_\lambda^1 |D^*F|^2 \, dV dt & < \delta_0^2.
\end{align}
Then
\begin{equation}\label{decaylemmaest}
\sup_{\tau \leq t < T} e(r,t) < C \sqrt{\epsilon} \left( \left(\frac{\lambda }{r} \right)^2 + r^2 \right) \qquad \left( 2 \lambda \leq r \leq 1/2 \right).
\end{equation}
\end{lemma}
\begin{proof} By \cite{uhlenbeckremov} or \cite{radecay}, there exists a universal constant $C_0$ such that any Yang-Mills connection on $U^{3/4}_{3\lambda/2}$ satisfying (\ref{decaylemmaassumption1}) also obeys (\ref{decaylemmaest}).

Assume that the claim fails for a sequence $\delta_i \searrow 0.$ Then for each $i,$ there exists a solution $A_i(x,t)$ on $\LB 0, T_i \right),$ with $T_i > \tau,$ satisfying (\ref{decaylemmaassumption1}) and
$$\int_{0}^{T_i} \!\!\!\! \int_\lambda^1 |D^*F_{A_i}|^2 \, dV dt \leq \delta_i^2$$
but such that
\begin{equation}\label{decaylemmacontrad}
 e(r_i, t_i) \geq  C_0 \sqrt{\epsilon} \left( \left(\frac{\lambda}{r_i} \right)^2 + r_i^2 \right)
 \end{equation}
for a certain radius $r_i$ and time $\tau \leq t_i < T_i.$ 
Translating in time and choosing a subsequence, we may assume that $t_i = \tau$ and
$$r_i \to r_0 \in \LB 2\lambda , 1/2 \RB.$$

According to Lemma \ref{annularantibubble}$b,$ each derivative of the curvature of $A_i(x,\tau)$ is uniformly bounded. We may therefore apply the Coulomb-gauge-patching argument\footnote{To give a proof without gauge fixing, one can observe that the weak inequality (\ref{f1evol}) is preserved by $C^0$ convergence, with $\partial_t f_1 \longrightarrow 0,$ and rerun the estimates of the previous Theorem with $T = \infty.$} of \cite{donkron}, \S 4.4.2, to conclude that
$$A_i(x,\tau) \to \bar{A}(x) \mbox{ in } C^\infty\left( U^{3/4}_{3\lambda/2} \right)$$
 after changing gauges and choosing a subsequence. Since $D^*F_{A_i} \to D^*F_{\bar{A}}$ and
$$\| D^*F_{A_i} \|_{L^\infty \left( U^{3/4}_{3\lambda/2} \right) } \leq C \delta_i \to 0$$
it follows that $\bar{A}$ is Yang-Mills. But (\ref{decaylemmaassumption1}) and (\ref{decaylemmacontrad}), at radius $r_0,$ are again satisfied by $\bar{A}.$ This is a contradiction.
\end{proof}

\begin{cor}\label{decaycor} For $\epsilon, \lambda > 0,$ there exists $\delta_0 = \delta_0(\epsilon,\lambda) > 0$ as follows. Let $A(t)$ be a smooth solution on $U_\lambda^1\times \LB 0, 1 \RB,$ with
\begin{align}
\label{decaycorassumption} \sup_{\substack{\lambda \leq \rho \leq 1/2}} \int_{\rho}^{2\rho} |F(1)|^2 \, dV & \leq  \epsilon < \epsilon_0, \qquad \sqrt{\lambda}\delta \leq \delta_\lambda \\
\nonumber \int_{0}^{1} \!\!\! \int_\lambda^{2\lambda} |D^*F|^2 \, dV dt & \leq \delta_\lambda^2, \qquad \int_{0}^{1} \!\!\! \int_\lambda^1 |D^*F|^2 \, dV dt \leq \delta^2 < \delta_0^2.
\end{align}
Then
\begin{align}
\tag{$a$} e(r,1) & \leq C \sqrt{\epsilon} \left( \left(\frac{\lambda}{r} \right)^2 + r^2 \right) \qquad \left( 2 \lambda \leq r \leq 1/2 \right) \\
\tag{$b$} r^2 h(r,1) & \leq C \left( \delta_\lambda \frac{\lambda}{r} + \delta r^3 \right) \qquad \left( 2 \lambda \leq r \leq 1/2 \right).
\end{align}
\end{cor}
\begin{proof}
For $\delta_0$ sufficiently small, Lemma \ref{annularantibubble}$a$ implies (\ref{decaylemmaassumption1}), so Lemma \ref{decaylemma} immediately gives
\begin{align*}
\sup_{\frac{1}{3} \leq t \leq 1} e(r,t) & < C \sqrt{\epsilon} \left( \left(\frac{\lambda }{r} \right)^2 + r^2 \right).
\end{align*}
Lemma \ref{annularantibubble}$b$ also gives
\begin{align}\label{stupidelta}
\sup_{\frac{r^2}{3} \leq t \leq 1} r^2 h(r,t) & < C \delta.
\end{align}
In view of (\ref{stupidelta}), we may then apply Theorem \ref{decaythm} on the time interval $\LB 1/3, 2/3 \RB,$ with $\beta_1 = 1/2$ and
$$\delta_1 = \delta \lambda^{-1/2}, \quad \quad \delta_2 = 0, \quad \quad \delta_3 = \delta_4 = \delta.$$
This yields
$$\sup_{2/3 \leq t \leq 1} r^2 h(r,t) \leq C \left( \frac{\lambda}{r} \right)^{1/2} \left( \delta \lambda^{-1/2} r^3 + \delta \left( \frac{\lambda}{r} \right)^{1/2} \right) + C \delta r^3 \leq C \delta \left( \frac{\lambda}{r} +  r^{5/2} \right).$$
Again applying the Theorem over $\LB 2/3, 1 \RB,$ with $\beta_2 = 1/2$ and $\delta_3 = \delta_\lambda,$ yields ($b$).
\end{proof}

\section{Proof of Main Theorem}\label{completionsection}

Let $x \in M,$ and write $B_R = B_R(x).$ As above, we will assume for simplicity that the metric $g$ is flat, and $B_1 \Subset  M.$ 

Throughout this section, $A(t)$ will denote a smooth solution of the Yang-Mills flow on $B_1 \times \left( -1, T \right),$ with $T > 0.$
Let
$$0 < \epsilon < \epsilon_0, \qquad 0 < \bar{\lambda} < \lambda_0^2, \qquad E \geq \epsilon.$$
Here $\epsilon_0 > 0$ is a universal constant which may decrease in the course of the proofs, and $0 < \lambda_0 < 1$ will be fixed by Lemma \ref{basecaselemma} below. 

\begin{defn}\label{lambdadef} Given $0 \leq \tau < T,$ define the \emph{curvature scale} $\lambda(\tau)$ to be the minimal number $0 \leq \lambda \leq 1$ such that
\begin{equation}\label{lambdadefequation}
\sup_{\substack{\lambda < \rho < 1 \\ \tau \leq t < T }}  \int_{\rho/2}^{\rho} |F(t)|^2 \, dV < \epsilon.
\end{equation}
By convention, also let
\begin{equation*}
\lambda(\tau) = \begin{cases} \lambda(0) & \left( \tau \leq 0\right) \\
\lim_{\tau' \nearrow \,\, T} \lambda(\tau') & \left( \tau \geq T \right).
\end{cases}
\end{equation*}
The function $\lambda(\tau)$ is decreasing and continuous from the left, and strictly decreasing at time $\tau$ only if
\begin{equation}\label{lambdacondition}
\int_{\lambda(\tau)/2}^{\lambda(\tau)} |F(\tau)|^2 \, dV = \epsilon.
\end{equation}
Notice that for a given time $\tau,$ if no scale $\lambda < 1$ satisfies (\ref{lambdadefequation}), then $\lambda(\tau) = 1$ fulfills the definition vacuously.
\end{defn}

Our basic assumptions on $A(t)$ will be as follows:
\begin{align}
\label{basicassumptions1} \sup_{-1 < t < T \,\,\,} \!\! \int_{B_1} |F(t)|^2 \, dV & \leq E \\
\label{basicassumptions2} \int_{-1}^T \! \int_{B_1} |D^*F|^2 \, dV dt \leq \delta^2 & < \delta_0^2 \\
\label{basicassumptions3} \bar{\lambda} \leq \lambda(0) \leq \lambda_0.
\end{align}
Here $\delta_0 > 0$ is to be determined. 

For suitable $\hat{\rho} \leq 1,$ we will often perform the parabolic rescaling 
\begin{equation}\label{rescaling}
A_{\hat{\rho}} (x, t) = \hat{\rho} A \left( \hat{\rho} x, \hat{\rho}^2 t + t_0 \right)
\end{equation}
which preserves (YM), as well as the basic assumptions (\ref{basicassumptions1}-\ref{basicassumptions3}), after choosing $t_0 \geq 0$ to preserve (\ref{basicassumptions3}) if necessary.




\begin{lemma}\label{basecaselemma} 
There exists a universal constant $\lambda_0 > 0$ as follows. Let
$$0 \leq \tau_1 < T, \qquad \lambda_1 = \lambda(\tau_1), \qquad \tau_2 = \tau_1 + \frac{\left( \epsilon \lambda_0 \lambda_1 \right)^2 }{ E \delta^2}.$$ 
Assume (\ref{basicassumptions1}-\ref{basicassumptions3}), with $\delta_0$ depending only on $\epsilon.$ Then
\begin{equation*}
\lambda(\tau) > \lambda_0 \lambda_1 \quad (\tau \leq \tau_2).
\end{equation*}
\end{lemma}
\begin{proof} 
Assume, without loss of generality, that (\ref{lambdacondition}) is satisfied at time $\tau_1.$

Let $\tau_1 \leq \tau \leq \tau_2.$ 
By Lemma \ref{annularantibubble}$a$, we have 
\begin{equation}\label{basecaselemmapf2}
\epsilon = \int_{\lambda_1/2}^{\lambda_1} |F(\tau_1)|^2 \, dV \leq \int_{\lambda_1/4}^{2\lambda_1} |F(\tau)|^2 \, dV + \gamma
\end{equation}
where
\begin{align*}
\gamma = C \delta \left( \delta + \frac{\sqrt{\left( \tau - \tau_1 \right) E}}{\lambda_1} \right) & \leq C \left( \delta^2 + \lambda_0 \epsilon \right).
\end{align*}
Assuming that $\delta^2 < \epsilon / 4C$ and $\lambda_0 \leq 1/4C,$ we have
$$\gamma < \epsilon/2.$$
Rearranging (\ref{basecaselemmapf2}) yields
\begin{equation}\label{basecaselemmapf3}
\frac{\epsilon}{2} < \epsilon - \gamma \leq \int_{\lambda_1/4}^{2\lambda_1} |F(\tau)|^2 \, dV.
\end{equation}

Assume, for the sake of contradiction, that
$$\lambda(\tau) \leq \lambda_0 \lambda_1$$
or in particular that
\begin{equation}\label{basecaselemmaepsilonest}
\sup_{\lambda_0 \lambda_1 \leq \rho \leq 1} \int_{\rho / 2}^\rho |F(\tau)|^2 \, dV \leq \epsilon.
\end{equation}
By Lemma \ref{annularantibubble}$a,$ for $\delta$ sufficiently small (depending only on $\lambda_0,\epsilon$), for any $\rho, t$ with
$$2 \lambda_0 \lambda_1 \leq \rho \leq \lambda_1 / 2\lambda_0, \qquad \lambda_1^2 \leq t \leq \tau$$ 
we have
\begin{equation}\label{4epsilon}
\begin{split}
\int_{\rho / 2}^\rho |F(t)|^2 \, dV & \leq \int_{\rho / 4}^{2\rho} |F(\tau)|^2 \, dV + \gamma \\
& \leq 3 \epsilon + \gamma \leq 4 \epsilon.
\end{split}
\end{equation} 

We now let $\hat{\rho} = \lambda_1 / 2 \lambda_0,$ and define the rescaled solution
$$A_{\hat{\rho}}(x,t) = \hat{\rho} A \left( \hat{\rho}x, \hat{\rho}^2 \left( t - \tau \right) + \tau \right).$$
Then from (\ref{4epsilon}), this solution satisfies
\begin{equation*}
\sup_{\substack{2 \lambda_0^2 \leq \rho \leq 1/ 2 \\ \tau - 4\lambda_0^2 \leq t \leq \tau}} \int_\rho^{2\rho} |F_{A_{\hat{\rho}} }(t) |^2 \, dV \leq 4 \epsilon.
\end{equation*}
Applying Lemma \ref{decaylemma}, (\ref{decaylemmaest}) now reads
\begin{equation*}
\sup_{\tau - 4\lambda_0^2 \leq t \leq \tau} e_{A_{\hat{\rho}}} (r,t) < C \sqrt{\epsilon} \left( \left( \frac{2 \lambda_0^2}{r} \right)^2 + r^2 \right), \qquad \left( 4 \lambda_0^2 \leq r \leq 1/2 \right).
\end{equation*}
Scaling back, we have
\begin{equation*}
e(s,\tau) \leq C \sqrt{\epsilon} \left( \left( \frac{2 \lambda_0^2}{\hat{\rho}^{-1} s} \right)^2 + \hat{\rho}^{-2} s^2 \right), \qquad \left( 4 \lambda_0^2 \hat{\rho} \leq s \leq \hat{\rho}/2 \right)
\end{equation*}
which, for all $\lambda_0 \leq r \leq 1/4,$ implies a bound
\begin{equation}\label{basecaselemmaest}
\begin{split}
\int_{ \lambda_1 \lambda_0 / r}^{\lambda_1 r/ \lambda_0} |F(\tau)|^2 \, dV & \leq C \int_{ \lambda_1 \lambda_0 / r}^{\lambda_1 r / \lambda_0} e(s,\tau)^2 \frac{ds}{s} \\
& \leq C \epsilon \int_{ \lambda_1 \lambda_0 / r}^{\lambda_1 r/ \lambda_0} \left( \left( \frac{\lambda_0 \lambda_1}{s} \right)^4 + \left( \frac{\lambda_0 s}{\lambda_1} \right)^4 \right) \frac{ds}{s} \\
& \leq C \epsilon r^4.
\end{split}
\end{equation}

Assume
$$\lambda_0 \leq \frac{1}{4(2C)^{1/4}}.$$
Letting $r = 4 \lambda_0$ in (\ref{basecaselemmaest}), we obtain
$$\int^{4\lambda_1}_{\lambda_1/4} |F(\tau)|^2\, dV 
\leq \frac{ C \epsilon }{2C} \leq \frac{\epsilon}{2}$$
which contradicts (\ref{basecaselemmapf3}).
\end{proof}

\begin{prop}\label{shorttimeprop} 
Let
$$0 < \alpha < 1/2, \quad \quad 0 \leq \beta \leq 1 + \alpha$$ 
$$0 \leq \tau_0 \leq \tau_1 \leq \tau_2 \leq T, \qquad 0 < \rho < \lambda_0$$
$$
0 < \kappa \leq \kappa_0, \qquad \kappa_1 = \kappa \sqrt{\frac{\epsilon}{E}}, \qquad \rho^{\frac{2}{3}(1 - \alpha)} \leq \kappa_2 \leq \kappa \kappa_1^2
$$
$$\bar{\rho} = \kappa_2^{-1/4} \sqrt{\rho}, \qquad \bar{\tau} = \kappa_2^{-1/2} \rho^{1 - \alpha}
$$
$$0 < \rho \delta^2 \leq \delta_{\!\rho}^2 \leq \delta^2 \leq \kappa_2^2 \epsilon.$$
Here $0 < \kappa_0 \leq \lambda_0$ is a universal constant. 

Suppose that for $\rho \leq r \leq 1,$ $A(t)$ satisfies
\begin{align*}
& (1) \quad \sup_{\tau_0 \leq t < \tau_2} r^2 \left( | F| + r |\nabla F| + r^2 |\nabla^{(2)} F| \right) \leq \sqrt{\epsilon} \left( \left( \frac{\rho}{r} \right)^{2 - \alpha} + \kappa_2 r^{2 - \alpha} \right) \\
& (2) \quad \sup_{\tau_0 \leq t < \tau_2} r^3 \left( | D^* F| + r |\nabla D^*F| \right) \leq \delta \left( \left( \frac{\rho}{r} \right)^{1-\alpha} + r^{3 - \beta}  \right) \\
& (3) \quad \int_{\tau_0 - \rho^2}^{\tau_2} \! \int_{B_{2\rho}} |D^*F|^2 \, dV dt \leq \delta_{\!\rho}^2, \qquad  \quad \int_{\tau_0 - 1}^{\tau_2} \! \int_{B_1} |D^*F|^2 \, dV dt \leq \delta^2.
\end{align*}
Then for $\tau_0 \leq t < \tau_2$ and $\rho \leq r \leq 1,$ there hold:
\begin{equation}\tag{$1''$}
\begin{split}
& r^2 \left( | F(t)| + r |\nabla F(t)| + r^2 |\nabla^{(2)} F(t)| \right) \leq \\
& \quad \quad C_1 \sqrt{\epsilon} \left(\left( \frac{\rho}{r} \right)^{2 - \alpha} \left(\left(\frac{\rho}{r} \right)^{\alpha} + w^{3} (r,t - \tau_0) \right) 
+ \kappa_2 r^{2 - \alpha} \left( r^\alpha + w^{\alpha}(r,t - \tau_0) \right)
 \right)
 \end{split}
\end{equation}
 \begin{equation}\tag{$2''$}
 \begin{split}
& r^3 \left( | D^* F(t)| + r |\nabla D^*F(t)| \right) \leq \\
& \quad \quad \qquad \qquad C_1 \left( \delta_{\!\rho} \frac{\rho}{r} + \delta \left( \left(\frac{\rho}{r} \right)^{1 - \alpha} w^3(r,t - \tau_0) + r^{3 - \beta} w^\beta(r, t - \tau_0) +  r^3 \right) \right)
\end{split}
 \end{equation}
 \begin{align}
\tag{$3''$} & \int_{\tau_0 + \bar{\tau}}^{\tau_2} \! \int_{B_{2\bar{\rho}}} r^2 |D^*F|^2 \, dV dt \leq C_1 \rho^2 \left(\delta_{\!\rho}^2 \left| \log \rho \right| + \delta^2 | \log \kappa_2 | \right) \\
\tag{$4''$} & \int_{\tau_1}^{\min \LB \tau_1 + 1, \tau_2 \RB} \!\!\!\! \int_{r}^{2r} |F|^2 \, dV dt \leq C_1 \epsilon \left( \left( \frac{\rho}{r} \right)^{4} + \kappa_2^2 r^4 \right). 
\end{align}
Here $C_1$ depends on $\alpha$ and $\beta.$
\end{prop}
\begin{proof} Let 
$$\alpha_1 = \alpha_2 = \beta_1 = \alpha, \qquad \beta_2 = \beta$$
$$\delta_1 = \delta_2 = \delta_4 = \delta, \qquad \delta_3 = \delta_\rho$$
$$\qquad \qquad \quad \epsilon_1 = \epsilon, \qquad \epsilon_2 = \kappa_2^2 \epsilon.$$
With these choices, ($1$-$3$) imply the corresponding items in Theorem \ref{decaythm} (see \S\ref{sizeofcomponentssection}), and the conclusions ($1''$-$2''$) are equivalent to ($1'$-$2'$).\footnote{Strictly speaking, one should rescale by a factor of 3/4 before applying Theorem \ref{decaythm}.}

The inequality ($3'$) of Theorem \ref{decaythm} above reads 
\begin{equation}\label{priorto3''}
\int_{\tau_0 + \bar{\tau}}^{\tau_2} \! \int_{\rho}^{2\bar{\rho}} h^2 \, dV dt \leq C \left( \rho^2 \delta_{\!\rho}^2  \log \frac{\bar{\rho}}{\rho} + \delta^2 \left(
\rho^{2 - 2 \alpha}\bar{\rho}^{2\alpha} w^{4} \left(\bar{\rho}, \bar{\tau} \right) + \bar{\rho}^{6 - 2\alpha} w^{2 \alpha} \left( \bar{\rho}, \bar{\tau} \right) \right) \right).
\end{equation}
We have
$$\log \frac{\bar{\rho}}{{\rho}} = - \frac{1}{2} \log \rho - \frac{1}{4} \log \kappa_2.$$
Note that $\bar{\tau} = \bar{\rho}^2 \rho^{-\alpha},$ so
$$w^2(\bar{\rho}, \bar{\tau}) = \frac{\bar{\rho}^2}{\bar{\rho}^2 + \bar{\tau}} = \frac{1}{1 + \rho^{-\alpha}} \leq \rho^\alpha.$$
Then 
(\ref{priorto3''}) reads
\begin{align*}
\int_{\tau_0 + \bar{\tau}}^{\tau_2} \! \int_{\rho}^{2\bar{\rho}} h^2 \, dV dt & \leq C \left( - \delta_{\!\rho}^2  \rho^2 \left(\log \rho + \log \kappa_2 \right) + \delta^2 \left(
\rho^{2 - 2 \alpha}\bar{\rho}^{2\alpha} \rho^{2\alpha} + \bar{\rho}^{6 - 2\alpha} \right) \right) \\
& \leq C \rho^2 \left( \delta_{\!\rho}^2 \left| \log \rho \right| + \delta^2 \left(\frac{\rho^{\alpha}}{\kappa_2^{\alpha / 2}} + \frac{\rho^{1 - \alpha}}{\kappa_2^{(3 - \alpha)/2}} + | \log \kappa_2 | \right) \right).
\end{align*}
Since $\rho^{1 - \alpha} \leq \kappa_2^{3/2},$ this implies ($3''$).

To prove ($4''$), we integrate ($1''$) in time using (\ref{wintegral}-\ref{wintegral2}). This yields
\begin{equation}\label{shorttimeprop4''alternate}
\int_{\tau_1}^{\min \LB \tau_1 + 1, \tau_2 \RB } \!\!\!\! \int_{r}^{2r} e^2 \, dV dt \leq C \epsilon \left( \left( \frac{\rho}{r} \right)^{4} + \left(\frac{\rho}{r} \right)^{4 - 2\alpha}r^2 + \kappa_2^2 r^4 \right) \quad \left(\forall \,\,\, \rho \leq r \leq 1/2 \right).
\end{equation}
Note that
\begin{align*}
\left(\frac{\rho}{r} \right)^{4 - 2\alpha}r^2 \leq \left( \frac{\rho}{r} \right)^3 r^2 & = \left( \frac{\rho}{r} \right)^{5/2} \frac{\rho^{1/2}}{\kappa_2^{3/4}} \kappa_2^{3/4} r^{3/2}.
\end{align*}
Since by assumption
$$\frac{\rho^{1/2}}{\kappa_2^{3/4}} \leq \frac{\rho^{1-\alpha}}{\kappa_2^{3/2}} \leq 1$$
we have
\begin{equation*}
\left(\frac{\rho}{r} \right)^{4 - 2\alpha}r^2 \leq \left( \frac{\rho}{r} \right)^{5/2} \kappa_2^{3/4} r^{3/2} = \left( \frac{\rho}{r} \right)^{5/2} \left( \kappa_2 r^2 \right)^{3/4}.
\end{equation*}
Applying Young's inequality, we obtain
\begin{equation*}
\left(\frac{\rho}{r} \right)^{4 - 2\alpha}r^2 \leq \frac{5}{8} \left( \frac{\rho}{r} \right)^{4} + \frac{3}{8} \kappa_2^2 r^4.
\end{equation*}
Hence (\ref{shorttimeprop4''alternate}) implies ($4''$).
\end{proof}

\begin{prop}\label{taubarprop} Let
$$\alpha = 1/4, \qquad \beta = 5/4$$
$$0 < \kappa \leq \kappa_0, \qquad \kappa_1 = \kappa \sqrt{\frac{\epsilon}{E}}, \qquad \kappa_2 = \kappa \kappa_1^2$$
$$0 < \mu \leq 1, \qquad \rho \leq \kappa_2 \mu^3$$
$$\bar{\tau} = \kappa_2^{-1/2} \rho^{3/4}.$$
Let $\tau_0, \tau_1, \tau_2 \geq 0$ be three times satisfying
\begin{align}
\label{rescalednottoosmall} & \tau_0 + \bar{\tau} \leq \tau_1 \leq \tau_2 \leq \tau_1 + 1.
\end{align}

Assume that $A(t)$ satisfies (\ref{basicassumptions1}-\ref{basicassumptions2}), with $\delta_0$ depending on $E, \epsilon,\kappa,$ and $\mu,$ as well as (1-3) of Proposition \ref{shorttimeprop}.
Suppose further that $\lambda(\tau_0) = \rho,$ and the curvature scales at $\tau_1$ and $\tau_2$ satisfy
\begin{equation}\label{lambdataurho}
\lambda(\tau_1) \geq \mu \rho
\end{equation}
and
\begin{align}
\label{lambdatau12} \lambda(\tau_2) \leq \frac{\kappa_1}{ \lambda_0 } \lambda(\tau_1).
\end{align}
Then
\begin{equation}\label{harmonicbound}
\int_{\tau_0 - \rho^2}^{\tau_2} \! \int_{B_{2\rho}} |D^*F|^2 \, dV dt > \frac{c \mu^2 \epsilon}{|\log \rho \, |}.
\end{equation}
\end{prop}
\begin{proof} Write
$$\rho = \lambda(\tau_0), \qquad \lambda_1 =  \lambda(\tau_1), \qquad \lambda_2 = \lambda(\tau_2)$$
$$ \bar{\rho} = \kappa_2^{-1/4} \sqrt{\rho}.$$
By advancing $\tau_1,$ we may assume without loss that (\ref{lambdacondition}) is satisfied at $\tau = \tau_1.$ 
Then
$$\int_{\lambda_1/2}^{\lambda_1} |F(\tau_1)|^2 \, dV = \epsilon$$
and
\begin{equation}\label{tau1largeness}
\frac{\epsilon \lambda_1^2}{4} \leq \int_{B_{2\rho}} | F(\tau_1) |^2 r^2 dV.
\end{equation}

Note from ($4''$) and (\ref{obviousbound}) that
\begin{align*}\label{Smidwaybound}
\nonumber \left| S(\bar{\rho}, \tau_1, \tau_2) \right| & \leq C \int_{\tau_1}^{\tau_2} \!\!\!\! \int_{\bar{\rho}}^{2\bar{\rho}} |F|^2 \, dV dt \\
& \leq C \epsilon \left( \kappa_2 \rho^4 / \rho^2 + \kappa_2^2 \rho^2 / \kappa_2 \right) \\
\nonumber & \leq C \kappa_2 \epsilon \rho^2
\end{align*}
where we have also used (\ref{rescalednottoosmall}).
Provided $C \kappa \leq \lambda_0^2 / 10,$ we have
$$C \kappa_2 = C \kappa \kappa_1^2 \leq \frac{\left( \lambda_0 \kappa_1 \right)^2}{10}$$
and, from (\ref{lambdataurho})
\begin{equation}\label{Nsmallness}
\left| S(\bar{\rho}, \tau_1, \tau_2) \right| \leq \frac{\epsilon \lambda_1^2}{10}.
\end{equation}

Next, Lemma \ref{basecaselemma} asserts that $\lambda(\tau)$ can jump at most by $\lambda_0$ at a discontinuity. We may therefore assume, without loss of generality, that
$$\kappa_1 \lambda_1 \leq \lambda_2 \leq \frac{\kappa_1 \lambda_1}{\lambda_0}.$$
For $\delta$ sufficiently small (depending only on $\epsilon,\kappa_1,$ and $\mu$), Lemma \ref{annularantibubble}$a$ implies
$$\int_{r}^{2r} |F(t)|^2 \, dV < 2 \epsilon$$
$$\forall \,\,\, \tau_2 - \lambda_2^2 \leq t \leq \tau_2, \qquad 2 \lambda_2 \leq r \leq \frac{\rho}{\kappa_1 \mu^{5/2} }.$$
We may therefore apply Lemma \ref{decaylemma}, 
to obtain
\begin{equation}\label{immediatedecay}
r^2 |F(\tau_2)| \leq C \sqrt{\epsilon} \left( \frac{\lambda_2^2}{r^2} + \frac{\kappa_1^2 \mu^{5} r^2}{\rho^2} \right). 
\end{equation}
Let
$\rho_1 = \rho / \kappa_1^{1/3} \mu^{4/3} .$ Combining (\ref{immediatedecay}) with ($1$) above, we obtain
\begin{align*}
\nonumber \int_{B_{2\bar{\rho}}} |F(\tau_2)|^2 & r^2 dV  \leq \left( \int_{B_{2 \lambda_2}} \!\!\! + \int_{2 \lambda_2}^{\rho_1} + \int_{\rho_1}^{2\bar{\rho}}\right) |F(\tau_2)|^2 \, r^2 dV \\
\nonumber & \leq 4 E \lambda_2^2 + C \epsilon \left( \int_{2 \lambda_2}^{\rho_1} \left( \frac{\lambda_2^4}{r^2} + \kappa_1^4 \mu^{10} \frac{r^{6}}{\rho^{4}} \right) \frac{dr}{r} + \!\! \int_{\rho_1}^{2\bar{\rho}} \left( \frac{\rho^{4 - 2\alpha}}{r^{2 - 2\alpha} } + \kappa_2^2 r^{6 - 2\alpha} \right) \, \frac{dr}{r} \right) \\
& \leq 4 \lambda_0^{-2} \kappa^2 \epsilon \lambda_1^2 + C \epsilon \left( \kappa_1^2 \lambda_1^2 + \kappa_1^4 \mu^{10} \rho^2 \left( \frac{\rho_1}{\rho} \right)^6  + \rho^2 \left( \left( \frac{\rho}{\rho_1} \right)^{3/2} + \kappa_2^{5/8} \rho^{3/4} \right) \right) \\
& \leq C \kappa^2 \epsilon \lambda_1^2 + C \epsilon \left( \kappa_1^2 \mu^2 \rho^2  + \rho^2 \left( \kappa_1^{1/2} \mu^2 + \kappa_2 \mu^2 \right) \right).
\end{align*}
Using (\ref{lambdataurho}), this simplifies to
\begin{align*}
\nonumber \int_{B_{2\bar{\rho}}} |F(\tau_2)|^2 r^2 dV & \leq C \epsilon \lambda_1^2 \left(\kappa^2 + \kappa^{1/2} + \kappa^{3} \right).
\end{align*}
Assuming $C \sqrt{\kappa } \leq 1/10,$ we then have
\begin{equation}\label{ballsmallness}
\int_{B_{2\bar{\rho}}} |F(\tau_2)|^2 \, r^2 dV \leq \frac{\epsilon \lambda_1^2}{10}.
\end{equation}

Finally, we apply the weighted energy identity (\ref{mainenidentity}), with cutoff at $\bar{\rho}:$
\begin{equation*}
\int | F(\tau_1) |^2 \varphi_{\bar{\rho}} r^2 dV + S({\bar{\rho}} , \tau_1, \tau_2) = \int | F(\tau_2) |^2 \varphi_{\bar{\rho}} r^2 dV + 2 \int_{\tau_1}^{\tau_2} \!\!\!\! \int | D^*F |^2 \varphi_{\bar{\rho}} r^2 dV dt.
\end{equation*}
Inserting (\ref{tau1largeness}), (\ref{Nsmallness}), (\ref{ballsmallness}), 
we have
\begin{align*}
\nonumber \epsilon \lambda_1^2 \left( \frac{1}{4} - \frac{1}{10} - \frac{1}{10} \right) & \leq 2 \int_{\tau_1}^{\tau_2} \!\!\!\! \int_{B_{2\bar{\rho}}} | D^*F |^2 r^2 dV dt.
\end{align*}
Let
\begin{equation}\label{deltabarrhodef}
\bar{\delta}_\rho^2 = \max \LB \rho \delta^2, \int_{\tau_0 - \rho^2}^{\tau_2} \! \int_{B_{2\rho}} | D^*F |^2 dV dt \RB.
\end{equation}
In view of (\ref{rescalednottoosmall}), we may apply ($3''$) above, to further obtain
\begin{align*}
\label{needtodivide} \rho^2 \left( \frac{\mu^2}{20} \epsilon - C \delta^2 |\log \kappa_2| \right) & \leq C \bar{\delta}_{\!\rho}^2 \rho^2 | \log \rho |.
\end{align*}
For $\delta$ sufficiently small, depending on $E,\epsilon,\kappa,$ and $\mu,$ this simplifies to
\begin{equation}
\label{wevedivided} \frac{\mu^2 \epsilon}{C |\log \rho \, |} < \bar{\delta}_{\!\rho}^2.
\end{equation}
Then, we must have $\bar{\delta}_\rho^2 = \int_{\tau_0 - \rho^2}^{\tau_2} \! \int_{B_{2\rho}} | D^*F |^2 dV dt$ in (\ref{deltabarrhodef}); hence (\ref{wevedivided}) implies (\ref{harmonicbound}), as desired.
\end{proof}

\begin{thm}\label{technicaltheorem} 
Assume that $A(t)$ satisfies (\ref{basicassumptions1}-\ref{basicassumptions3}), with $\delta_0 > 0$ depending on $E, \epsilon$ and $\bar{\lambda}.$ Let
$$\delta_\tau = \sqrt{ \int_{\tau - 1}^{\tau + 1} \! \int_{B_{2 \lambda(\tau)}} |D^*F|^2 \, dV dt }, \quad \quad \bar{\delta}_\tau = \sup \left\{ \sqrt{\lambda(\tau)} \, \delta, \delta_{\tau'} \right\}_{\tau \leq \tau' < T - 1}.$$ 
For all
$$0 \leq \tau + 1 \leq t < T, \quad \quad \lambda(\tau) \leq r \leq 1/2, \quad \quad k \in \N$$
there hold
\begin{align}
\label{technicaltheoremassumptions} r^{2 + k} \left| \nabla^{(k)} F (t) \right| & \leq
C_{k} \sqrt{\epsilon} \left( \left( \frac{\lambda(\tau)}{r} \right)^2 + r^2 \right) \\
\label{technicaltheoremassumptions2} r^{3 + k} \left| \nabla^{(k)} D^* F (t) \right| & \leq C_{k} \left( \bar{\delta}_\tau \frac{\lambda(\tau) }{r} + \delta r^3 \right)
\end{align}
\begin{equation}\label{technicaltheoremmainestimate}
\lambda \left( \tau + 1 \right) > \kappa_1 \lambda(\tau)^{1 + K_1 \delta_\tau^2}.
\end{equation}
Here $K_1, \kappa_1 > 0$ depend on $E$ and $\epsilon.$
\end{thm}

\begin{rmk}\label{technicalremark} 
Before giving the proof of Theorem \ref{technicaltheorem}, it is worth pointing out the main technical difficulty which was omitted from the sketch in \S \ref{explanatoryremark}. While the curvature scale $\lambda(\tau)$ may be decreasing very rapidly in a potential finite-time blowup, the decay estimates require a minimal waiting time, called $\bar{\tau}$ in Propositions \ref{shorttimeprop}-\ref{taubarprop}, before taking effect. 
In the proof below, this will require us to carry out a bootstrap argument in which the decay estimates allow for sufficient control of $\lambda(\tau)$ and vice-versa.

The key step of the argument, beginning with (\ref{technicaltheoremcontradiction}), lies in establishing (\ref{technicaltheoremmainestimate}) from the other bootstrap assumptions. If the hypotheses of Proposition \ref{taubarprop} were always satisfied, then (\ref{technicaltheoremmainestimate}) would follow immediately from the standard proof of divergence of the harmonic series, in keeping with \S \ref{explanatoryremark}.
However, one cannot assume that (\ref{rescalednottoosmall}) is satisfied in general; indeed, 
the times $\tau_i$ may well be accumulating in such a way that (\ref{rescalednottoosmall}) fails for almost all terms in the series.

Fortunately, \emph{time is on our side}. In the event that (\ref{rescalednottoosmall}), or rather its time-rescaled version (\ref{sigmaeq}) below, is violated over a string of time intervals---the set of such intervals is called $J^c$ in the proof---the interval just beforehand, on which (\ref{sigmaeq}) \emph{is} satisfied, must be extremely short. After rescaling, Proposition \ref{taubarprop} will imply that enough energy has already been spent on this interval to cover the whole gap---see (\ref{sigmaclaim}) below.


\end{rmk}

\begin{proof}[Proof of Theorem \ref{technicaltheorem}] We proceed by descending induction on the curvature scale. 

As a base case, the Theorem may be established for all solutions $A(t)$ which, in addition to the stated hypotheses, satisfy
\begin{equation}\label{technicaltheorembasecasetime}
\bar{\lambda} \leq \lambda(\tau) \leq \lambda_0.
\end{equation}
Since $\kappa_1 \leq \kappa_0 \leq \lambda_0,$ and
$$\delta \leq \frac{\epsilon \lambda_0 \bar{\lambda} }{\sqrt{E}}$$
the estimate (\ref{technicaltheoremmainestimate}) follows from Lemma \ref{basecaselemma}. The estimates (\ref{technicaltheoremassumptions}-\ref{technicaltheoremassumptions2}) follow from Corollary \ref{decaycor} and Lemma \ref{annularantibubble}$b.$ 
This completes the base case. 

We now wish to establish the Theorem for all $0 \leq \lambda(\tau) \leq \bar{\lambda},$ and a fixed $\delta_0 > 0.$ By the base case, we have the freedom to assume that $\bar{\lambda}$ is arbitrarily small, which we shall use repeatedly. 

Let
\begin{equation}\label{k1defn}
\kappa_1 = \kappa_0 \sqrt{\frac{\epsilon}{E}}, \quad \quad \kappa_2 = \kappa_0 \kappa_1^2, \qquad K_1 = \frac{3}{c \kappa_1^4 \epsilon}.
\end{equation}
Here $\kappa_0$ is the universal constant determined above, and $c$ is a universal constant to be determined below.

Before proceeding to the induction step, we rescale by a factor $\hat{\rho} = \sqrt{\kappa_2},$ and make the following alternative set of hypotheses:
$$\forall \,\,\, 0 \leq \tau \leq t < T, \qquad2\lambda(\tau) \leq r \leq 1/2$$
\begin{equation}\tag{$a$}\label{technicaltheorema}
e(r,t) < C_0 \sqrt{\epsilon} \cdot \max \LB \left( \frac{\lambda(\tau)}{r} \right)^{7/4}, \kappa_2 r^{7/4} \RB
\end{equation}
\begin{equation}\tag{$b$}\label{technicaltheoremb}
\begin{split}
r^2 h(r,t) < C_0 \delta \cdot \max \LB \left( \frac{\lambda(\tau)}{r} \right)^{3/4}, r^{7/4}  \RB
\end{split}
\end{equation}
\begin{equation}\tag{$c$}\label{technicaltheoremc}
\lambda \left( \tau + 1 \right) > \kappa_1 \lambda(\tau)^{1 + K_1 \delta_\tau^2}.
\end{equation}
Here $C_0 > 1$ 
is a universal constant, for which ($a$-$b$) hold in the base case. By Proposition \ref{shorttimeprop}, 
with $\alpha = 1/4$ and $\beta = 5/4,$ and Lemma \ref{annularantibubble}, (\ref{technicaltheorema}-\ref{technicaltheoremb}) clearly imply (\ref{technicaltheoremassumptions}-\ref{technicaltheoremassumptions2}). The hypothesis ($c$) implies (\ref{technicaltheoremmainestimate}), after undoing the rescaling and redefining constants. Hence it suffices to establish the induction hypotheses ($a$-$c$).

For induction, we assume that $\lambda_c < \bar{\lambda}$ is the largest scale on which the Theorem fails, {\it i.e.}, there exists a solution $A(t)$ satisfying (\ref{basicassumptions1}-\ref{basicassumptions3}), and a
time $\tau = \tau_c$ with
$$\lambda_0 \lambda_c < \lambda(\tau_c) \leq  \lambda_c$$ 
for which (\ref{technicaltheorema}-\ref{technicaltheoremc}) are not all satisfied. We will argue, provided that 
$\delta$ is sufficiently small---depending on $E, \epsilon,$ and $\bar{\lambda},$ but independently of $\lambda_c$---that this presents a contradiction.

Note first that $\lambda_c > 0.$ If $\lambda_c = 0,$ then the Theorem holds for all solutions and times $\tau$ with $\lambda(\tau) > 0.$ But then ($c$) implies an a priori bound\footnote{See Corollary \ref{expcor} for a sharper bound.}
$$| \log \lambda(\tau) | < | \log \bar{\lambda}| + K_2 e^{K_2 \tau} \quad \left( \tau < T \right).$$
Therefore all solutions as in (\ref{basicassumptions1}-\ref{basicassumptions3}) satisfy $\lambda(\tau) > 0$ for $\tau < T,$ and we are done. Hence it suffices to assume $\lambda_c > 0.$

Under the operation (\ref{rescaling}),
the curvature scale $\lambda_{\hat{\rho}}$ 
of $A_{\hat{\rho}}$ becomes
$$\lambda_{\hat{\rho}}(\hat{\rho}^{-2}\tau) = \hat{\rho}^{-1} \lambda(\tau).$$
Rescaling, we may assume that $\lambda_c = \lambda(\tau_c),$ and that $\tau_c$ is the latest time with this property. Then, after rescaling further by any $\lambda_0^{-1} \lambda_c \leq \hat{\rho} < 1,$ the induction hypotheses hold. 

We first claim that ($a$-$b$) will automatically hold as long as
\begin{equation}\label{taurange}
\lambda(\tau) \geq \lambda_c^{1 + \gamma}
\end{equation}
where $\gamma > 0$ is a sufficiently small universal constant, in particular $\gamma \leq 1/50.$ 
Let
$$\rho_c = \frac{ \lambda_c^{1/2}}{\kappa_2^{2/7} }.$$
By induction, ($a$) is trivially satisfied for 
$r \geq \lambda_0^{-1} \rho_c$ and all $\tau_c \leq t < T.$ 

Given any $\tau \geq \tau_c$ as in (\ref{taurange}), we can rescale by a factor $\hat{\rho} = \lambda_0 \lambda_c^{\gamma} < 1,$ and again apply the induction hypothesis ($a$). Noting the scale-invariance of the first term on the LHS, we immediately have that ($a$) holds for
$r \leq \lambda_0 \lambda_c^{\gamma} \rho_c$
and all $\tau \leq t < T.$
Hence, to prove ($a$) for $\tau$ satisfying (\ref{taurange}), it remains to check the interval
\begin{equation}\label{middlerrange}
\lambda_0 \lambda_c^{\gamma} \rho_c \leq r \leq \lambda_0^{-1} \rho_c.
\end{equation}

Assume
\begin{equation}\label{gammaassumption}
K_1 \delta^2 \leq \gamma.
\end{equation}
By the induction hypothesis ($c$), we have
$$
\lambda_c > \kappa_1 \lambda(\tau_c - 1)^{1 + \gamma} $$
and
$$\lambda(\tau_c - 1) < \left( \lambda_c / \kappa_1 \right)^{1/(1 + \gamma)}.$$
Applying Proposition \ref{shorttimeprop} with $\rho = \lambda(\tau_c - 1)$ and $\tau_0 = \tau_c - 1,$ we obtain
\begin{align*}
e(r,t) & \leq C C_0 \sqrt{\epsilon} \left( \left( \frac{\left( \lambda_c / \kappa_1 \right)^{1/(1 + \gamma)}}{r} \right)^2 + r^2 \right) \qquad \left( t \geq \tau_c \right).
\end{align*}
Then, for $r$ as in (\ref{middlerrange}), we have
\begin{equation}\label{eestimatefora}
\begin{split}
e(r,t) & \leq C C_0 \sqrt{\epsilon} \left( \left( \frac{\left( \lambda_c / \kappa_1 \right)^{1 - \gamma}}{\lambda_c^{\gamma} \rho_c} \right)^2 + \rho_c^2 \right) \\
& \leq C_1 C_0 \sqrt{\epsilon} \left( \lambda_c^{1 - 4\gamma} + \lambda_c \right)
\end{split}
\end{equation}
where $C_1$ depends on $\kappa_1.$ Since $1 - 6 \gamma > 7/8,$ we have
$$C_1 \left( \lambda_c^{1 - 4\gamma} + \lambda_c \right) \leq C_1 \lambda_c^{2\gamma + 7/8} \leq C_1 \bar{\lambda}^{\gamma / 4} \left( \lambda_c^{2 \gamma + 1} \right)^{7/8} < \kappa_2 \left( \lambda_0 \lambda_c^\gamma \rho_c \right)^{7/4}$$
provided $ \bar{\lambda}$ is sufficiently small. 
Hence, (\ref{eestimatefora}) implies that ($a$) is also satisfied in the range (\ref{middlerrange}). This completes the proof of ($a$) for $\tau$ as in (\ref{taurange}), and the proof of ($b$) is similar.

It remains to assume that the hypothesis (\ref{technicaltheoremc}) fails, or
\begin{equation}\label{technicaltheoremcontradiction}
\lambda(\tau_c + 1) \leq \kappa_1 \lambda_c^{1 + K_1 \delta_{\tau_c}^2 }.
\end{equation}
Writing $\log = \log_{\kappa_1^{-1}},$ let
\begin{align*}
\nonumber L & = - \log \lambda_c, \qquad N = \lceil K_1 \delta_{\tau_c}^2 L \rceil
\end{align*}
which then obey $1 \leq N \leq \gamma L + 1,$ by (\ref{gammaassumption}).
For $i = 1, \ldots, N + 1,$ let $\tau_i$ be the latest time with
\begin{equation}\label{lambdaidef}
\lambda_i = \lambda(\tau_i) \geq \kappa_1^{i-1} \lambda_c
\end{equation}
Then $\tau_1 = \tau_c,$ and by Lemma \ref{basecaselemma} 
$$\kappa_1^{i - 1} \lambda_c \leq \lambda_{i} \leq \kappa_1^{i - 1} \lambda_0^{-1} \lambda_c.$$
This implies
\begin{equation*}
\kappa_1 \lambda_0 \lambda_{i} \leq \lambda_{i + 1} \leq \kappa_1 \lambda_0^{-1} \lambda_{i}
\end{equation*}
for $1 \leq i \leq N,$ which shows that (\ref{lambdataurho}-\ref{lambdatau12}) are satisfied, with $\mu = \lambda_0 \kappa_1,$ for each triple of times $\tau_{i -1}, \tau_{i }, \tau_{i + 1},$ with $2 \leq i \leq N.$ Recall from (\ref{taurange}) that since
$$\lambda_i \geq \lambda_c^{1 + \gamma}$$
the assumptions ($a$-$b$) 
are also satisfied for each $\tau_i.$
Hence ($1$-$3$) of Proposition \ref{shorttimeprop}, with $\delta_\rho = \delta_i,$ are satisfied for each triple of times, as required to apply Proposition \ref{taubarprop}.

To ensure a contradiction, we also define $\tau_0$ as follows. If $\lambda(\tau_c - 1/2) \geq \lambda_c / \lambda_0 \kappa_1^2,$ then let $\tau_0 \geq \tau_c - 1/2$ be the latest time with this property; otherwise, let $\tau_0 = \tau_c - 1/2.$ With this choice, (\ref{lambdataurho}-\ref{lambdatau12}) also hold, with $\mu = \left(\lambda_0 \kappa_1 \right)^2,$ for the triple $\tau_0, \tau_1, \tau_2.$ 

Let
\begin{align}\label{technicaltheoremdefinitions}
\nonumber \sigma_i & = \tau_{i+1} - \tau_i \qquad \left(i = 0, \ldots, N \right)\\
\rho_i & = \frac{\lambda_{i -1} }{\sqrt{\sigma_{i}}} \quad \quad \left(i = 1, \ldots, N \right) \\
\nonumber \delta_i^2 & = \int_{\tau_{i-1} - \lambda_{i-1}^2}^{\tau_{i+1}} \! \int_{B_{2 \lambda_{i - 1} }} |D^*F|^2 \, dV dt \quad \quad \left(i = 1, \ldots, N\right).
\end{align}
The assumption (\ref{technicaltheoremcontradiction}) implies
$$1 \geq \tau_{N+1} - \tau_1 = \sum_{i = 1}^N \sigma_i $$
and in particular, $\sigma_i \leq 1.$ For $\delta$ sufficiently small, by Lemma \ref{basecaselemma}, we may also assume
\begin{equation}\label{rhoibasicbounds}
\rho_i \leq \kappa_2^{10}.
\end{equation}

Let $J = \{ i_j \}_{j = 1}^M \subset \{1, \ldots,  N \}$ be the set of all $i$ such that
\begin{equation}\label{sigmaeq}
\begin{split}
\rho_i^{3/4} \leq \sqrt{\kappa_2} \frac{\sigma_{i-1}}{\sigma_{i}}.
\end{split}
\end{equation}
Note that after rescaling by $\sqrt{\sigma_{i}},$ (\ref{sigmaeq}) is precisely (\ref{rescalednottoosmall}), and Proposition \ref{taubarprop} can be applied. Also let $i_{M + 1} = N + 1$ for notational purposes. Let $J^c$ denote the complement of $J$ in $\{1, \ldots, N\}.$

Letting $N_0 = \lfloor \frac{3}{4} N \rfloor,$ we first claim that 
$J \cap \left\{1, \ldots, N_0 + 1 \right\} $ is nonempty. Assuming the contrary, we have $k \in J^c$ for all $1 \leq k \leq N_0 + 1.$ The negation of (\ref{sigmaeq}) reads
$$\rho_{k}^{3/4} = \frac{\lambda_{k - 1}^{3/4}}{\sigma_{k}^{3/8}} > \kappa_2^{1/2} \frac{\sigma_{k - 1}}{\sigma_{k}}$$
and
\begin{equation}
\begin{split}
\sigma_{k - 1} & < \kappa_2^{-1/2} \lambda_{k - 1}^{3/4} \sigma_{k}^{5/8} \leq \lambda_c^{2/3}.
\end{split}
\end{equation}
Hence
$$\tau_{N_0+1} - \tau_{0} = \sum_{i = 0}^{N_0} \sigma_i \leq \lambda_c^{2/3} \left( N_0 + 1 \right) \leq \lambda_c^{2/3} | \log \lambda_c | \leq \lambda_c^{1/2}$$
for $ \bar{\lambda}$ sufficiently small. Rescaling by
$$\hat{\rho} = \sqrt{\tau_{N_0 + 1} - \tau_0} \leq \lambda_c^{1/4}$$
and writing $\hat{\tau}_i$ for the rescaled times, we have
\begin{equation}\label{scalingcontradpriortosqueeze}
\lambda_{\hat{\rho}} (\hat{\tau}_0) \geq \frac{\lambda_c}{\lambda_0 \kappa_1^2 \hat{\rho} } \geq \frac{\lambda_c}{\lambda_0 \kappa_1^2 \lambda_c^{1/4}} > \lambda_c^{3/4}.
\end{equation}
But
\begin{equation*}\label{scalingcontradsqueeze}
\begin{split}
\lambda_{\hat{\rho}}(\hat{\tau}_0 + 1) = \lambda_{\hat{\rho}}(\hat{\tau}_{N_0 + 1}) = \frac{\lambda_{N_0 + 1}}{\hat{\rho}} & \leq \frac{ \lambda_c }{\lambda_0 \hat{\rho}} \kappa_1^{N_0}. 
\end{split}
\end{equation*}
Note from (\ref{scalingcontradpriortosqueeze}) that $\frac{ \lambda_c }{\lambda_0 \hat{\rho}} \leq \kappa_1^2 \lambda_{\hat{\rho}} (\hat{\tau}_0).$ Then
\begin{equation*}
\begin{split}
\lambda_{\hat{\rho}}(\hat{\tau}_0 + 1) \leq  \kappa_1^{N_0 + 2} \lambda_{\hat{\rho}}(\hat{\tau}_0) & \leq \kappa_1^{\frac{3}{4} N + 1} \lambda_{\hat{\rho}}(\hat{\tau}_0)  \\
& \leq \kappa_1 \lambda_c^{\frac{3}{4} K_1 \delta_{\tau_c}^2} \lambda_{\hat{\rho}}(\hat{\tau}_0) \\
& \leq \kappa_1 \lambda_{\hat{\rho}}(\hat{\tau}_0)^{1 + K_1 \delta_{\hat{\tau}_0}^2 }.
\end{split}
\end{equation*}
Since $\lambda_{\hat{\rho}}(\hat{\tau}_0) > \lambda_c,$ this contradicts the induction hypothesis (\ref{technicaltheoremc}), establishing the claim. 

Next, note from (\ref{technicaltheoremdefinitions}) that
\begin{align}\label{rhosigmabound}
\frac{\rho_{i+1}}{\rho_{i}} = \sqrt{\frac{\sigma_{i}}{\sigma_{i+1}}} \frac{\lambda_i}{\lambda_{i-1}} \leq \sqrt{\frac{\sigma_{i}}{\sigma_{i+1}}} \left( \lambda_0 \kappa_1 \right)^{-2}.
\end{align}
If $k \in J^c$ with $k \geq 2,$ the negation of (\ref{sigmaeq}) again reads
\begin{align*}
\rho_{k}^{3/4} > \kappa_2^{1/2} \frac{\sigma_{k-1}}{\sigma_{k}} \geq \kappa_2^{1/2} \left( \lambda_0 \kappa_1 \right)^4 \left( \frac{\rho_k}{\rho_{k - 1}} \right)^2
\end{align*}
where we have substituted (\ref{rhosigmabound}). Using (\ref{rhoibasicbounds}), we have
$$\rho_{k - 1}^2 \geq \kappa_2^{1/2} \left( \lambda_0 \kappa_1 \right)^4 \rho_k^{5/4} \geq \kappa_2^{5/2} \rho_k^{5/4} \geq \rho_k^{3/2}$$
$$\rho_{k-1} \geq \rho_k^{3/4}.$$
Hence, if $1 \leq i_1 < i_2 \leq N + 1$ are such that $\{ i_1 + 1, \ldots, i_2 - 1\} \subset J^c,$ then
\begin{align}
\label{sigmaprecalc} \log \rho_{i_1} & \geq \left(\frac{3}{4}\right)^{i_2 - i_1 - 1} \log \rho_{i_2 - 1} \geq - \left(\frac{3}{4}\right)^{i_2 - i_1 - 1} \left( 1 + \gamma \right) L.
\end{align}

We now claim that 
\begin{equation}\label{sigmaclaim}
\delta_{i_j}^2 > c \kappa_1^4 \epsilon \frac{i_{j+1} - i_j}{L} \qquad \left( i_j \in J \right).
\end{equation}
If $i_{j + 1} = i_j + 1,$ this follows directly from Proposition \ref{taubarprop}, with $\mu = \left( \lambda_0 \kappa_1 \right)^2.$ 
In general, 
by (\ref{sigmaprecalc}), we have
\begin{equation*}\label{sigmacalc}
\begin{split}
\log \rho_{i_j} & \geq - \left( \frac{3}{4} \right)^{\left( i_{j+1} - i_j - 1 \right)} \left( 1 + \gamma \right) L.
\end{split}
\end{equation*}
Applying Proposition \ref{taubarprop} again yields
\begin{align*}
\delta_{i_j}^2 > \frac{c \mu^2 \epsilon }{- \log \rho_{i_j} } & \geq \left( \frac{4}{3} \right)^{\left( i_{j+1} - i_j - 1 \right)} \frac{c \left( \lambda_0 \kappa_1 \right)^4 \epsilon }{\left( 1 + \gamma \right) L} \geq c \kappa_1^4 \epsilon \frac{i_{j+1} - i_j}{L}
\end{align*}
as claimed. 

Finally, note from Lemma \ref{basecaselemma} that
$$\LB \tau_{i} - \lambda_{i}^2 , \tau_{i + 2} \RB \subset \LB \tau_{i - 1}, \tau_{i + 2} \RB$$
for $\delta$ sufficiently small. Hence, by (\ref{technicaltheoremdefinitions}), the domains of integration of the $\delta_i$'s overlap at most 3 times. 
Using (\ref{sigmaclaim}), we obtain
\begin{equation*}
3 \delta_{\tau_c}^2 \geq \sum_{i = 1}^{N} \delta_i^2 \geq \sum_{j = 1}^M \delta_{i_j}^2 > \frac{c \kappa_1^4 \epsilon}{L} \sum_{j = 1}^M \left( i_{j+1} - i_j \right) \geq \frac{c \kappa_1^4 \epsilon N}{4L} \geq \frac{c \kappa_1^4 \epsilon \left( K_1 \delta_{\tau_c}^2 L \right)}{ L} \geq c K_1 \kappa_1^4 \epsilon \delta_{\tau_c}^2.
\end{equation*}
In view of (\ref{k1defn}), this simplifies to
$$3 > c K_1 \kappa_1^4 \epsilon = 3$$
which is a contradiction.

We have established ($a$-$c$), which imply (\ref{technicaltheoremassumptions}-\ref{technicaltheoremmainestimate}).
\end{proof}

\begin{cor}\label{expcor} Assume (\ref{basicassumptions1}-\ref{basicassumptions3}), with $\delta_0$ depending on $E, \epsilon,$ and $\bar{\lambda}.$ For $\tau \geq 0,$ the curvature scale satisfies
$$\lambda(\tau) > c e^{- K_2\tau} \cdot \lambda(0)$$
for a constant $K_2$ depending on $E$ and $\epsilon.$
\end{cor}
\begin{proof}
Let $k_1 = - \log \kappa_1.$ For $ n \in \N,$ write
$$\ell_n = - \log \lambda(n)$$
and let
$$\delta_n = \delta_{\tau = n}$$
be as in the statement of Theorem \ref{technicaltheorem} (but different from (\ref{technicaltheoremdefinitions})). Applying logarithms to (\ref{technicaltheoremmainestimate}) yields
\begin{align}
\nonumber \ell_{n + 1} & < k_1 + \left( 1 + K_1 \delta_n^2 \right) \ell_n \\
\label{corbootstrap} \ell_{n + 1} - \ell_n & < k_1 + K_1 \delta_n^2 \ell_n.
\end{align}
We claim
\begin{equation}\label{corinduction}
\ell_n \leq 2 k_1 n + \ell_0.
\end{equation}
The base case $n = 0$ is trivial. For induction, assume (\ref{corinduction}) for a fixed $n \geq 0.$ Summing (\ref{corbootstrap}), we obtain
\begin{align*}
\ell_{n + 1} - \ell_0 = \sum_{m = 0}^n \left( \ell_{m+1} - \ell_m \right) & \leq (n + 1) k_1 + K_1 \left(2 k_1 n + \ell_0 \right) \sum_{m = 0}^{n} \delta_m^2 \\
& \leq \left( k_1 + 2 K_1 \left(2 k_1 + \ell_0 \right) \delta^2 \right)(n + 1).
\end{align*}
Assuming $2 K_1 \left( 2 k_1 + | \log \bar{\lambda}| \right) \delta^2 \leq k_1,$ this proves the claim for $n + 1,$ completing the induction.

Exponentiating (\ref{corinduction}) gives the desired result.
\end{proof}

\begin{proof}[Proof of Theorem \ref{mainthm}]\label{proofofmainthm} According to Proposition \ref{epsilonreg}, it suffices to prove (\ref{mainthmlimsupbound}) for $\epsilon_0,$ the universal constant of the above theorems. 

Assume, for the sake of contradiction, that
\begin{equation}\label{epsilon0contradiction}
\lim_{\lambda \to 0} \limsup_{t \to T} \int_{B_{\lambda}} |F(t)|^2 \, dV = E_0 \geq \epsilon_0.
\end{equation}
The outer limit exists because the $\limsup$ is finite, by (\ref{globalenineq}), and decreasing with respect to $\lambda.$

For $\delta > 0$ arbitrary, we may choose $\lambda_1 > 0$ and $\tau_1 < T$ such that
\begin{align*}\label{ltecontradassumptions}
\sup_{\tau_1 \leq t < T} \int_{B_{\lambda_1}} |F(t)|^2 \, dV \leq E_0 + \delta.
\end{align*}
Since $|D^*F|^2$ is integrable, by (\ref{globalenineq}), we may further assume
\begin{equation*}
\int_{\tau_1}^T \!\!\! \int_{B_{\lambda_1}} |D^*F|^2 \, dV dt < \delta^2.
\end{equation*}
After rescaling, we may assume that $\lambda_1 = 1, \tau_1 = -1,$ and $T > 0.$ 
Hence, the basic assumptions (\ref{basicassumptions1}-\ref{basicassumptions2}) are satisfied.

Note first that
\begin{equation}\label{lambdanonzero}
\lambda(\tau) > 0 \quad \left(\tau < T \right)
\end{equation}
provided that in Definition \ref{lambdadef}, we choose
$$\epsilon < \frac{\epsilon_0}{2 C}.$$
For, assume that $\lambda(\tau) = 0$ for some $\tau < T.$
Since $A(\tau)$ is smooth, we may choose $0 < \lambda < 1$ such that
$$\int_{B_{\lambda}} |F(\tau)|^2 \, dV \leq \frac{\epsilon_0}{2}.$$
Applying Proposition \ref{spendingprop}, with $\epsilon_1 = \epsilon_0/2$ and $\epsilon_2 = \epsilon,$ we have
$$\sup_{\tau \leq t < T} \int_{B_{\lambda/2}} |F(t)|^2 \, dV \leq \frac{\epsilon_0}{2} + C \epsilon < \epsilon_0.$$
This contradicts (\ref{epsilon0contradiction}), establishing (\ref{lambdanonzero}). 

On the other hand, we must have 
\begin{equation}\label{lambdato0}
\lambda(\tau) \longrightarrow 0 \quad (\tau \to T).
\end{equation}
If not, since $\lambda(\tau)$ is decreasing, we have $\lambda(\tau) \to \lambda > 0.$ Hence there exists $\tau$ with
\begin{equation}\label{lambdastar}
T - \lambda^2 \leq \tau < T
\end{equation}
for which, letting $U = U_{\lambda(\tau)/2}^{\lambda(\tau)},$ there holds
$$\int_{U} |F(\tau)|^2 \, dV \geq \epsilon.$$
But then
\begin{equation*}
\begin{split}
\int_{B_{\lambda/2}} |F(\tau)|^2 \, dV \leq \left( \int_{B_1} - \int_{U} \right) |F(\tau)|^2 \, dV
& \leq E_0 + \delta - \epsilon.
\end{split}
\end{equation*}
In view of Lemma \ref{antibubble}$a$ and (\ref{lambdastar}), provided that
$$\delta \leq \frac{\epsilon}{100 ( \sqrt{E_0} + 1 ) }$$ 
we have
\begin{equation*}
\sup_{\tau \leq t < T} \int_{B_{\lambda/4}} |F(t)|^2 \, dV \leq E_0 - \epsilon / 2.
\end{equation*}
This contradicts (\ref{epsilon0contradiction}), establishing (\ref{lambdato0}). 

But in view of (\ref{lambdanonzero}-\ref{lambdato0}), 
we may choose $0 \leq \tau_0 < T$ such that $0 < \lambda(\tau_0) \leq \lambda_0.$
Rescaling so that $\tau_0 = 0$ and $\lambda(0) = \lambda_0,$ the basic assumption (\ref{basicassumptions3}) is now satisfied. 
Since $T < \infty,$ Corollary \ref{expcor} implies that $\lambda(\tau) \, \not \!\! \longrightarrow 0$ as $\tau \to T,$ which in turn contradicts (\ref{lambdato0}).

Hence (\ref{epsilon0contradiction}) is impossible; the limit must be less than $\epsilon_0,$ as desired. 
\end{proof}

\appendix

\section{Annular heat equation}


Recall the heat operator
\begin{equation*}
\square = \partial_t - \left( \partial_r^2 + \frac{1}{r}\partial_r - \frac{4}{r^2} \right)
\end{equation*}
defined in (\ref{square}) above. In this appendix, we study a general solution of the inhomogeneous heat equation
$$\square u(r,t) = \eta(r,t)$$
on $\LB \rho, R \RB \times \LB 0, \infty \right),$ where $0 < \rho < R,$ with boundary conditions
$$u(r,0) = \varphi(r) \qquad (\rho \leq r \leq R)$$
$$u(\rho, t) = \psi(t), \qquad u(R, t) = \xi(t) \qquad (t > 0).$$
Propositions \ref{trivialinitialprop}-\ref{outsideprop} estimate separately the initial and boundary components of the homogeneous solution, while Propositions \ref{inhomogeneoushardprop}-\ref{inhomogeneousprop} estimate the particular solution. The proof of Theorem \ref{decaythm} relies on these results.

For the sake of completeness, we will sometimes include more detailed formulae than are strictly needed above.

\subsection{Radial heat kernel} Denote the annulus
$$U^n(\rho,R) = \left\{ \rho < r < R \right\} \subset \R^n$$
and let $H_{(\rho, R)}^n(r,s,t)$ be the radial Dirichlet heat kernel on $U^n(\rho,R).$ 

By the comparison principle, $H_{(\rho, R)}^n$ is dominated by the spherical average of the full heat kernel of $\R^n.$ An elementary computation with the latter shows
\begin{equation*}
0 \leq H_{\left( \rho ,R \right)}^n(r,s,t) \leq C_n \frac{ e^{- \left( r - s \right)^2 / 4t} }{ t^{\frac{1}{2}} \left(rs + t \right)^{\frac{n - 1}{2}} }.
\end{equation*}

Let $\square$ be the heat operator (\ref{square}). Note that
$$\square \, u(r,t) = 0 \iff \left( \partial_t + \Delta_{\R^6} \right)  \frac{u(r,t)}{r^2} = 0.$$
The solution $u_\varphi$ of the Dirichlet problem
\begin{equation}\label{initialprob}
\square u_\varphi(r,t) = 0
\end{equation}
\begin{equation}\label{initialboundaryprob}
\begin{split}
\qquad u_\varphi(r,t) & = \varphi(r) \qquad (\rho \leq r \leq R) \\
\qquad u_\varphi(\rho,t) & = u_\varphi(R,t) = 0 \qquad (0 < t < \infty).
\end{split}
\end{equation}
is therefore given by
\begin{equation*}
\frac{u_\varphi(r,t)}{r^2} = \int \frac{\varphi(s)}{s^2} \, H_{\left( \rho , R \right)}^6(r,s,t) \, s^5 ds
\end{equation*}
and
\begin{equation}\label{initialkernelbound0}
\begin{split}
u_\varphi(r,t) & = \int \varphi(s) \, r^2 H_{\left( \rho , R \right)}^6(r,s,t) \, s^3 ds \\
& \leq C \int_\rho^R | \varphi(s) | \, e^{-(r-s)^2 / 4t} \frac{r^2 s^3}{ \left( rs + t \right)^{5/2} } \, \frac{ds}{t^{1/2}}.
\end{split}
\end{equation}
The following Lemma will be used repeatedly below.

\begin{lemma}\label{holderlemma} Let
$$s > 0, \qquad 0 <  r_1 \leq r_2, \qquad 0 \leq t_1 < t \leq t_2, \qquad a,b,c,d \geq 0$$
$$w^m(s,t) = \left( \frac{s^2}{s^2 + t} \right)^{m/2}.$$
There holds
\begin{align*}\tag{$a$}
\int_{r_1}^{r_2} & e^{-(r-s)^2/4t} \frac{s^{2a} r^{2b} \left( r^2 + t \right)^c }{\left( rs + t \right)^{a+b+c+d}} \frac{dr}{t^{1/2}} \\
& \leq  \frac{C_{a,b,c,d}}{\left( s^2 + t \right)^{d}} \frac{r_2 - r_1}{r_2 - r_1+\sqrt{t}} \begin{cases}
\left( \dfrac{s^{2}}{r_1^2 + t} \right)^{a} w^{2b}(r_2, t) \, e^{-\left( r_1 - s\right)^2 / 5t} & (s \leq r_1) \\
w^{2a}(s,t) w^{2b}(r_2, t) & (r_1 \leq s \leq r_2) \\
w^{2a}(s,t) \left( \dfrac{r_2^2}{s^2 + t}\right)^{b} e^{-\left( s - r_2 \right)^2 / 5t} & (s \geq r_2).
\end{cases}
\end{align*}
Provided $a +d > 1,$ also
\begin{align*}\tag{$b$}
 \int_{t_1}^{t_2} \!\!\!\! \int_{r_1}^{r_2} & e^{-(r-s)^2/4t} \frac{s^{2a} r^{2b} \left( r^2 + t \right)^c }{\left( rs + t \right)^{a+b+c+d}}  \frac{dr}{t^{1/2}} dt \\
& \leq \frac{C_{a,b,c,d}}{\left( s^2 + t_1\right)^{d-1}} \dfrac{t_2 - t_1}{t_2+ s^2} \begin{cases}
\left( \dfrac{s^2}{r_1^2 + t_1} \right)^{a} w^{2b+1}(r_2, t_1) \, e^{-\left( r_1 - s\right)^2 / 5t_2} & (s \leq r_1) \\[3mm]
w^{2a}(s,t_1) w^{2b+1}(r_2, t_1) & ( r_1 \leq s \leq r_2) \\[3mm]
w^{2a}(s,t_1) \left( \dfrac{r_2^2}{s^2 + t_1}\right)^{b \, + \frac{1}{2}} e^{-\left( s - r_2 \right)^2 / 5t_2} & (s \geq  r_2).
\end{cases}
\end{align*}
\end{lemma}
\begin{proof} First note for $r \leq r_2$ that
\begin{equation}\label{b=0}
\frac{r^2}{rs + t} \leq \begin{cases} w^2(r_2, t) \dfrac{r^2 + t}{rs + t} \\
\dfrac{r_2^2}{s^2} \dfrac{s^2}{rs + t}. 
\end{cases}
\end{equation}
Hence it will suffice to prove ($a$) with $b=0.$

Let
$$u = \frac{r-s}{\sqrt{t}}, \qquad v = \frac{s}{\sqrt{t}}, \qquad u_1 = \frac{r_1-s}{\sqrt{t}}, \qquad u_2 = \frac{r_2-s}{\sqrt{t}}.$$
In these coordinates, we have
\begin{equation}\label{holderlemma1}
\begin{split}
\int_{r_1}^{r_2} \, e^{-(s-r)^2 / 4t} \frac{s^{2a}(r^2 + t)^{c}}{ \left( rs + t \right)^{a + c + d} } \, \frac{dr}{t^{1/2}} & = t^{-d} \int_{u_1}^{u_2} e^{-u^2/4} \frac{ v^{2a}\left( |u+v|^2 + 1 \right)^{c}}{\left(v|u+v| + 1\right)^{a + c + d}} \, du.
\end{split}
\end{equation}
Note that
\begin{equation*}
\begin{split}
v^2 \leq v (|u| + |u + v|) & \leq v |u| + v |u+v| 
 \leq \frac{1}{2} \left( u^2 + v^2 \right) + v |u+v| \\
 v^2 & \leq u^2 + 2 v |u + v|
\end{split}
\end{equation*}
and
\begin{equation*}
\begin{split}
1 + v^2 & \leq 1 + u^2 + 2v |u+v| 
\leq \left( 1 + v|u+v| \right) (2 + u^2).
\end{split}
\end{equation*}
Therefore
\begin{equation*}\label{holderlemma2}
\begin{split}
\left( 1 + v|u+v| \right)^{-1} \leq (2 + u^2) \left( 1 + v^2 \right)^{-1} \\
\frac{v^2}{1 + v|u+v|} \leq (2 + u^2) \frac{v^2}{1 + v^2}, \qquad \quad
\frac{1 + |u+v|^2}{1+v|u+v|} & \leq 2 + u^2.
\end{split}
\end{equation*}
Equation (\ref{holderlemma1}) becomes
\begin{equation*}
\begin{split}
\int_{r_1}^{r_2} \, e^{-(s-r)^2 / 4t} \frac{s^{2a}(r^2 + t)^{c}}{ \left( rs + t \right)^{a + c + d} } \, \frac{dr}{t^{1/2}}  & \leq \frac{v^{2a}}{ t^d \left( 1+v^{2a} \right) \left( 1 + v^{2d} \right)} \int_{u_1}^{u_2} e^{-u^2/4}\left( 2 + u^2 \right)^{a + c + d}  \, du \\
& \leq \frac{C}{t^d}\frac{v^{2a}}{1+v^{2(a+d)}} \cdot e^{- \left( \max \LB u_1, - u_2, 0 \RB \right)^2/5} \cdot \min \LB u_2 - u_1, 1 \RB.
\end{split}
\end{equation*}
Undoing the change-of-variable, and substituting (\ref{b=0}), we obtain
\begin{equation*}
(a) \, \leq \frac{C}{\left(s^2 + t \right)^d} \frac{r_2 - r_1}{r_2 - r_1 + \sqrt{t}} \begin{cases}
w^{2a}(s, t) w^{2b}(r_2, t) \, e^{-\left( r_1 - s\right)^2 / 5t} & (s \leq r_1) \\
w^{2a}(s,t) w^{2b}(r_2, t) & (r_1 \leq s \leq r_2) \\
w^{2(a + b)}(s,t) \left( \dfrac{r_2}{s} \right)^{2b} e^{-\left( s - r_2 \right)^2 / 5t} & (s \geq r_2).
\end{cases}
\end{equation*}
For $s \geq r_1$ this is the desired statement. For $s \leq r_1/2,$ we observe that
\begin{align}\label{dumbsrcomputation}
\frac{e^{-(s-r_1)^2 / 5t} }{s^2 + t} & \leq \frac{r_1^2 + t}{(r_1^2 + t)(s^2 + t)} e^{-(s-r_1)^2 / 5t}
\leq \frac{r_1^2 + t}{t} \frac{e^{-(s-r_1)^2 / 5t}}{r_1^2 + t} \leq C \frac{e^{-(s-r_1)^2 / 6t}}{r_1^2 + t}.
\end{align}

To prove ($b$), note that
\begin{equation*}
\frac{r_2 - r_1}{r_2 - r_1 + \sqrt{t}} \leq \frac{r_2}{r_2 + \sqrt{t}} \leq w^1(r_2, t).
\end{equation*}
For $s \geq 2 r_2,$ calculating as in (\ref{dumbsrcomputation}) also yields
$$\frac{e^{-(s - r_2)^2/5t}}{r_2 + \sqrt{t}} \leq C \frac{e^{-(s - r_2)^2/6t}}{s + \sqrt{t}}.$$
Then ($b$) follows by integrating in time using the following formulae. For $a > 1,$ we compute
\begin{align}\label{wintegral}
\int_{t_1}^{t_2} w^{2a}(r,t) \, dt 
= \left. \frac{r^{2a}}{1 - a} \left( r^2 + t \right)^{1-a} \right|^{t_2}_{t_1}
& = \frac{r^2}{a-1} w^{2a - 2}(r,t_1) \left( 1 - \left( \frac{r^2 + t_1}{r^2 + t_2} \right)^{a-1} \right) \\ \nonumber
& \leq C_a r^2 w^{2a - 2}(r,t_1) \left( 1 - \frac{r^2 + t_1}{r^2 + t_2} \right) \\ \nonumber
& \leq C_a r^2 w^{2a - 2}(r,t_1) \frac{t_2 - t_1}{r^2 + t_2}.
\end{align}
For $0 \leq a < 1,$ we have
\begin{equation}\label{wintegral2}
\begin{split}
\int_{t_1}^{t_2} w^{2a}(r,t) \, dt 
& = \frac{r^{2a}}{1-a} \left( \left( r^2 + t_2 \right)^{1 - a} - \left( r^2 + t_1 \right)^{1-a} \right) \\
& \leq \frac{r^{2a}}{1-a} \left( t_2 - t_1 \right)^{1-a}.
\end{split}
\end{equation}
\end{proof}

\subsection{Initial data}\label{initialdatasection}

Let $u_\varphi$ be the solution of the Dirichlet problem (\ref{initialprob}-\ref{initialboundaryprob}). Per (\ref{initialkernelbound0}), we have a pointwise bound
\begin{equation}\label{initialkernelbound}
\begin{split}
|u_\varphi(r,t)| & \leq C \int_\rho^R | \varphi(s) | \, e^{-(r-s)^2 / 4t} \frac{r^2 s^3}{ \left( rs + t \right)^{5/2} } \, \frac{ds}{t^{1/2}}.
\end{split}
\end{equation}

\begin{prop}\label{trivialinitialprop} Let
$$w^a(r,t) = \left( \frac{r^2}{r^2 + t} \right)^{a/2}.$$
For $-3 \leq k \leq 2,$ assuming
$$|\varphi (r) | \leq A r^k$$
there holds
\begin{equation*}
|u_\varphi(r,t)| \leq C A \, r^k w^{2-k}(r,t) w^{4+k}(R,t).
\end{equation*}
\end{prop}
\begin{proof} We have
\begin{align*}
| u_\varphi(r,t) | & \leq C A \int_{\rho}^R \, e^{-(r-s)^2 / 4t} \frac{r^2 s^{3 + k} }{ \left( rs + t \right)^{5/2} } \, \frac{ds}{t^{1/2}} \\
& \leq C A r^k \int_{\rho}^R \, e^{-(r-s)^2 / 4t} \frac{r^{2-k} s^{3 + k} }{ \left( rs + t \right)^{5/2} } \, \frac{ds}{t^{1/2}}.
\end{align*}
Applying Lemma \ref{holderlemma}$a$ yields the claim.
\end{proof}

\begin{prop}\label{initialprop} Let
$$\rho \leq \bar{r} \leq R, \qquad  0 \leq t_1 \leq t_2 \leq \infty, \qquad -4 \leq m < 4.$$
The solution $u_\varphi$ of (\ref{initialprob}-\ref{initialboundaryprob}) satisfies
\begin{equation*}
\begin{split}
& \int_{t_1}^{t_2} \!\!\!\! \int_{\rho}^{\bar{r}} u_\varphi^2 \, r^m dr dt \\
& \quad \leq C_{m}
 \left( w^{m + 4}(\bar{r}, t_1) \int_{\rho}^{\bar{r}} \, s^{m+2} + \bar{r}^{m+2} \int_{\bar{r}}^{R} w^{5}(s,t_1) \left( \dfrac{\bar{r}^2}{s^2 + t_1} \right)^{3/2} \right) \varphi^2(s) \, ds
\end{split}
\end{equation*}
\end{prop}
\begin{proof} 
We apply H\"older's inequality
\begin{align*}\label{initialprop1}
u_\varphi (r,t)^2 & \leq \left( C\int_{\rho}^{R}  |\varphi(s)| \, e^{-(r-s)^2/4t} \frac{r^2 s^3}{(rs + t)^{5/2}} \frac{ds}{t^{1/2}} \right)^2 \\
& \leq C \int_{\rho}^{R} \varphi^2(s) \, e^{-(r-s)^2/4t} \frac{r^4 s^6}{(rs + t)^5} \frac{ds}{t^{1/2}} \cdot \int_{\rho}^{R} e^{-(r-s)^2/4t} \frac{ds}{t^{1/2}} \\
& \leq C \int_{\rho}^{R} \varphi^2(s) \, e^{-(r-s)^2/4t} \frac{r^4 s^6}{(rs + t)^{5}} \frac{ds}{t^{1/2}}.
\end{align*}
Then compute
\begin{equation*}
\begin{split}
\int_{t_1}^{t_2} \!\!\!\! \int_{\rho}^{\bar{r}} u_\varphi^2(r,t) & \, r^m dr dt  \leq C \int_{\rho}^{R} \varphi^2(s) \LB \int_{t_1}^{t_2} \!\!\!\! \int_{\rho}^{\bar{r}} e^{-(r-s)^2 / 4t} \frac{s^{4 - m} r^{4+m} }{(rs + t)^5} \frac{dr}{t^{1/2}} \, dt \RB \, s^{m + 2} ds.
\end{split}
\end{equation*}
Applying Lemma \ref{holderlemma}$b$ yields the claim.
\end{proof}

\subsection{Inner boundary data}\label{innerappendixsection} 

Next, we construct and estimate a solution of the boundary-value problem 
\begin{equation}\label{insideq}
\square u_\psi(r,t) = 0
\end{equation}
with
\begin{equation}\label{insideboundconds}
\begin{split}
u_\psi(\rho,t) = \psi(t), \qquad u_\psi(1,t) = 0  \qquad (t_0 \leq t \leq t_2) \\
u_\psi(r,t_0) = 0 \qquad (\rho < r \leq 1).
\end{split}
\end{equation}

Denote the complementary error function
$$\mbox{erfc}(x) = \frac{2}{\sqrt{\pi}} \int_x^\infty e^{-\xi^2} d\xi.$$
This is decreasing and satisfies (see \cite{carslaw}, appendix)
\begin{equation}\label{erfcbound} \max\LB 1-\frac{2}{\sqrt{\pi}} x, \frac{e^{-x^2}}{\sqrt{\pi} x} \left( 1 - \frac{1}{2x^2}\right) \RB \leq \mbox{erfc}(x) \leq \min \LB 1, \frac{e^{-x^2}}{\sqrt{\pi} x} \RB \quad \left( x \geq 0 \right).
\end{equation}
For $r \geq 1,$ let
$$\overline{u}_1(r,t) = \text{erfc}\left( \dfrac{r-1}{2\sqrt{t}} \right).$$
Then $0 \leq \bar{u}_1 \leq 1$ and $u_1$ is smooth away from $(r,t) = (1,0),$ with
\begin{equation*}
\begin{split}
\left( \partial_t - \partial_r^2 \right) \overline{u}_1(r,t) & = 0 \qquad (r > 1 \text{ or } t > 0) \\
\overline{u}_1(1,t) = 1, \quad \quad \quad & \overline{u}_1(r,0)  = 0 \qquad \left(r > 1 \right).
\end{split}
\end{equation*}

Note that
$$\left( \partial_t - \left( \partial_r^2 + \frac{n - 1}{r} \partial_r \right) \right) \overline{u}_1 = -\frac{n-1}{r} \partial_r \overline{u}_1 \geq 0$$
hence $\overline{u}_1(r,t)$ is a nonnegative supersolution for the heat equation in $\R^n \setminus B_1(0).$ But for $n \geq 3$ and $R\geq 2,$ let
$$h_R(r) = \frac{ r^{2-n} - R^{2-n}}{1-R^{2-n}} \quad \quad \underline{u}_1(r,t) = h_R(r) \overline{u}_1(r,t).$$
Then
\begin{align*}
\left( \partial_t - \left( \partial_r^2 + \frac{n-1}{r} \partial_r \right) \right) \underline{u}_1 & = \frac{ - 2 \partial_r r^{2-n} \partial_r \overline{u}_1 - \left( r^{2-n} - R^{2-n} \right) \frac{n-1}{r} \partial_r \overline{u}_1 }{1-R^{2-n}} \\
& = \frac{(n-3)r^{1-n} + (n-1)R^{2-n}r^{-1}}{1-R^{2-n}} \partial_r \overline{u}_1 \leq 0.
\end{align*}
Hence $\underline{u}_1(r,t)$ is a subsolution with
$$\underline{u}_1(1,t) = 1, \qquad \underline{u}_1(R,t) =0 $$
on $U^n(1,R).$ Now, let
$$u_1 = u_1^{n,R} = h_R - H^n_{(1,R)} \ast h_R$$
be the solution of the heat equation on $U^n(1,R) \times \LB 0, \infty \right) $ with
$$ u_1(1,t) = 1, \quad \quad u_1(R,t) = 0 \qquad (0 < t < \infty) $$
$$u_1(r,0) \equiv 0 \qquad  (1 < r \leq R).$$

\begin{lemma}\label{u1boundlemma} The solution $u_1$ satisfies
\begin{equation}\tag{$a$}\label{u1bound}
\begin{split}
h_R(r) \cdot \max \LB 1 - \frac{r-1}{\sqrt{\pi \, t}} , \frac{2 \sqrt{t}}{\sqrt{\pi}(r-1)} \right. & \left. \left( 1 - \frac{2t}{(r-1)^2} \right) e^{-(r-1)^2 / 4t } \RB \\
& \leq u_1(r,t) \leq \min \LB h_R(r) , \frac{ 2 \sqrt{t}}{\sqrt{\pi} (r-1)} e^{-(r-1)^2 / 4t} \RB
\end{split}
\end{equation}
\begin{equation}\tag{$b$}\label{u1tbound}
0 \leq \partial_t u_1(r,t) \leq \frac{C_n e^{-(r-1)^2/5t}}{t(t+1)^{n/2-1}} \cdot \begin{cases} \min \LB (r - 1)/\sqrt{t}, 1 \RB & ( t \leq 1) \\
\min \LB r - 1, 1 \RB & (t \geq 1). \end{cases}
\end{equation}
\end{lemma}
\begin{proof}
The maximum principle implies
$$\underline{u}_1 \leq u_1 \leq \min \LB h_R, \overline{u}_1 \RB$$
which by (\ref{erfcbound}) yields (\ref{u1bound}).

For any $0 < t \leq 2, 1 \leq r \leq R - 1$ and $1/4 \leq \eta \leq 1,$ let
$$V_{r,t,\eta} = \LB r , r+\sqrt{\eta t} \, \RB \times \LB (1-\eta) t, t \RB.$$
We may rescale by a factor of $\sqrt{t/2}$ and apply (\ref{u1bound}) and the derivative estimates of Lieberman \cite{lieberman} Ch. 4, on $V_{r,t,1/2},$ to obtain
\begin{equation}\label{Vest}
\begin{split}
\| D^k u_1 \|_{L^\infty({V_{r,t,1/4}})} & \leq \frac{C}{t^{k/2}} \| u_1 \|_{L^\infty\left( V_{r,t,1/2} \right)} \\
& \leq \frac{C}{t^{k/2}} \min \LB 1, \frac{\sqrt{t}}{r-1} e^{-(r - 1)^2 / 4t} \RB \\
& \leq \frac{C}{t^{k/2}} e^{-(r-1)^2 / 4t}.
\end{split}
\end{equation}
Note that $\partial_t u_1$ satisfies
\begin{equation}\label{u1tevol}
\begin{split}
\left( \partial_t + \Delta \right) \partial_t u_1 & = 0 \\
\partial_t u_1(1,t) = \partial_t u_1(R,t) & = 0 \qquad (t>0).
\end{split}
\end{equation}
From (\ref{Vest}) we may further write
$$| \partial_t u_1(r,t) | = |\Delta u_1(r,t)| \leq \frac{C}{t} e^{-(r - 1)^2 / 4t} \min \LB \frac{r-1}{\sqrt{t}}, 1 \RB$$
for $ 0 < t \leq 2.$
Note that (\ref{erfcbound}) also implies that for each $r,$ $\partial_t u_1(r,t) \geq 0$ for $t$ sufficiently small, but not identically zero; so the maximum principle implies $\partial_t u_1(r,t) \geq 0$ for all $t>0.$
This establishes (\ref{u1tbound}) for $0 < t \leq 2.$

Returning to (\ref{Vest}) with $t=1,$ we have
$$0 \leq \partial_t u_1(r,1) \leq C e^{-(r - 1)^2 / 4}$$
hence for $r,t\geq 1,$ by (\ref{u1tevol})
\begin{align*}
\partial_t u_1(r,t) & = \int_1^R H_{\left( 1, R \right)}^n(r, s, t-1) \, \partial_t u_1(s,1) \, dV_s \leq C \int_1^\infty  \frac{e^{-(r-s)^2/4(t-1) - (s - 1)^2 / 4}\cdot s^{n-1}}{\left(rs + t - 1\right)^{(n-1)/2} \left(t-1 \right)^{1/2}} ds \\
& \quad \quad \quad \quad \quad \quad \quad \quad \quad \quad \quad \quad \leq \frac{C}{(t-1)^{n/2}} \int_1^\infty e^{-(r-s)^2/4(t-1) - (s - 1)^2 / 4} \, s^{n-1} ds.
\end{align*}
Note that
\begin{align*}
\frac{(r-s)^2}{t-1} + (s-1)^2 = \frac{1}{t} \left( \frac{t}{t-1} \left( r-s\right)^2 + t (s-1)^2 \right) \geq \frac{1}{t} \left( r - s + s -1 \right)^2 = \frac{(r-1)^2}{t}
\end{align*}
and
\begin{align*}
\frac{(r-s)^2}{t-1} + (s-1)^2 
& \geq \frac{4(r-1)^2}{5t} + \frac{(s-1)^2}{5}.
\end{align*}
Hence for $t\geq 2$ we have
\begin{equation*}
\partial_t u_1(r,t) \leq C \, \frac{e^{-(r-1)^2/5t} }{(t-1)^{n/2}} \int_1^R e^{-(s-1)^2/20} \, s^{n-1} ds \leq \frac{C}{t^{n/2}} e^{-(r-1)^2/5t}.
\end{equation*}
The extra factor of $r-1$ for $1 \leq r \leq 2$ may again be obtained from (\ref{u1tevol}) and the boundary estimates of \cite{lieberman}, Ch. 4.
\end{proof}

Now, put $n=6$ and $\rho = 1/R.$ We let
\begin{equation*}
\begin{split}
G_\rho(r,t) & = \frac{r^2}{\rho^2}\partial_t \left( u_1 \left(r / \rho, t / \rho^2 \right) \right) = \frac{r^2}{\rho^4} \left( \partial_t u_1 \right) \left(r / \rho, t / \rho^2 \right).
\end{split}
\end{equation*}
From Lemma \ref{u1boundlemma}\ref{u1tbound}, we have the bound
\begin{equation*}
\quad 0 \leq G_\rho(r,t) \leq \frac{C \rho^2 r^2 }{t(t+\rho^2)^2} e^{-(r-\rho)^2/5t} \cdot \begin{cases} \min \LB (r - \rho)/\sqrt{t}, 1 \RB & ( t \leq \rho^2) \\
\min \LB (r - \rho)/\rho, 1 \RB & (t \geq \rho^2). \end{cases}
\end{equation*}
Moreover, by Duhamel's principle \cite{carslaw}, given $\psi(t) \in C^0(\LB t_0 , t_2 \RB),$ the function
\begin{equation}\label{upsidef}
u_\psi(r,t) = \int_{t_0}^t \psi(\tau) \, G_\rho(r, t - \tau) \, d\tau
\end{equation}
on $U^4(\rho,1)$ solves (\ref{insideq} - \ref{insideboundconds}).

\begin{prop}\label{insideprop} Let
$$\frac{3\rho}{2} \leq r_1 \leq r \leq r_2 \leq 1, \qquad t_0 \leq t_1 \leq t \leq t_2, \qquad \bar{r}_2 = \min \LB r_2, \sqrt{t_2 - t_1 } \RB$$
$$C_{\ref{insideprop}} = C e^{ - \left( r_1 - \rho \right)^2/5\left( t_2 - t_0 \right)}.$$
The solution $u_\psi$ of (\ref{insideq}-\ref{insideboundconds}) satisfies
\begin{align}\tag{$a$}
\quad \quad \quad \quad \quad \quad \left| u_\psi (r, t) \right| & \leq 
\, C_{\ref{insideprop}} \left( \frac{\rho}{r} \right)^2 \sup_{t_0 \leq \tau \leq t} |  \psi(\tau) | \\ 
\tag{$b$}
r^6 u_\psi^2(r, t) & \leq 
\, C_{\ref{insideprop}} \, \rho^4 \int_{t_0}^t \psi^2(\tau) \, d\tau \\ 
\tag{$c$}
\int_{r_1}^{1} u_\psi^2(r, t) \, r^2 dV & \leq 
\, C_{\ref{insideprop}} \, \rho^4 \int_{t_0}^t \psi^2(\tau) \, d\tau \\ 
\tag{$d$}
\quad \quad \int_{t_1}^{t_2} \!\!\!\! \int_{r_1}^{r_2} u_\psi^2(r,t) \, dV dt & \leq 
\, C_{\ref{insideprop}} \, \rho^4 \left( \int_{t_0}^{t_1} + \log \left( 1 + \frac{ \bar{r}_2 }{ r_1 } \right) \int_{t_1}^{t_2} \right) \psi^2(\tau) \, d\tau.
\end{align}
\end{prop}
\begin{proof} Let $u = u_\psi.$ From (\ref{upsidef}) we have
$$ |u_\psi(r,t)| \leq C \frac{\rho^2}{r^2} \int_{t_0}^t | \psi(\tau) | \frac{r^4 e^{-(r-\rho)^2/5(t - \tau)}}{(t-\tau+\rho^2)^2} \, \frac{d\tau}{t-\tau}$$
and ($a$) follows by removing $\sup \psi$ and changing variables.

To prove ($b$), for $0 \leq k \leq 3,$ H\"older's inequality may be applied:
\begin{align}\label{insideholder}
\nonumber & r^{2k} \left( \int_{t_0}^t \psi(\tau) \frac{r^2 e^{-(r-\rho)^2/5(t-\tau)}}{(t-\tau)(t-\tau+\rho^2)^2} d\tau \right)^2 \\
& \quad \leq \int_{t_0}^t \psi^2(\tau) \, \frac{e^{-(r-\rho)^2/5(t-\tau)}}{\left(t-\tau + \rho^2\right)^{3-k}} \, d \tau \cdot \int_{t_0}^t \left(\frac{r^{2 + k} \left( r - \rho\right)}{\left(t - \tau\right)\left( t - \tau + \rho^2 \right)^{\frac{1+k}{2}}}\right)^2 \frac{e^{-(r-\rho)^2/5(t-\tau)} d\tau}{(r-\rho)^2} \\
\nonumber & \quad \leq C \int_{t_0}^t \psi^2(\tau) \, \frac{e^{-(r-\rho)^2/5(t-\tau)}}{\left( t - \tau + \rho^2 \right)^{3-k}} \, d \tau.
\end{align}
We have used the assumption $r \geq r_1 \geq \frac{3}{2} \rho.$ Applying (\ref{insideholder}) with $k=3$ gives ($b$).

Applying (\ref{insideholder}) with $k=2,$ we have
\begin{align}\label{utimeboundcalc}
\nonumber \int_{r_1}^{r_2} u^2(r,t) \, r^2 dV & \leq C \rho^4 \int_{r_1}^{r_2} \left( \int_{t_0}^t \psi(\tau) \frac{r^2 e^{-(r-\rho)^2/5(t-\tau)}}{(t-\tau)(t-\tau+\rho^2)^2} \, d\tau \right)^2 \, r^5 \, dr \\
& \leq C \rho^4 \int_{r_1}^{r_2} \!\!\!\! \int_{t_0}^t \psi^2(\tau) \, \frac{e^{-(r-\rho)^2/5(t-\tau)}}{t - \tau + \rho^2} r d \tau dr \\
\nonumber & \leq C \rho^4 \int_{t_0}^t \psi^2(\tau) \left( \int_{r_1}^{r_2} \frac{e^{-(r-\rho)^2 / 5(t-\tau)}}{t-\tau + \rho^2} \, r dr \right) \, d\tau \\
\nonumber & \leq C \rho^4 e^{-(r_1 - \rho)^2 / 5(t_2 - t_1)} \int_{t_0}^t \psi^2(\tau) \, d\tau
\end{align}
which is ($c$).

To prove ($d$), first assume $t_0 = t_1.$ Applying (\ref{insideholder}) with $k=1,$ we obtain
\begin{align}\label{insideproplast}
\nonumber \int_{t_1}^{t_2} \!\!\!\! \int_{r_1}^{r_2} u^2 \, dV dt & \leq C \rho^4 \int_{t_1}^{t_2} \!\!\!\! \int_{r_1}^{r_2} \left( \int_{t_1}^t \psi(\tau) \frac{r^2 e^{-(r-\rho)^2/5(t-\tau)}}{(t-\tau)(t-\tau+\rho^2)^2} \, d\tau \right)^2 \, r^3 dr dt \\
& \leq C \rho^4 \int_{t_1}^{t_2} \!\!\!\! \int_{r_1}^{r_2} \!\!\!\! \int_{t_1}^t \psi^2(\tau) \, \frac{e^{-(r-\rho)^2/5(t-\tau)}}{\left(t-\tau + \rho^2\right)^2} \,r d\tau dr dt \\
\nonumber & \leq C \rho^4 \int_{t_1}^{t_2} \!\!\!\! \int_{t_1}^t \psi^2(\tau) \left( \int_{r_1}^{r_2} \frac{e^{-(r-\rho)^2 / 5(t-\tau)}}{(t-\tau + \rho^2)^2 } \, r dr \right) \, d\tau  dt.
\end{align}
Letting $\bar{\tau} = t - \tau,$ the domain of integration
$$t_1 \leq t \leq t_2, \quad \quad t_1 \leq \tau \leq t$$
becomes
$$ t_1 \leq \bar{\tau} + \tau \leq t_2, \quad \quad t_1 \leq \tau, \quad \quad 0 \leq \bar{\tau}$$
which we relax to
$$t_1 \leq \tau \leq t_2, \quad \quad 0 \leq \bar{\tau} \leq t_2 - \tau.$$
Hence (\ref{insideproplast}) may be rewritten
\begin{equation}\label{priortodest}
\int_{t_1}^{t_2} \!\!\!\! \int_{r_1}^{r_2} u^2 \, dV dt \leq C \rho^4 \int_{t_1}^{t_2} \psi^2(\tau) \, A(t_2 - \tau) \, d\tau
\end{equation}
where
\begin{align}\label{Anastyest}
\nonumber A (t - \tau) = \int_{r_1}^{r_2} \!\!\!\! \int_0^{t - \tau} e^{-(r-\rho)^2/5\bar{\tau}} \frac{d\bar{\tau}}{\left(\bar{\tau} +\rho^2 \right)^2 } \, r dr & \leq C  \int_{r_1}^{r_2} \frac{ e^{-(r-\rho)^2 / 5 \left( t - \tau \right)}}{(r-\rho)^2} \, r dr \\
& \leq C \int_{r_1}^{r_2} e^{-(r-\rho)^2 / 5 \left( t_2 - t_1 \right)} \, \frac{dr}{r} \\
\nonumber & \leq C e^{ - \left(r_1 - \rho \right)^2/5\left( t_2 - t_1 \right)} \log \left( 1 + \frac{\bar{r}_2}{r_1} \right).
\end{align}
We have used $ \tau \geq t_1$ and $r \geq r_1 \geq \frac{3}{2} \rho.$ The statement ($d$) follows by applying (\ref{priortodest}) over $\LB t_1, t_2 \RB,$ together with ($c$) over $\LB t_0, t_1 \RB$ and Proposition \ref{initialprop} at time $t_1.$
\end{proof}


\subsection{Outer boundary data} 

Next, we wish to solve
\begin{equation}\label{outsideq}
\square u_\xi(r,t) = 0
\end{equation}
with
\begin{align}\label{outsidedata}
u_\xi(\rho,t) = 0, \qquad u_\xi(1,t) = \xi(t) \qquad (t_0 \leq t \leq t_2) \\
\nonumber u_\xi(r,t_0) = 0 \qquad (\rho < r \leq 1).
\end{align}

Denote the error function
$$\mbox{erf}(x) = \frac{2}{\sqrt{\pi}} \int^x_{-\infty} e^{-\xi^2} d\xi$$
which is increasing and 
satisfies
\begin{equation}\label{erfbound} \max\LB 1-\frac{2}{\sqrt{\pi}} |x|, \frac{e^{-x^2}}{\sqrt{\pi} |x|} \left(1 - \frac{1}{2x^2} \right) \RB \leq \mbox{erf}(x) \leq \min \LB 1, \frac{e^{-x^2}}{\sqrt{\pi} |x|} \RB \quad \left( x \leq 0 \right).
\end{equation}
For $r \leq 1,$ let
$$v_0(r,t) = \, \text{erf}\left( \dfrac{r-1}{2\sqrt{t}} \right)$$
which satisfies
\begin{equation*}
\begin{split}
\left( \partial_t - \partial_r^2 \right) v_0(r,t) & = 0 \qquad (r < 1 \text{ or } t > 0) \\
v_0(1,t) = 1, & \qquad v_0(r,0) = 0 \qquad \left(r < 1 \right). \\
\end{split}
\end{equation*}
For $n \geq 3$ and $\rho \leq 1/2,$ also let
$$h_{\rho}(r) =  \frac{ \rho^{2-n} - r^{2-n} }{ \rho^{2-n} -1 }, \quad \quad \quad \underline{v}_1 = h_\rho v_0.$$
Then
\begin{equation*}
\begin{split}
\left( \partial_t - \left( \partial_r^2 + \frac{n-1}{r} \partial_r \right) \right) \underline{v}_1 & = \frac{ 2 \partial_r r^{2-n} \partial_r v_0 - \left( \rho^{2-n} - r^{2-n} \right) \frac{n-1}{r} \partial_r v_0 }{\rho^{2-n} - 1} \\
& = \frac{(3-n) r^{1-n} + (1-n)\rho^{2-n}r^{-1}}{\rho^{2-n} -1} \partial_r v_0 \leq 0.
\end{split}
\end{equation*}
Hence $\underline{v}_1(r,t)$ is a nonnegative subsolution for the heat equation in $U^n(\rho,1).$
Let
$$\overline{v}_1(r,t) = r^{2-n} v_0(r,t).$$
Then
\begin{equation*}
\begin{split}
\left( \partial_t - \left( \partial_r^2 + \frac{n-1}{r} \partial_r \right) \right) \overline{v}_1 & = \left(- 2 \partial_r r^{2-n} - r^{2-n} \frac{n-1}{r} \right) \partial_r v_0 \\
& = \left( n - 3 \right) r^{1-n} \partial_r v_0 \geq 0.
\end{split}
\end{equation*}
Hence $\overline{v}_1(r,t)$ is a supersolution. Now, let
$$v_1 = v_1^{n,\rho} = h_\rho - H^n_{\left( \rho, 1 \right)} \ast h_\rho$$
be the solution of the heat equation on $U^n(\rho,1) \times \LB 0, \infty \right) $ with
$$v_1(\rho,t) = 0, \qquad v_1(1,t) = 1 \qquad (0 < t < \infty) $$
$$v_1(r,0) \equiv 0 \qquad  (\rho \leq r < 1).$$

\begin{lemma}\label{vlemma} For $n \geq 3,$ the solution $v_1$ satisfies
\begin{equation}\tag{$a$}\label{v1bound}
\begin{split}
h_\rho(r) \cdot \max \LB 1 - \frac{1-r}{\sqrt{\pi \, t}} , \frac{2 \sqrt{t}}{\sqrt{\pi}(1-r)} \right. & \left. \left( 1 - \frac{2t}{(r-1)^2} \right) e^{-(r-1)^2 / 4t } \RB \\
& \leq v_1(r,t) \leq \min \LB h_\rho(r) , \frac{ 2 \sqrt{t} \, r^{2-n}}{\sqrt{\pi} (1-r)} e^{-(r-1)^2 / 4t} \RB
\end{split}
\end{equation}
\begin{equation}\tag{$b$}\label{v1tbound1}
0 \leq \partial_t v_1(r,t) \leq C_n e^{- \left( t + t^{-1} \right) / C_n} \quad \quad \left(t > 0, \quad \rho \leq r \leq 1/2 \right)
\end{equation}
\begin{equation}\tag{$c$}\label{v1tbound2}
0 \leq \partial_t v_1(r,t) \leq C_n \begin{cases} \dfrac{e^{-{(r-1)^2/Ct}} }{t}\min \LB \dfrac{1 - r}{\sqrt{t}}, 1 \RB & \left( 0 < t \leq 1, \quad 1/2 \leq r \leq 1 \right) \\
(1-r) \, e^{-t/C_n} & \left(t \geq 1, \quad 1/2 \leq r \leq 1 \right). \end{cases}
\end{equation}

\end{lemma}
\begin{proof}
The maximum principle implies
$$\underline{v}_1 \leq v_1 \leq \min \LB h_\rho, \overline{v}_1 \RB$$
which by (\ref{erfbound}) yields (\ref{v1bound}).

Arguing as in the proof of Lemma \ref{u1boundlemma}, for any $0 < t \leq 1$ and $1/2 \leq r \leq 1,$ we have
\begin{equation*}
\begin{split}
\left| D^k v_1 (r,t) \right|
& \leq \frac{C}{t^{k/2}} \min \LB 1, \frac{\sqrt{t}}{r-1} e^{-(r - 1)^2 / 4t} \RB \leq \frac{C}{t^{k/2}} e^{-(r-1)^2 / 5t}
\end{split}
\end{equation*}
and
\begin{equation}\label{v1thalfbound}
0 \leq \partial_t v_1(r,t) \leq \frac{C}{t} e^{-(r - 1)^2 / 5t} \min \LB \frac{|r-1|}{\sqrt{t}}, \quad 1 \RB \quad \quad (0 < t \leq 1, \quad 1/2 \leq r \leq 1).
\end{equation}
For $0 < t \leq \rho^2,$ we may apply the same derivative estimate uniformly for any $\rho \leq r \leq 1.$ This yields
\begin{equation*}
0 \leq \partial_t v_1(r,t) \leq \frac{C}{t} r^{2-n} e^{-(r - 1)^2 / 5t} \min \LB \frac{|r-\rho|}{\sqrt{t}}, 1 \RB \quad \quad (0 < t \leq \rho^2, \quad \rho \leq r \leq 1).
\end{equation*}
In particular, we have
\begin{equation}\label{v1tinitialbound}
\begin{split}
\partial_t v_1(r, \rho^2) & \leq C \rho^{-2} r^{2-n} e^{-(r - 1)^2 / 5 \rho^2} \\
& \leq C e^{-(r-1)^2 / 6 \rho^2}  \quad (\rho \leq r \leq 1/2).
\end{split}
\end{equation}

Based on (\ref{v1thalfbound}) and (\ref{v1tinitialbound}) with $r = 1/2,$ we may again use the maximum principle. This implies, for $t \geq \rho^2$ and $\rho \leq r \leq 1/2,$ the bound
\begin{equation}\label{1overtbound}
\begin{split}
0 \leq \partial_t v_1(r, t) 
& \leq C e^{-1/24t} \qquad \left( t \geq \rho^2, \quad \rho \leq r \leq 1/2 \right).
\end{split}
\end{equation}
Lastly, since $\partial_t v_1(r, 1)$ is a subsolution of the heat equation on $B_1$ for $t \geq 1,$ with $\partial_t v_1(r, t) \leq C,$ we have
\begin{equation}\label{explvbound}
\partial_t v_1(r,t) \leq C e^{-\lambda_1 t} (1-r) \quad \left( t \geq 1 \right)
\end{equation}
where $\lambda_1$ is the first Dirichlet eigenvalue of $B_1.$ 
Combining (\ref{v1thalfbound}), (\ref{1overtbound}), and (\ref{explvbound}) yields the desired estimates (\ref{v1tbound1}) and (\ref{v1tbound2}).
\end{proof}

Now, put $n=6$ and let
\begin{equation*}
\begin{split}
K_\rho(r,t) & = r^2 \partial_t v_1 \left(r, t \right).
\end{split}
\end{equation*}
Given $\xi(t) \in C^0(\LB t_0 , t_2 \RB),$ the function
\begin{equation*}
u_\xi (r,t) = \int_{t_0}^t \xi(\tau) \, K_\rho(r, t - \tau) \, d\tau
\end{equation*}
on $U^4(\rho,1)$ solves (\ref{outsideq} - \ref{outsidedata}).

\begin{prop}\label{outsideprop} Let
$$\rho \leq r_1 \leq r \leq r_2 \leq \frac{3}{4}, \qquad t_0 \leq t \leq t_2.$$
The solution $u_\xi$ of (\ref{outsideq}-\ref{outsidedata}) satisfies
\begin{equation}\tag{$a$}
\begin{split}
\quad \quad u_\xi^2(r,t) \leq 
 C r ^4 e^{-\frac{c}{t - t_0} } \int_{t_0}^{t}  \xi^2(\tau) \, e^{- c \left( t - \tau + \frac{1}{t - \tau } \right)} \, d\tau
\end{split}
\end{equation}
\begin{equation}\tag{$b$}
\begin{split}
\int_{t_0}^{t_2} \!\!\!\! \int_{r_1}^{r_2} u_\xi^2(r,t) \, dV dt \leq 
\, C r_2 ^{8} e^{- \frac{c}{ t_2 - t_0} }  \int_{t_0}^{t_2}  \xi^2(\tau) \, e^{- \frac{c}{ t_2 - \tau } } \, d\tau. \quad \quad \quad \quad
\end{split}
\end{equation}
\end{prop}
\begin{proof} The proof is a simple calculation based on Lemma \ref{vlemma}$b$, which we omit. 
\end{proof}

\subsection{Inhomogeneous term}\label{inhomogeneousappendix} 
Finally, let
\begin{align}\label{inhomogeneous}
\nonumber u_\eta(r,t) & = \int_{t_0}^t \!\! \int_{\rho}^R \eta(s,\tau) \, r^2 H_{(\rho, R)}^6(r, s, t-\tau) \, d V_s d\tau \\
& \leq \int_{t_0}^t \!\! \int_\rho^R \left| \eta(s, \tau) \right| e^{-(r-s)^2 / 4(t - \tau)} \frac{r^2 s^3}{\left( rs + t - \tau \right)^{5/2} } \, \frac{ds}{(t - \tau)^{1/2}} \, d\tau.
\end{align}
For $t \geq t_0,$ this satisfies
$$\square u_\eta(r,t) = \eta (r,t)$$
with zero initial and boundary values.

\begin{prop}\label{inhomogeneoushardprop} Let $u_\eta$ be given by (\ref{inhomogeneous}), with $t_0 = 0,$ and put
$$0 \leq \alpha < 4, \quad \quad \alpha \neq 1$$
$$\bar{\alpha} = \min \LB \alpha, 1 \RB, \quad \quad 0 \leq \beta < 4 - \alpha.$$
Assume
$$\left| \eta(r,t) \right| \leq A r^{k-2} w^{\alpha}(r,t) w^\beta(R,t) \qquad \left(\forall \,\, \rho \leq r \leq R, \quad 0 \leq t < t_2 \right).$$
Then
\begin{equation*}
\left| u_\eta (r, t) \right| \leq C_{k,\alpha,\beta} A w^{\beta}(R,t) \begin{cases} \dfrac{\rho^{k + 2 + \bar{\alpha}}}{r^{2 + \bar{\alpha}}} w^\alpha(r,t) & (k < - 2 - \bar{\alpha} ) \\
r^k w^\alpha(r,t) & ( - 2 - \bar{\alpha} < k < 2 - \alpha ) \\
r^k w^{2-k}(r,t) & ( 2 - \alpha < k < 2) \\
r^2 R^{k-2} w^{\alpha}(R,t) & (k > 2).
\end{cases}
\end{equation*}
\end{prop}
\begin{proof}
We first assume $\beta = 0.$
Let $N = \lceil - \log_2 \rho \rceil, s_i = 2^i \rho$ for $0 \leq i < N,$ and $s_N = 1.$ Write
$$\eta = \sum_{i=0}^N \eta_i, \quad \quad \text{supp } \eta_i \subset \LB s_i, s_{i+1} \RB$$
and solve
$$\square u_i = \eta_i$$
using (\ref{inhomogeneous}), so that $u = \sum u_i.$

Letting $A_i = A (2s_i)^k,$ the assumption becomes
\begin{equation*}
|\eta_i (s,t)| \leq \frac{A_i}{s_i^2}w^\alpha(s_i, t).
\end{equation*}
To begin, we fix $r$ and estimate $u_i(r,t).$ For $s_i < r/2,$ further write
$$\eta_i = \eta^0_i + \eta^1_i, \quad \quad \square u_i^0 = \eta_i^0, \quad \quad \square u_i^1 = \eta^1_i$$
with
$$\text{supp } \eta^0_i \subset \LB s_i, s_{i+1} \RB \times \LB 0 , r^2 \RB, \quad \quad \text{supp } \eta^1_i \subset \LB s_i, s_{i+1} \RB \times \LB r^2, \infty \RB.$$
Letting $\bar{r} = \min \LB r, \sqrt{t} \RB,$ we have
\begin{align}\label{inhomogeneoushard1}
| u_i^0(r,t) | & \leq C A_i \int_0^{\bar{r}^2} \!\!\!\!\! \int_{s_i}^{s_{i+1}} \frac{e^{-(r-s)^2/4(t - \tau)} r^2 s }{(rs + t - \tau)^{5/2}} w^{\alpha}(s,\tau) \, \frac{ds}{(t-\tau)^{1/2}} d\tau \\\nonumber
& \leq C A_i \left( \int_0^{\bar{r}^2} \!\!\!\!\! \int_{s_i}^{s_{i+1}} \frac{e^{-(r-s)^2/4(t - \tau)} r^4 s^2 }{(rs + t - \tau)^{5}} \, \frac{ds}{(t-\tau)^{1/2}} d\tau \right)^{1/2} \\\nonumber
& \quad \quad \quad \quad \quad \cdot \left( \int_0^{\bar{r}^2} \!\!\!\!\! \int_{s_i}^{s_{i+1}} w^{2\alpha}(s_i, \tau) \frac{ e^{-(r-s)^2/4(t - \tau)} }{(t- \tau)^{1/2}} \, ds d\tau \right)^{1/2}. \nonumber
\end{align}

For the first integral, we change variables $\bar{\tau} = t - \tau,$ to obtain
\begin{align*}
\int_0^{\bar{r}^2} \!\!\!\!\! \int_{s_i}^{s_{i+1}} \frac{e^{-(r-s)^2/4(t - \tau)} r^4 s^2 }{(rs + t - \tau)^{5}} \, \frac{ds}{(t-\tau)^{1/2}} d\tau & = \int_{t-\bar{r}^2}^t \! \int_{s_i}^{s_{i+1}} \frac{e^{-(r-s)^2/4\bar{\tau}} r^4 s^2 }{(rs + \bar{\tau})^{5}} \, \frac{ds}{\bar{\tau}^{1/2}} d\bar{\tau} \\
& \leq \frac{C}{r^2 + t - \bar{r}^2} \frac{\bar{r}^2}{r^2 + t} w^4(r , t - \bar{r}^2) \left( \frac{s_i^2}{r^2 + t - \bar{r}^2} \right)^{3/2} \\
& \leq C \frac{r^6 s_i^3}{\left( r^2 + t \right)^{11/2}}.
\end{align*}
We have used Lemma \ref{holderlemma}$b$ while noting that $r^2 + t - \bar{r}^2 \geq \frac{1}{2} (r^2 + t).$
For the second integral of (\ref{inhomogeneoushard1}), note for $r \geq 3 s_i$ that
\begin{equation}\label{eandtbound}
\frac{e^{- (r-s)^2/4(t-\tau)}}{(t-\tau)^{1/2}} \leq \frac{e^{- r^2/36(t-\tau)} }{r} \frac{r}{(t-\tau)^{1/2}} \leq \frac{C}{r}.
\end{equation}
By (\ref{wintegral}), we have
\begin{equation*}
\begin{split}
\int_0^{\bar{r}^2} \!\!\!\!\! \int_{s_i}^{s_{i+1}} w^{2\alpha}(s_i, \tau) \frac{ e^{-(r-s)^2/4(t - \tau)} }{(t- \tau)^{1/2}} \, ds d\tau & \leq C \frac{s_i}{r} \int_0^{\bar{r}^2} w^{2\alpha}(s_i, \tau) \, d\tau \\
& \leq C_\alpha
\begin{cases} \frac{s_i^3}{r} & (\alpha > 1) \\
\frac{s_i}{r} s_i^{2\alpha} \left( r^2 \right)^{1 - \alpha} & (0 \leq \alpha < 1)
\end{cases}
\end{split}
\end{equation*}
\begin{equation*}
\quad \quad \quad \quad \quad \quad \quad \quad \quad \quad \leq C_\alpha r^2 \left( \frac{s_i}{r} \right)^{1 + 2\bar{\alpha}}.
\end{equation*}
Returning to (\ref{inhomogeneoushard1}), we conclude
\begin{equation}\label{inhomogeneoushard1.5}
|u_i^0(r,t)| \leq C A_i \left( \frac{r^8 s_i^3}{(r^2 + t)^{11/2}} \left( \frac{s_i}{r} \right)^{1 + 2\bar{\alpha}} \right)^{1/2} \leq C A_i \left( \frac{s_i}{r} \right)^{2 + \bar{\alpha}} w^{11/2}(r,t).
\end{equation}

Next, note that $u^1_i(r,t) = 0$ for $t < r^2,$ while for $t \geq r^2$ we have
\begin{align}\label{inhomogeneoushard2}
|u^1_i(r,t)| & \leq C A_i \int_{\bar{r}^2}^t \! \int_{s_i}^{s_{i+1}} \frac{e^{-(r-s)^2/4(t - \tau)} r^2 s }{(rs + t - \tau)^{5/2}} \, \frac{ds}{(t-\tau)^{1/2}} w^{\alpha}(s_i ,\tau)d\tau.
\end{align}
Lemma \ref{holderlemma}$a,$ for $r \geq 3 s_i,$ gives
\begin{align*}
\int_{s_i}^{s_{i+1}} \frac{e^{-(r-s)^2/4(t - \tau) } r^2 s}{ \left( rs + t - \tau \right)^{5/2}} \frac{ds}{(t-\tau)^{1/2}} & \leq C \frac{r^2 s_i}{ \left( r^2+ t - \tau \right)^{5/2}} \frac{s_i}{s_i + \sqrt{t - \tau}} e^{-(r - s_i)^2/5(t - \tau)} \\
& \leq C \frac{r^2 s_i^2}{ \left( r^2+ t - \tau \right)^{3}}.
\end{align*}
Note that for $\tau \geq r^2,$ we have $s_i^2 + \tau \geq \frac{1}{2}(r^2 + \tau)$ and
$$w^\alpha(s_i, \tau) \leq C \frac{s_i^\alpha}{(r^2 + \tau)^{\alpha / 2}}$$
For $0 \leq \tau \leq t,$ also
$$(r^2 + t - \tau)(r^2 + \tau) = (r^2 + t)r^2 - \tau r^2 + r^2 \tau + ( t - \tau) \tau \geq r^2 (r^2 + t).$$
Then (\ref{inhomogeneoushard2}) becomes
\begin{align}\label{inhomogeneoushard3}
|u^1_i(r,t)| & \leq C A_i r^2 s_i^{2 + \alpha} \int_{r^2}^t \frac{d\tau}{(r^2 + t - \tau)^{3}( r^2 + \tau)^{\alpha/2}} \\\nonumber
& \leq C A_i \frac{r^{2} s_i^{2 + \alpha}}{\left( r^2(r^2 + t)\right)^{\alpha/2}} \int_{r^2}^t \frac{d \tau}{(r^2 + t - \tau)^{3 - \alpha / 2}} \\\nonumber
& \leq C_\alpha A_i \frac{r^{2 - \alpha} s_i^{2 + \alpha}}{(r^2 + t)^{\alpha/2}} {r}^{2\left( 1- (3 - \alpha / 2) \right)} \\
& \leq C_\alpha A_i \left( \frac{s_i}{r} \right)^2 \left( \frac{s_i^2}{ r^2 + t } \right)^{\alpha/2} = C_\alpha A_i \left( \frac{s_i}{r} \right)^{2 + \alpha} w^\alpha(r,t).
\end{align}
for $3 - \alpha / 2 > 1,$ or $\alpha < 4.$ We conclude from (\ref{inhomogeneoushard1.5}) and (\ref{inhomogeneoushard3}) that
\begin{equation}\label{inhomogeneoushard4}
\begin{split}
|u_i(r,t)| \leq |u_i^0| + |u_i^1| & \leq C A_i \left( \frac{s_i}{r} \right)^{2 + \bar{\alpha}} \left(w^{11/2}(r,t) + w^{\alpha }(r,t) \right) \\
& \leq C A_i \left( \frac{s_i}{r} \right)^{2 + \bar{\alpha}} w^{\alpha }(r,t)
\end{split} \quad (r \geq 3 s_i).
\end{equation}

For $r \leq 3 s_i$ and $0 \leq \alpha < 4,$ in similar fashion, we have
\begin{align}\label{inhomogeneoushard5}
\nonumber |u_i(r,t)| & \leq C A_i \int_0^{t} \!\!\! \int_{s_i}^{s_{i+1}} \frac{e^{-(r-s)^2/4(t - \tau)} r^2 s }{(rs + t - \tau)^{5/2}} w^{\alpha}(s,\tau) \, \frac{ds}{(t-\tau)^{1/2}} d\tau \\\nonumber
& \leq C A_i \left( \int_0^{s_i^2} + \int_{s_i^2}^t \right) \frac{r^2 s_i^{2 + \alpha}}{(s_i^2 + t - \tau)^3 (s_i^2 + \tau)^{\alpha/2}} \, d\tau \\\nonumber
& \leq C A_i \left( \left(\frac{r}{s_i}\right)^2 w^6(s_i, t) + \frac{r^2 s_i^{2 + \alpha}}{\left( s_i^2(s_i^2 + t) \right)^{\alpha/2}} \int_{s_i^2}^t \frac{d\tau}{(s_i^2 + t - \tau)^{3 - \alpha/2}} \right) \\
& \leq C A_i \left( \left(\frac{r}{s_i}\right)^2 w^6(s_i, t) + \frac{r^2 s_i^{2}}{(s_i^2 + t)^{\alpha/2}} s_i^{\alpha - 4} \right) \leq C A_i \left(\frac{r}{s_i}\right)^2 w^\alpha(s_i, t).
\end{align}

We now let $M = \lceil - \log_2 r \rceil$ and $A_i = A s_i^k.$ Summing (\ref{inhomogeneoushard4}) and (\ref{inhomogeneoushard5}) yields
\begin{align}\label{inhomogeneoushard6}
|u(r,t)| & \leq CA \left( \sum_{i = 0}^{M} \frac{s_i^{k+2+\bar{\alpha}}}{r^{2+\bar{\alpha}}} w^\alpha(r,t) + \sum_{i = M + 1}^N s_i^{k - 2} r^2 w^{\alpha}(s_i, t) \right).
\end{align}
Note for $s_i \geq r$ that
$$w^{\alpha}(s_i , t) \leq \left( \frac{s_i}{r} \right)^\alpha w^\alpha(r,t).$$
For $k < -2 - \bar{\alpha},$ since $k -2 + \alpha$ is negative ($\alpha < 4$), (\ref{inhomogeneoushard6}) reads
\begin{align*}
|u (r,t) | & \leq CA \left( \frac{\rho^{k + 2 + \bar{\alpha}}}{r^{2 + \bar{\alpha}} } w^\alpha(r,t) + \sum_{i = M + 1}^{N} r^{2 - \alpha} s_i^{k - 2 + \alpha} w^\alpha(r,t) \right) \\
& \leq CA \left( \frac{\rho^{k + 2 + \bar{\alpha} }}{r^{2 + \bar{\alpha} } } + r^k \right) w^\alpha(r,t) \leq CA \frac{\rho^{k + 2 + \bar{\alpha} }}{r^{2 + \bar{\alpha} } } w^\alpha(r,t).
\end{align*}
Next, if $k > -2 - \bar{\alpha}$ and $k - 2 + \alpha < 0,$ we have simply
\begin{align*}
|u(r,t)| & \leq CA r^k w^\alpha(r,t).
\end{align*}
For $ -2 - \bar{\alpha} < k < 2$ and $k -2 + \alpha > 0,$ we have
\begin{align*}
|u(r,t)| & \leq CA \left( r^k w^\alpha(r,t) + \sum_{i = M + 1}^N r^2 \frac{s_i^{k - 2 + \alpha}}{ (s_i^2 + t)^{\alpha/2} } \right).
\end{align*}
Notice that the terms of the summation attain a maximum either at $i = \lceil \frac12 \log t \rceil,$ if $r^2 < t,$ or at $i = M + 1,$ if $r^2 \geq t,$ and decay exponentially on either side. We therefore have
\begin{align*}
|u(r,t)|  & \leq CA \left( r^k w^\alpha(r,t) + r^2 ( r^2 + t)^{\frac{k-2}{2}} \right) \\
& \leq CA \left( r^k w^\alpha(r,t) + r^k w^{2-k}(r,t) \right) \leq CA r^k w^{2-k}(r,t).
\end{align*}
For $k > 2,$ we have
\begin{align*}
|u(r,t)| & \leq CA \left( r^k w^\alpha(r,t) + \sum_{i = M + 1}^N r^2 s_i^{k - 2} w^\alpha(s_i, t) \right) \leq C A r^2 R^{k-2} w^\alpha(R, t).
\end{align*}

The case $0 < \beta < 4 - \alpha$ follows by applying the $\beta = 0$ case over time intervals
$$\LB i R^2, (i + 1)R^2 \RB$$
together with Proposition \ref{trivialinitialprop}, and summing. 
\end{proof}

\begin{prop}\label{inhomogeneousprop} Let $u_\eta$ be given by (\ref{inhomogeneous}), and put
$$\rho \leq r_1 \leq r_2 \leq R, \quad \quad t_0 \leq t_1 \leq t_2.$$
Assume
\begin{equation*}
\begin{split} \int_{t_0}^{t_1} \!\!\!\! \int_{r}^{2r} \left( s^2 \eta(s,t) \right)^2 \, dV_s dt \leq B_0 r^m \\
\int_{t_1}^{t_2} \!\!\!\! \int_{r}^{2r} \left( s^2 \eta(s,t) \right)^2 \, dV_s dt \leq B_1 r^m
\end{split} \qquad (\rho \leq r \leq R/2).
\end{equation*}
Let $B = B_0 + B_1,$ and if $m > 0$ choose $0 < \alpha < m.$ Then
\begin{align}
\tag{$a$}
r^6 u_\eta^2(r,t) & \leq C_{m,\alpha} B \begin{cases} \rho^m & (m < 0) \\
r^{m - \alpha} R^\alpha & (0 < m - \alpha < 10) \\
r^{10} R^{m - 10} & (m > 10)
\end{cases} \\
\tag{$b$} \int_{r_1}^{r_2} u_\eta^2(r,t_1) \, r^2 dV & \leq C_{m,\alpha} B_0 \begin{cases}
\rho^m & (m < 0) \\
r_2^{m - \alpha} R^\alpha & (0 < m < 8) \\
r_2^8 R^{m - 8} & (m > 8)
\end{cases} \\
\tag{$c$} \int_{t_1}^{t_2} \!\!\!\! \int_{r_1}^{r_2} u_\eta^2 \, dV dt & \leq C_{m,\alpha} \begin{cases} \rho^m \left( B_0 + B_1 \log\left( 1 + \dfrac{r_2}{r_1} \right) \right) & (m < 0) \\
B r_2^{m - \alpha} R^\alpha & (0 < m < 8) \\
B r_2^8 R^{m - 8} & (m > 8).
\end{cases}
\end{align}
\end{prop}
\begin{proof}
Let $u_i$ be as in the previous proof, and write
$$E_i = \int_{t_0}^{t_1} \!\!\!\! \int_{s_i}^{s_{i+1}} \left( s^2 \eta(s,t) \right)^2 \, dV_s dt.$$

Let $0 \leq k \leq 2$ and apply H\"older's inequality
\begin{align}\label{inhomogeneousholder}
\nonumber r^{2k} u_i(r,t)^2 & \leq C \left( \int_{t_0}^t \!\!\! \int_{s_i}^{s_{i+1}} \left| \eta_i(s,\tau) \right| e^{-(r-s)^2 / 4(t - \tau)} \frac{r^{2+k} s^3}{\left( rs + t - \tau \right)^{5/2} } \, \frac{ds}{(t - \tau)^{1/2}} \, d\tau \right)^2 \\
\nonumber & \leq \int_{t_0}^t \!\!\! \int_{s_i}^{s_{i+1}} \eta^2(s,\tau) e^{-(r-s)^2 / 4(t - \tau)} \frac{ s^6}{\left( rs + t - \tau \right)^{2-k} } \, \frac{ds}{(t - \tau)^{1/2}} \, d\tau \\
& \quad \quad \quad \quad \cdot \int_{t_0}^t \int_{s_i}^{s_{i+1}} e^{-(r-s)^2 / 4(t - \tau)} \frac{ r^{2k + 4} }{\left( rs + t - \tau \right)^{k + 3} } \, \frac{ds}{(t - \tau)^{1/2}} \, d \tau \\
\nonumber & \leq C_m \frac{t - t_0}{r^2 + t - t_0} \min \LB \left( \frac{r}{s_i} \right)^{2k + 4} , \frac{s_i}{r} \RB \int_{t_0}^t \int_{s_i}^{s_{i+1}} s^6 \eta^2 \frac{e^{-(r-s)^2 / 4(t - \tau)}}{\left( rs + t - \tau \right)^{2-k} } \, \frac{ds}{(t - \tau)^{1/2}} \, d\tau
\end{align}
where we have applied Lemma \ref{holderlemma}$b.$

First assume $r \geq 3s_i$ and apply (\ref{inhomogeneousholder}) with $k=2,$ as well as (\ref{eandtbound}), to obtain
\begin{equation}\label{inhomogeneous1}
 r^4 u_i^2(r,t) \leq C \frac{t - t_0}{r^2 + t - t_0} \frac{s_i}{r^2} \int_{t_0}^t \int_{s_i}^{s_{i+1}} \left( s^2 \eta(s,\tau) \right)^2 s^2 \, ds d\tau \leq \frac{C}{r^2} \frac{t - t_0}{r^2 + t - t_0} E_i.
 \end{equation} 
For $r \leq 3s_i,$ instead write
\begin{align}\label{inhomogeneous2}
 u_i(r,t)^2 \leq \int_{t_0}^t \!\! \int_{s_i}^{s_{i+1}} \eta^2(s,\tau) \, s^6 ds d\tau \cdot \int_{t_0}^t \int_{s_i}^{s_{i+1}} e^{-(r-s)^2 / 2(t - \tau)} \frac{ r^4 }{\left( rs + t - \tau \right)^{5}} \, \frac{ds d \tau}{(t - \tau)}.
\end{align}
Applying Lemma \ref{holderlemma}$a,$ we have
\begin{align*}
\int_{t_0}^t \!\! \int_{s_i}^{s_{i+1}} \frac{e^{-(r-s)^2 / 2(t - \tau)} r^4 }{\left( rs + t - \tau \right)^{5}} \, \frac{ds \, d \tau}{(t - \tau)} &  \leq C \int_0^{t - t_0} \frac{1}{(s_i^2 + \bar{\tau})^3} \frac{s_i}{s_i + \sqrt{\bar{\tau}}} \frac{r^4}{(s_i^2 + \bar{\tau})^2} \frac{d\bar{\tau}}{\sqrt{\bar{\tau}}} \\
& \leq C s_i r^4 \int_0^{t - t_0} \frac{d \sqrt{\bar{\tau}} }{(s_i + \sqrt{\bar{\tau}})^{11}} \leq C \frac{r^4}{s_i^9}.
\end{align*}
Then (\ref{inhomogeneous2}) reads
\begin{equation}\label{inhomogeneous3}
r^6 u_i(r,t)^2 \leq C \left( \frac{r}{s_i} \right)^{10} E_i \quad \quad (r \leq 3 s_i).
\end{equation}
Combining (\ref{inhomogeneous1}) and (\ref{inhomogeneous3}) yields
\begin{equation}\label{inhomogeneous4}
r^6 u_i^2(r,t) \leq C E_i \min \LB 1, \left( \frac{r}{s_i} \right)^{10} \RB.
\end{equation}

Note for $\alpha > 0$ that
\begin{equation}\label{inhomogeneousalphaeq}
\begin{split}
u^2 = \left( \sum_i u_i s_i^{-\alpha/2} s_i^{\alpha/2} \right)^2 & \leq \sum u_i^2 s_i^{-\alpha} \sum s_i^{\alpha} \leq C_\alpha R^\alpha \sum u_i^2 s_i^{-\alpha}.
\end{split}
\end{equation}
To prove ($a$), assuming $m - \alpha > 0,$ from (\ref{inhomogeneous4}) and (\ref{inhomogeneousalphaeq}), we have
\begin{align*}
r^6 u^2(r,t) & \leq C R^\alpha \sum r^6 u_i^2 s_i^{-\alpha}  \\
& \leq C B R^\alpha \left( \sum_{i = 0}^j s_i^{m - \alpha} + r^{10} \sum_{i = j+1}^{N} s_i^{m - \alpha - 10} \right) \leq \begin{cases} r^{m-\alpha} R^\alpha & (m - \alpha < 10) \\
r^{10} R^{m - 10} & (m - \alpha > 10).
\end{cases}
\end{align*}
If $m < 0,$ we replace $s_i$ by $s_i^{-1}$ in (\ref{inhomogeneousalphaeq}) and take $\alpha = -m/2.$ We then have
\begin{align*}
r^6 u^2(r,t) & \leq C \rho^{-\alpha}\sum r^6 u_i^2 s_i^{\alpha} \leq C B \rho^{m/2} \left( \sum_{i = 0}^j s_i^{m/2} + r^{10} \sum_{i = j+1}^{N} s_i^{m/2 - 10} \right) \leq CB \rho^m
\end{align*}
which establishes ($a$).

To prove ($b$), return to (\ref{inhomogeneousholder}) with $k=2,$ integrate, and apply Fubini's Theorem. To prove ($c$), integrate (\ref{inhomogeneousholder}) with $k=1.$
\end{proof}


\end{document}